\documentclass[a4paper,11pt,fleqn]{article}
\usepackage{amsmath,amssymb,amsthm}
\usepackage{mathrsfs} %pour mathscr
\usepackage{euscript}
\usepackage{color}
\usepackage[shortlabels]{enumitem}
\usepackage{bm}
\usepackage{charter}
\usepackage{graphicx}
\usepackage[normalem]{ulem}

\definecolor{labelkey}{rgb}{0,0.08,0.45}
\definecolor{refkey}{rgb}{0,0.6,0.0}
\definecolor{Brown}{rgb}{0.45,0.0,0.05}
\definecolor{dgreen}{rgb}{0.00,0.49,0.00}
\definecolor{dblue}{rgb}{0,0.08,0.75}
\RequirePackage[colorlinks,hyperindex]{hyperref} %click pdf
\hypersetup{linktocpage=true,citecolor=dblue,linkcolor=dgreen}

\newtheorem{theorem}{Theorem}[section]
\newtheorem{corollary}[theorem]{Corollary}
\newtheorem{lemma}[theorem]{Lemma}
\newtheorem{fact}[theorem]{Fact}
\newtheorem{proposition}[theorem]{Proposition}

\theoremstyle{definition}

\newtheorem{algorithm}[theorem]{Algorithm}
\newtheorem{remark}[theorem]{Remark}
\newtheorem{example}[theorem]{Example}

\numberwithin{equation}{section}

\providecommand{\norm}[1]{\lVert#1\rVert}
\providecommand{\scalarp}[1]{\langle#1\rangle}
\providecommand{\abs}[1]{\lvert#1\rvert}
\newcommand{\EE}{\ensuremath{\mathsf E}}
\newcommand{\E}{\ensuremath{\mathsf{J}}}
\newcommand{\Fsc}{\ensuremath{\mathfrak{E}}}
\newcommand{\UU}{\ensuremath{\mathsf{U}}}
\newcommand{\As}{\ensuremath{\mathsf{A}}}
\newcommand{\Id}{\ensuremath{\mathsf{Id}}}
\newcommand{\PP}{\ensuremath{\mathsf P}}
\newcommand{\R}{\ensuremath \mathbb{R}}
\newcommand{\HH}{\ensuremath \mathsf{H}}
\newcommand{\WW}{\ensuremath \mathsf{W}}
\newcommand{\VV}{\ensuremath \mathsf{V}}
\newcommand{\XX}{\ensuremath \mathsf{X}}
\newcommand{\HHs}{\ensuremath \mathsf{H}}
\newcommand{\Gammas}{\ensuremath \mathsf{\Gamma}}
\newcommand{\Lambdas}{\ensuremath \mathsf{\Lambda}}
\newcommand{\YC}{\ensuremath \mathcal{Y}}
\newcommand{\GG}{\ensuremath \mathsf{G}}
\newcommand{\pp}{\ensuremath \mathsf{p}}

\newcommand{\xx}{\ensuremath \mathsf{x}}
\newcommand{\ee}{\ensuremath \mathsf{e}}
\newcommand{\as}{\ensuremath \mathsf{a}}
\newcommand{\bs}{\ensuremath \mathsf{b}}
\newcommand{\ww}{\ensuremath \mathsf{w}}
\newcommand{\zz}{\ensuremath \mathsf{z}}
\newcommand{\vv}{\ensuremath \mathsf{v}}
\newcommand{\uu}{\ensuremath \mathsf{u}}
\newcommand{\yy}{\ensuremath \mathsf{y}}
\newcommand{\N}{\ensuremath \mathbb{N}}
\newcommand{\iter}{\ensuremath n}
\newcommand{\prox}{\ensuremath{\text{\sf prox}}}
\DeclareMathOperator*{\esssup}{\ensuremath{\text{\rm ess\:sup}}}
\DeclareMathOperator*{\dom}{\ensuremath{\text{\rm dom}}}
\DeclareMathOperator*{\diam}{\ensuremath{\text{\rm diam}}}
\DeclareMathOperator*{\argmin}{\text{\rm argmin}}
\newcommand{\minimize}[2]{\ensuremath{\underset{\substack{{#1}}}%
{\text{\rm minimize}}\;\;#2 }}

\newcommand{\Fss}{\ensuremath{{\mathsf{\Phi}}}}
\newcommand{\Fs}{\ensuremath{{\mathsf{F}}}}
\newcommand{\fs}{\ensuremath{{\mathsf{g}}}}
\newcommand{\gs}{\ensuremath{{\mathsf{h}}}}

\newcommand{\bHH}{\ensuremath {\mathsf{H}}}
\newcommand{\bSS}{\ensuremath {\mathsf{S}}}
\newcommand{\bWW}{\ensuremath{{\WW}}}
\newcommand{\bVV}{\ensuremath{{\VV}}}
\newcommand{\bXX}{\ensuremath{{\XX}}}
\newcommand{\bAs}{\ensuremath{{\As}}}
\newcommand{\bxx}{\ensuremath \xx}
\newcommand{\byy}{\ensuremath \yy}
\newcommand{\bzz}{\ensuremath \zz}
\newcommand{\bvv}{\ensuremath \vv}
\newcommand{\buu}{\ensuremath \uu}
\newcommand{\bee}{\ensuremath \ee}
\newcommand{\bbs}{\ensuremath \bs}
\newcommand{\bas}{\ensuremath \as}
\newcommand{\bvarepsilon}{\ensuremath \varepsilon}
\newcommand{\beps}{\ensuremath \epsilon}
\newcommand{\bx}{\ensuremath x}
\newcommand{\by}{\ensuremath y}
\newcommand{\bv}{\ensuremath v}
\newcommand{\bu}{\ensuremath u}
\newcommand{\bfs}{\ensuremath{{\mathsf{f}}}}
\newcommand{\bgs}{\ensuremath{{\gs}}}
\newcommand{\bzero}{\ensuremath{{0}}}
\newcommand{\bvarphi}{\ensuremath{{\varphi}}}
\newcommand{\bDelta}{\ensuremath{{\Delta}}}
\newcommand{\bGammas}{\ensuremath{{\Gammas}}}

\begin{document}

\title{ {\sffamily Parallel Random Block-Coordinate Forward-Backward Algorithm:\\ 
A Unified Convergence Analysis}}
\author{Saverio Salzo\thanks{Istituto Italiano di Tecnologia, Via Melen, 83,  
        16152 Genova, Italy 
        ({\tt saverio.salzo@iit.it}).}\ \  and Silvia Villa\thanks{Universit\`a degli Studi di Genova, Via Dodecaneso, 35, 16146 Genova, Italy ({\tt silvia.villa@unige.it}). {\scriptsize Supported by the H2020-MSCA-RISE project NoMADS-GA No.~777826 and by Gruppo Nazionale per l'Analisi Matematica, la Probabilit\`a e le loro Applicazioni (GNAMPA) of the Istituto Nazionale di Alta Matematica (INdAM).}}}
\date{}

\maketitle

\begin{abstract}
We study the block-coordinate forward-backward 
algorithm in which the blocks are updated in a random and possibly parallel manner,
according to arbitrary probabilities.
The algorithm allows different stepsizes along the block-coordinates
to fully exploit the smoothness properties of the objective function.
In the convex case and in an infinite dimensional setting, we establish almost sure weak convergence of the iterates and 
the asymptotic rate $o(1/n)$ for the mean of the function values. 
We derive linear rates
under strong convexity and error bound conditions. 
Our analysis is based on an abstract convergence principle for stochastic descent algorithms which allows to extend and simplify existing results.  
\end{abstract}

\vspace{1ex}

\section{Introduction and problem setting}

Random block-coordinate descent algorithms are nowadays among the methods of choice for solving large scale optimization problems \cite{Nes12,Ric16,Wri15}. Indeed, they have low complexity and low memory requirements and, additionally, they are amenable for distributed and parallel implementations \cite{Ric16c,Ric16}.
In the last decade a number of works have appeared on the topic which address 
several aspects, that is: the way the block sampling is performed, the composite structure, the partial separability, and the smoothness/geometrical properties of the objective function, accelerations, and iteration complexity  \cite{Com15,Com18,Fer15,LuX15,Nec16,Nec17,Nes12,Qu16,Qu16b,Ric15,Ric16,Tap18}.

In this work we consider the following optimization problem
\begin{equation}
\label{eq:20190110b}
\minimize{\bxx \in \bHH}{\bfs(\bxx)+\bgs(\bxx)},
\qquad\bgs(\bxx)= \sum_{i=1}^m \gs_i(\xx_i),
\end{equation}
where $\bHH$ is the direct sum of $m$ separable real Hilbert spaces 
 $(\HH_i)_{1 \leq i \leq m}$, that is,
\begin{equation*}
\bHH = \bigoplus_{i=1}^m \HH_i,\qquad 
(\forall\, \bxx = (\xx_i)_{1 \leq i \leq m}, \byy = (\yy_i)_{1 \leq i \leq m} \in \bHH)\quad
\scalarp{\bxx,\byy} = \sum_{i=1}^m \scalarp{\xx_i,\yy_i},
\end{equation*}
and the following assumptions hold:
\begin{enumerate}[{\rm H1}]
\item\label{eq:A1} 
$\bfs\colon \bHH \to \R$ is convex and differentiable,
\item\label{eq:A2} 
for every $i=1, \dots, m$, $\gs_i\colon \HH_i \to \left]-\infty, +\infty\right]$ is  proper, convex, and lower semicontinuous.
\end{enumerate}
The objective of this study is a stochastic algorithm, called
\emph{parallel random block-coordinate forward-backward algorithm}, that depends 
on a random variable $\bvarepsilon$ satisfying the following hypothesis
\begin{enumerate}[{\rm H3}] 
\item\label{eq:A3} $\bvarepsilon = (\varepsilon_1,\dots, \varepsilon_m)$ is a random variable with values in $\{0,1\}^m$ 
such that, for every $i \in \{1,\dots,m\}$, $\pp_i:=\PP(\varepsilon_i = 1)>0$
and $\PP\big(\bvarepsilon = (0,\dots, 0)\big) =0$.
\end{enumerate}
\begin{algorithm}
\label{algoRCD}
Let $(\bvarepsilon^\iter)_{\iter \in \N} 
=  (\varepsilon^\iter_1, \dots, \varepsilon^\iter_m)_{\iter \in \N}$ be a sequence of 
independent copies of $\bvarepsilon$. 
Let $(\gamma_i)_{1 \leq i \leq m} \in \R_{++}^m$ and
 $\bx^0 = (\bx^0_1,\dots, \bx^0_m)  \equiv \xx^0 \in \dom \bgs$ be a constant random variable. Iterate
 
\vspace{-2.5ex}
\begin{equation}
\label{eq:algoPRCD2}
\begin{array}{l}
\text{for}\;n=0,1,\ldots\\
\left\lfloor
\begin{array}{l}
\text{for}\;i=1,\dots, m\\[0.7ex]
\left\lfloor
\begin{array}{l}
x^{\iter+1}_i = 
x^\iter_i + \varepsilon^\iter_i \big(\prox_{\gamma_i \gs_{i}} 
\big(x^\iter_i - \gamma_{i} \nabla_i \bfs (\bx^{\iter})\big) - x^\iter_i\big).
\end{array}
\right.
\end{array}
\right.
\end{array}
\end{equation}
For every $\iter \in \N$,
we denote by $\Fsc_{\iter}$ the sigma-algebra generated by $\bvarepsilon^{0},\dots, \bvarepsilon^{\iter}$.
\end{algorithm}

In Algorithm~\ref{algoRCD}, the role of the random variable 
$\bvarepsilon^\iter$ is to select, at iteration $\iter$,
the blocks to update in parallel (those indexed in 
$\{i \in \{1,\dots, m\} \,\,\vert\, \varepsilon_i^\iter = 1\}$). 
When all block-coordinates are simultaneously updated at each iteration, Algorithm~\ref{algoRCD} reduces to
the (deterministic) forward-backward algorithm, which converges
 only if the stepsizes are appropriately set.
More specifically, if $\nabla \bfs$ is  $L$-Lipschitz continuous, 
then convergence is ensured if the stepsizes $\gamma_i$ are all equal and \emph{strictly less} than $2/L$ \cite{Com05,Dav16}. This fact is proved by using 
the so called \emph{descent lemma}, i.e., 
\begin{equation}
\label{eq:20191012b}
(\forall\, \bxx \in \bHH)(\forall\, \bvv\in \bHH)\qquad \bfs(\bxx + \bvv) \leq \bfs(\bxx) + \scalarp{\nabla \bfs(\bxx), \bvv}
+ \frac{L}{2} \norm{\bvv}^2.
\end{equation}
Indeed, \eqref{eq:20191012b} is itself an assumption concerning the smoothness of $\bfs$,
since it is well-known to be equivalent to the Lipschitz continuity of the gradient of $\bfs$ \cite[Theorem~18.15]{book1}.
By contrast, when 
the block-coordinates are updated one by one in a serial manner,
it is desirable to 
allow moving along the block-coordinates with different 
 stepsizes, depending on the 
Lipschitz constants of the partial gradients of $\bfs$ across 
the block-coordinates
\cite{Bec13,Nes12}. So, in this case 
it is more appropriate to assume that a descent lemma holds on each block-coordinate subspace individually, that is,
\begin{equation}
\label{eq:S00}
(\forall\, i=1,\dots,m)(\forall\, \bxx \in \bHH)(\forall\,\vv_i \in \bHH_i)\quad
\bfs(\bxx + \E_i\bvv_i)
\leq \bfs(\bxx) +  \scalarp{\nabla_i \bfs(\bxx), \bvv_i}
+ \frac{L_i}{2} \norm{\vv_i}^2,
\end{equation}
where $\E_i \bvv_i = (0, \dots, 0, \bvv_i,0, \dots, 0)$  ($\bvv_i$ occurring at the $i$-th position),
for some positive constants $L_i$'s.
In the setting of Algorithm~\ref{algoRCD}, multiple block-coordinates may be updated in parallel at each iteration, 
according to the random sampling $\varepsilon$. Therefore, 
it is reasonable to assume that there exists $(\nu_i)_{1 \leq i \leq m} \in \R_{++}^\N$ 
so that one of the \emph{generalized smoothness conditions} below holds
\begin{enumerate}[{\rm S1}] 
\item
\label{eq:S1}
$(\forall\, \bxx,\bvv \in \bHH)\quad
\EE[\bfs(\bxx + \bvarepsilon \odot \bvv)]
\leq \bfs(\bxx) +  \EE[\scalarp{\nabla \bfs(\bxx), \bvarepsilon \odot \bvv}]
+\displaystyle\frac{1}{2} \sum_{i=1}^m \pp_i \nu_i \norm{\vv_i}^2
$,
\item\label{eq:S2} 
$(\forall\, \bxx,\bvv \in \bHH)\quad
\bfs(\bxx + \bvarepsilon \odot \bvv)
\leq \bfs(\bxx) +  \scalarp{\nabla \bfs(\bxx), \bvarepsilon \odot \bvv}
+ \displaystyle\frac{1}{2} \sum_{i=1}^m \nu_i \varepsilon_i \norm{\vv_i}^2
\quad \PP \text{ a.s.}$,
\end{enumerate}
where $\bvarepsilon \odot \bvv = (\varepsilon_i \vv_i)_{1 \leq i \leq m} \in \bHH$.
Conditions \ref{eq:S1} and \ref{eq:S2} can be interpreted 
as descent lemmas on random block-coordinate subspaces,
 depending on the chosen random sampling of the block-coordinates.
They reduce to \eqref{eq:S00}, with $\nu_i=L_i$, if the sampling $\varepsilon$
selects only one block at a time almost surely (see Section~\ref{subsec:beta}).
We call $(\nu_i)_{1 \leq i \leq m}$ the 
 \emph{smoothness parameters} of $\bfs$.
 Then, similarly to the deterministic case,
we will adopt the following stepsize rule
\begin{equation}
\label{eq:stepsizerule}
(\forall\, i \in \{1,\dots, m\})\qquad \gamma_i < \frac{2}{\nu_i}.
\end{equation}
Another smoothness condition suitable for Algorithm~\ref{algoRCD}, which was considered in \cite{Nec16}, is
\begin{enumerate}[{\rm S3}] 
\item
\label{eq:S0}
$(\forall\, \bxx,\bvv \in \bHH)\quad
\bfs(\bxx + \bvv)
\leq \bfs(\bxx) +  \scalarp{\nabla \bfs(\bxx), \bvv}
+\displaystyle\frac{1}{2} \sum_{i=1}^m \nu_i \norm{\vv_i}^2$.
\end{enumerate}
Note that, \ref{eq:S0} $\ \Rightarrow\ $ \ref{eq:S2} $\ \Rightarrow\ $ \ref{eq:S1},
which in turn implies the Lipschitz continuity of 
the gradient of $\bfs$ 
(see Theorem~\ref{thm:stepsizes}\ref{thm:stepsizes_iv}). So, possibly with different values of the $\nu_i$'s,
the above conditions are all equivalent. The point is that in the parallel setting 
(where multiple blocks are updated in parallel at each iteration),
 \ref{eq:S1} may be fulfilled with values of $\nu_i$ that are much smaller than 
those related to the other two conditions, ultimately allowing to significantly increase the stepsizes
and hence speeding up the convergence. Moreover, 
\ref{eq:S1} makes parallelization particularly effective
on problems with a sparse structure and superior to the serial strategy (which updates a single block per iteration).
See the discussion after Theorem~\ref{thm:20171207a}.
The critical role played by assumption \ref{eq:S1} in the analysis of 
parallel randomized block-coordinate descent methods was pointed out
in \cite{Qu16b,Ric16,Ric16b,Tap18}. There, it was called \emph{expected separable overapproximation} (ESO) inequality. 
 Condition
\ref{eq:S2} is new and serves to guarantee that Algorithm~\ref{algoRCD}
is almost surely descending (Proposition~\ref{p:20181219c}), which is a property 
that is especially relevant when error bound conditions hold (see Section~\ref{sec:errorbounds}).
Note that in \cite{Ric16} the issue of monotonicity of the algorithm was addressed for each sampling separately without any general guidance.
Finally, we stress that, except for \cite{Com15,Com18} 
(which study the convergence of the iterates only), in all previous works 
the stepsizes $\gamma_i$'s are set equal to $1/\nu_i$. 
This is an unnecessary limitation that we remove,
so to match the standard stepsize rule of the forward-backward algorithm \cite{Com05,Dav16}.
\begin{remark}
\label{rmk:20190126c}
For every $i=1,\dots, m$, the canonical embedding of $\HH_i$ into $\bHH$ is
the operator $\E_i\colon \HH_i \to \bHH$, $\xx \mapsto (0, \dots, 0, \xx,0, \dots, 0)$,
where $\xx$ occurs in the $i$-th position.
Then Algorithm~\ref{algoRCD} can be written as
\begin{equation*}
\bx^{\iter+1} = \bx^\iter + \sum_{i=1}^m \varepsilon_i^\iter \E_i
\big(\prox_{\gamma_i \gs_{i}} \big(x^\iter_i - \gamma_{i} \nabla_i \bfs (\bx^{\iter})\big) - x^\iter_i\big).
\end{equation*}
\end{remark}

\subsection{Main contributions and comparison to previous work}

In the following we summarize the main contributions of this paper, where,
for the sake of brevity,  we set $\Fs=\bfs+\bgs$.  
We assume that \ref{eq:A1}--\ref{eq:A3} are satisfied and that \ref{eq:S1} is met.
Then, the following hold.
\begin{itemize}
\item 
Algorithm~\ref{algoRCD} is descending in expectation and 
$\EE[\Fs(\bx^{\iter})] - \inf \Fs  \to 0$, even if the infimum is not attained. 
If $\argmin \Fs \neq \varnothing$, then 
$\EE[\Fs(\bx^{\iter})] -  \inf \Fs= o(1/\iter)$. In addition, 
a nonasymptotic bound for $\EE[\Fs(\bx^{\iter})] -  \inf \Fs$ of order $O(1/\iter)$
holds. 
Finally, 
there exists a random variable $\bx_*$ with values in $\argmin \Fs$ such that 
$\bx^\iter \rightharpoonup \bx_*$ $\PP$-a.s. See Theorem~\ref{thm:20171207a}.  
\item 
If $\Fs$ is strongly convex or satisfies an error bound condition of Luo-Tseng type 
(see condition~\ref{eq:EB1}), 
then 
 the  iterates 
 as well as the corresponding  
function values generated by Algorithm~\ref{algoRCD}, converge linearly in expectation.
See Theorem~\ref{p:20181130a}, Theorem~\ref{thm:20181214a}, and 
Theorem~\ref{thm20200218a}.
\end{itemize}

Our results advance the state-of-the-art in the study of random block-coordinate 
descent methods under several aspects. We comment on this below.
1) While convergence of the function values has been intensively studied in the related literature 
(see e.g., \cite{Kar18,LuX15,Nec16,Nes12,Qu15,Ric16,Ric16b,Tap18}), surprisingly, in a convex setting, 
convergence of the iterates has been investigated only recently in \cite{Com15}, but with stepsizes set 
according to the global Lipschitz constant of $\nabla \bfs$. 
See also \cite{Fer19} which addresses the convergence of the iterates in the framework 
of primal-dual algorithms with a serial and uniform block sampling.
We improve the existing results, since we show \emph{convergence of the iterates} 
for Algorithm~\ref{algoRCD} in an \emph{infinite dimensional setting} 
even when the stepsizes are chosen according to the condition \ref{eq:S1},
which can incorporate the block Lipschitz constants of the gradient of $\bfs$ and is at the basis of the effectiveness of the 
parallel block-coordinatewise approach. 
2) The worst case \emph{asymptotic rate $o(1/\iter)$} for the mean of the function values is new in the setting of stochastic algorithms. 
3) Our analysis spotlights an abstract convergence principle for stochastic descent algorithms (Theorem~\ref{p:SMFejer}) which is essentially a special form
 of the stochastic quasi-Fej\'er monotonicity property, involving also the values of the objective functions.
This principle, previously investigated in 
a deterministic setting in \cite{Sal17}, allows to prove in 
a unified way both the almost sure convergence of the iterates and rates of convergence for the mean of the function values.
4) As a by-product of the above analysis we single out an inequality (Proposition~\ref{p:20190610a}) 
which is pivotal for studying 
\emph{the convergence under error
bound conditions}, improving the results and simplifying the analysis in \cite{Nec16}.
5) We allow for  {\em parallel and arbitrary sampling 
of the blocks} in a composite setting. 
The benefit of
such sampling in terms of convergence rate have been first investigated in \cite{Ric16b} 
for a strongly convex and smooth objective function.
In \cite{Qu16} a composite objective optimization problem was analyzed but for a slightly different algorithm. The rest of the studies deal either with parallel uniform sampling of the blocks \cite{Ric16}, or with the case where a single block is updated at each iteration \cite{LuX15,Nes12}. 
6) We also allow for {\em stepsizes larger than those considered in literature}
\cite{Kar18,LuX15,Nec16,Nes12,Qu15,Ric16,Ric16b,Tap18}, since we can let the stepsizes 
go beyond $1/\nu_i$ and be arbitrarily close to
$2/\nu_i$, matching the standard rule for the forward-backward algorithm. This provides additional flexibility to the algorithm. Indeed, in the strongly convex case we show that the optimal stepsizes are strictly 
larger than $1/\nu_i$.

The rest of the paper is organized as follows. 
In Section~\ref{sec:notation} we give notation and basic facts.
Section~\ref{sec:stepsizes} shows how to determine the smoothness parameters 
$\nu_i$ when $\bfs$ features a partially separable structure. 
In Section~\ref{sec:conana} we carry out the convergence analysis and give the related theorems.
Finally, Section~\ref{sec:app} shows three applications 
and Section~\ref{sec:experiments} provides some numerical experiments. 

\section{Notation and background}
\label{sec:notation}

\paragraph{Notation.}
We define $\R_+ = \left[0,+\infty\right[$, $\R_{++} = \left]0,+\infty\right[$, 
for every integer $s \geq 1$, $[s] = \{1, \dots, s\}$, and for every $a \in \R^s$, 
 $\mathrm{spt}(a) = \{i \in [s] \,\vert\, a_i\neq 0\}$.
 Scalar products and norms in Hilbert spaces are denoted by $\scalarp{\cdot,\cdot}$
 and $\norm{\cdot}$ respectively.
If $\UU\colon \HH \to \GG$ is a bounded linear operator between 
real Hilbert spaces, 
$\UU^\top \colon \GG \to \HH$ is its transpose operator, that is, the one satisfying $\scalarp{\UU \xx, \yy} = \scalarp{\xx, \UU^\top \yy}$, 
for every $(\xx,\yy) \in \HH\times \GG$.
Let $(\HH_i)_{1 \leq i \leq m}$ be $m$ separable real Hilbert spaces and let 
$\bHH=\bigoplus_{i=1}^m \HH_i$ be their direct sum.
For every $\bvv \in \bHH $ and $\beps \in \{0,1\}^m$
we set $\beps \odot\bvv = (\epsilon_i \vv_i)_{1 \leq i \leq m} \in \bHH$.
We will consider random variables with underlying probability space $(\Omega, \mathfrak{A}, \PP)$ taking values in $\HH_i$ or $\bHH$.  We use
the default font for random variables and sans serif font
for their realizations. The expected value operator is denoted by $\EE$.
 A copy of a random variable is  random variable having the same distribution of the given one.
Let $(w_i)_{1 \leq i \leq m} \in \R^m_{++}$. 
The direct sum operator 
$\bWW = \bigoplus_{i=1}^m w_i \Id_i$, 
where $\Id_i$ 
is the identity operator on $\HH_i$, is the positive bounded linear operator on $\bHH$ acting as
$\bxx=(\xx_i)_{1 \leq i \leq m} \mapsto (w_i \xx_i)_{1 \leq i \leq m}$. 
$\bWW$ defines an equivalent inner product on $\bHH$ 
\begin{equation*}
(\forall\, \bxx \in \bHH)(\forall\, \byy \in \bHH)\qquad\scalarp{\bxx, \byy}_{\bWW} = \scalarp{\bWW \bxx, \byy} 
= \sum_{i=1}^m w_i \scalarp{\xx_i, \yy_i},
\end{equation*}
which gives the norm $\norm{\bxx}_{\bWW}^2 
= \sum_{i=1}^m w_i \norm{\xx_i}^2$.
If $\bSS \subset \bHH$ and $\bxx \in \bHH$, we set
 $\mathrm{dist}_{\bWW}(\bxx,\bSS) 
= \inf_{\bm{\zz} \in \bSS} \norm{\bxx - \bm{\zz}}_{\bWW}$.
Let $\bvarphi\colon \bHH \to \left]-\infty,+\infty\right]$ be proper, convex, and lower semicontinuous. The domain of $\bvarphi$ is $\dom \bvarphi 
= \{\bxx \in \HH \,\vert\, \bvarphi(\bxx)<+\infty\}$ and the set 
of minimizers of $\bvarphi$ is $\argmin \bvarphi 
= \{\bxx \in \bHH \,\vert\, \bvarphi(\bxx) = \inf \bvarphi\}$. 
The subdifferential of $\bvarphi$ in the metric $\scalarp{\cdot,\cdot}_{\bWW}$ 
is the multivalued operator 
\begin{equation*}
\partial^{\bWW} \bvarphi \colon \bHH \to 2^{\bHH},\ \ 
\bxx \mapsto \partial^{\bWW} \bvarphi(\bxx)= \big\{\buu \in \bHH \,\vert\, 
(\forall\, \byy \in \bHH)\ \bvarphi(\byy) \geq \bvarphi(\bxx) 
+ \scalarp{\buu, \byy - \bxx}_{\bWW}\big\}.
\end{equation*}
In case $\bWW = \Id$, it is simply denoted by $\partial \bvarphi$.
Clearly $\partial^{\bWW} \bvarphi = \bWW^{-1} \partial \bvarphi$.
If the function $\bvarphi\colon \bHH\to \R$ is differentiable, 
then, for every $\bxx \in \bHH$, $\partial^{\bWW} \bvarphi(\bxx) = \{\nabla^{\bWW} \bvarphi(\bxx)\}$ and for all $\bvv \in \bHH$, $\scalarp{\nabla^{\bWW} \bvarphi(\bxx), \bvv}_{\bWW} = \scalarp{\nabla \varphi(\bxx), \bvv}$.
The proximity operator of $\bvarphi$ in the metric $\scalarp{\cdot,\cdot}_{\bWW}$  is defined as 
\begin{equation*}
\prox_{\bvarphi}^{\bWW} \colon \bHH\to \bHH,\quad
\prox_{\bvarphi}^{\bWW}(\bxx) = \argmin_{\bzz \in \bHH} \bvarphi(\bzz) + \frac 1 2 \norm{\bxx - \bzz}^2_{\bWW}. 
\end{equation*}
Referring to the functions in \eqref{eq:20190110b},
we denote by $\mu_{\Gammas^{-1}}$ and $\sigma_{\Gammas^{-1}}$ the moduli of strong convexity 
of $\bfs$ and $\bgs$ respectively,
 in the norm $\norm{\cdot}_{\bGammas^{-1}}$, where 
 $\bGammas = \bigoplus_{i=1}^m \gamma_i \Id_i $ 
 and the $\gamma_i$'s
 are the stepsizes occurring in Algorithm~\ref{algoRCD}.
 This means that $\mu_{\Gammas^{-1}}, \sigma_{\Gammas^{-1}} \in \R_+$ and that,
 for every $\bxx,\byy \in \bHH$,
\begin{align}
\label{eq:20181112b}\bfs(\byy) &\geq \bfs(\bxx) + \scalarp{\nabla \bfs(\bxx), \byy - \bxx} 
+ \frac{\mu_{\Gammas^{-1}}}{2} \sum_{i=1}^m \frac{1}{\gamma_i}\norm{\yy_i - \xx_i}^2,\\
\label{eq:20181112b1}(\forall\, \bvv \in \partial \bgs(\bxx))\qquad
\bgs(\byy) &\geq \bgs(\bxx) + \scalarp{\bvv, \byy - \bxx} 
+ \frac{\sigma_{\Gammas^{-1}}}{2} \sum_{i=1}^m \frac{1}{\gamma_i}\norm{\yy_i - \xx_i}^2.
\end{align}
Note that, since $\bgs$ is separable, by taking $\byy= \bxx + \E_i(\yy_i - \xx_i)$ in \eqref{eq:20181112b1}, we have
\begin{equation}
\label{eq:20181112a}
(\forall\, i \in [m])\quad
\gs_i(\yy_i) \geq \gs_i(\xx_i) + \scalarp{\vv_i, \yy_i - \xx_i} + \frac{\sigma_{\Gammas^{-1}}}{2} \frac{1}{\gamma_i} \norm{\yy_i - \xx_i}^2.
\end{equation}

\begin{remark}
\label{rmk:20190126e}
If \ref{eq:S1} is satisfied, the $\gamma_i$'s are chosen
as in \eqref{eq:stepsizerule}, and $\delta = \max_{1 \leq i \leq m} \gamma_i \nu_i$ (according to the convergence theorems), then we have
\begin{equation}
\label{eq:20190126e}
\mu_{\Gammas^{-1}} \leq \min_{1 \leq i \leq m}\gamma_i \nu_i \leq \delta <2.
\end{equation}
Indeed,
let $\bxx\in \bHH$,  $i \in [m]$ and $\vv_i \in \HH_i$, $\vv_i \neq 0$.
It follows from \eqref{eq:20181112b} with 
$\byy = \bxx + \bvarepsilon\odot \E_i \vv_i$ (where $\E_i$ is defined in 
Remark~\ref{rmk:20190126c}) and  \ref{eq:S1} that
\begin{align*}
\bfs(\bxx) + \EE[\scalarp{\nabla \bfs(\bxx), \bvarepsilon \odot \E_i \vv_i}] + \frac{\mu_{\Gammas^{-1}}}{2} \frac{\pp_i}{\gamma_i}\norm{\vv_i}^2
&\leq \EE[\bfs(\bxx + \bvarepsilon\odot \E_i \vv_i)] \\
&\leq \bfs(\bxx) +  \EE[\scalarp{\nabla \bfs(\bxx), \bvarepsilon \odot \E_i \vv_i}]
+\displaystyle\frac{1}{2}  \pp_i \nu_i \norm{\vv_i}^2.
\end{align*}
Thus, \eqref{eq:20190126e} follows.
\end{remark}

\begin{fact}[{\cite[Example~5.1.5]{book2}}]
\label{f:20190112d}
Let $\zeta_1$ and $\zeta_2$ be independent random variables with values in 
the measurable spaces $\mathcal{Z}_1$ and $\mathcal{Z}_2$ respectively. 
Let $\varphi\colon \mathcal{Z}_1\times \mathcal{Z}_2 \to \R$ be measurable
and suppose that $\EE[\abs{\varphi(\zeta_1,\zeta_2)}]<+\infty$.
Then $\EE[\varphi(\zeta_1,\zeta_2) \,\vert\, \zeta_1] = \psi(\zeta_1)$,
where for all $z_1 \in \mathcal{Z}_1$, $\psi(z_1) = \EE[\varphi(z_1, \zeta_2)]$.
\end{fact}

\begin{fact}
\label{lem:20181024a}
Let $\bvarepsilon$ be a random variable with values in  $\{0,1\}^m$ and, for all $i \in [m]$, $\pp_i = \PP(\varepsilon_i = 1)$.
Then $\EE[\varepsilon_i] = \pp_i$ and, for every
 $\bvv = (\vv_i)_{1 \leq i\leq m} \in \R^m$,  $\EE[\scalarp{\bvarepsilon, \bvv} ]
 = \sum_{i=1}^m \pp_i \vv_i$.
\end{fact}

\begin{fact}[\cite{book3}]
\label{lem:monosequence}
Let $(a_\iter)_{\iter \in \N}$ be a decreasing sequence in $\R_+$. 
If $\sum_{\iter=0}^{+\infty} a_\iter <+\infty$, then, 
for every $\iter \in \N$, $a_\iter \leq (1/(\iter+1)) \sum_{\iter=0}^{+\infty} a_\iter$
and $a_\iter = o\big( 1/(\iter+1)\big)$.
\end{fact}

\section{Determining the smoothness parameters}
\label{sec:stepsizes}
In this section we provide few scenarios for which the relaxed 
smoothness conditions \ref{eq:S1} and \ref{eq:S2} can be fully exploited, 
attaining tight values for the $\nu_i$'s.
This ultimately allows to take larger stepsizes and improves rates of convergence.
In \cite{Qu16b,Tap18} an extensive analysis of cases in which \ref{eq:S1}
is satisfied is presented.

\subsection{General estimates.}
We consider the following setting.
\begin{enumerate}[{\rm H4}]
\item\label{eq:B0} 
The function $\bfs\colon \bHH \to \R$ is such that
\begin{equation}
\label{eq:20190605b}
(\forall\, \bxx \in \bHH)\qquad
\bfs(\bxx) = \sum_{k=1}^p \fs_k \bigg(\sum_{i=1}^m \UU_{k,i} \xx_i\bigg),
\end{equation}
where, for every $k =1,\dots, p$, 
 $\fs_k \colon \GG_k \to \R$ is a convex differentiable function
 defined on a real Hilbert space $\GG_k$ 
 and, for every $i \in [m]$, $\UU_{k,i}\colon \HH_i \to \GG_k$ 
 is a bounded linear operator. Moreover, 
 $\bigcup_{k = 1}^p I_k \neq \varnothing$, where, for all 
 $k =1,\dots, p$, $I_k = \big\{i \in [m] \,\vert\, \UU_{k,i} \neq 0\big\}$,
and $\eta = \max_{1 \leq k \leq p} \mathrm{card}(I_k)$.
\end{enumerate}
We will also consider one of the following conditions.
 \begin{enumerate}[{\rm L1}]
\item\label{eq:B1} For every $i=1, \dots, m$ there exists $L_i>0$ such that, for every 
$\bxx \in \bHH$,  the function
$\nabla_i \bfs (\xx_1, \dots, \xx_{i-1}, \cdot, \xx_{i+1}, \dots, \xx_m) \colon \HH_i \to \HH_i$
 is $L_i$-Lipschitz continuous.
 \item\label{eq:B2}
For every $k=1,\dots, p$,
 $\nabla \fs_k \colon \GG_k \to \GG_k$ is $L^{(k)}$-Lipschitz continuous and for 
 every $i,j \in [m]$, $i\neq j$, the ranges of   $\UU_{k,i}$ and $\UU_{k,j}$
 are orthogonal.
\end{enumerate}

Assumption \ref{eq:B0} concerns the \emph{partial separability} of the function $\bfs$.
Depending on the number of the nonzero operators $\UU_{k,i}$,
$\fs_k$ might depend only on few block-variables $\xx_i$'s:
if $\eta = 1$, $\bfs$ is fully separable, whereas if $\eta = m$, $\bfs$ is not separable.
Note that \ref{eq:B1} is equivalent to \eqref{eq:S00} and, 
since $\bfs$ is convex, implies the global Lipschitz continuity of the 
gradient of $\bfs$ (Corollary~\ref{p:20181014a}). So either \ref{eq:B1} or \ref{eq:B2}
implies the global Lipschitz smoothness of $\bfs$.
However, considering the constants $L_i$'s or $L^{(k)}$'s leads in general to 
a finer analysis of the smoothness properties of $\bfs$, 
eventually determining parameters $\nu_i$ that are smaller than the global Lipschitz constant of $\nabla \bfs$.
Instances of problem~\eqref{eq:20190110b} where $\bfs$
has the structure shown in \ref{eq:B0},
occur very often in applications. In particular,
a prominent example is that of the \emph{Lasso problem} which
will be discussed in Section~\ref{subsec:lasso}.
The following theorem, which is proved in Appendix~\ref{sec:appB}, 
relates the smoothness parameters $(\nu_i)_{1 \leq i \leq m}$  to the block Lipschitz constants 
of the partial gradients of $\bfs$ and to the Lipschitz constants of the gradients 
of its components $\fs_k$'s in
\eqref{eq:20190605b}, as well as to the distribution of the random variable $\bvarepsilon$.

\begin{theorem}
\label{thm:stepsizes}
Assume  \ref{eq:A3} and \ref{eq:B0} and let
 $(\nu_i)_{1 \leq i \leq m} \in \R_{++}^\N$. 
 Then the following hold.
 \begin{enumerate}[{\rm (i)}]
 \item\label{thm:stepsizes_i} \ref{eq:B1}$\ \Rightarrow\  $ \ref{eq:S1} provided that
\begin{equation*}
(\forall\, i \in [m])\quad
\nu_i \geq \beta_{1,i} L_i,\quad\text{where}\quad
\beta_{1,i} := \EE \Big[
\max_{1 \leq k \leq p} \Big(\sum_{j\in I_k} \varepsilon_j \Big)\,\Big\vert\, \varepsilon_i=1
\Big].
\end{equation*}
\vspace{-4ex}
\item \label{thm:stepsizes_ii}
\ref{eq:B1}$\ \Rightarrow\  $ \ref{eq:S2} provided that
\begin{equation*}
(\forall\, i \in [m])\quad \nu_i \geq \beta_2 L_i,\quad\text{where}\quad
\beta_2 :=\esssup \Big(\max_{1 \leq k \leq p}\Big(\sum_{j \in I_k} \varepsilon_j \Big) \Big).
\end{equation*}
\vspace{-4ex}
\item\label{thm:stepsizes_iii}
\ref{eq:B2}$\ \Rightarrow\  $ \ref{eq:S0}
provided that
\begin{equation}
(\forall\, i \in [m])\quad
\nu_i \geq \tilde{L}_i := \Big\lVert \sum_{k=1}^p L^{(k)} \UU_{k,i}^\top \UU_{k,i}\Big\rVert.
\end{equation}
\vspace{-3ex}
\item\label{thm:stepsizes_iv} 
\ref{eq:S1}$\ \Rightarrow\  $\ref{eq:B1} with, for all $i \in [m]$, $L_i = \nu_i$. 
In particular, \ref{eq:S1} implies that $\bfs$ is Lipschitz smooth.
 \end{enumerate}
\end{theorem}

\begin{remark}\ 
\label{rmk:20190128e}
\begin{enumerate}[{\rm (i)}]
\item\label{rmk:20190128e_i} Suppose that in \ref{eq:B0}, 
for all $k \in [p]$,
$\GG_k = \bHH$, $\fs_k$ is
$L^{(k)}$-Lipschitz smooth, and, for all $i \in I_k$,
$\UU_{k,i} = \E_i$, the canonical embedding of $\HH_i$ into $\bHH$ (see Remark~\ref{rmk:20190126c}).
Then, \ref{eq:B2} holds and, for every $i \in [m]$, 
$\tilde{L}_i = \sum_{k\,\vert i \in I_k} L^{(k)}$.
Hence, in view of Theorem~\ref{thm:stepsizes}\ref{thm:stepsizes_iii}, 
\ref{eq:S0} 
is met with $\nu_i = \tilde{L}_i$.
This setting was studied in \cite{Nec16}.
\item\label{rmk:20190128e_ii} If $\nabla \bfs$ is $L$-Lipschitz continuous,
then \ref{eq:S0}
is satisfied with, 
for every $i \in [m]$, $\nu_i = L$.
Therefore, we cover the analysis of the random 
block-coordinate forward-backward algorithm given in
 \cite{Com15,Com18} which set the stepsizes as $\gamma_i < 2/L$. 
\item\label{rmk:20190128e_iii} Let, for every $k \in [p]$, $\bfs_k(\xx) =  \fs_k \big(\sum_{i=1}^m \UU_{k,i} \xx_i\big)$.
If, for every $k \in [p]$, \ref{eq:S1} (resp. \ref{eq:S2}) holds for $\bfs_k$ with $(\nu_i^{(k)})_{1 \leq i \leq m}$ ,
then \ref{eq:S1} (resp. \ref{eq:S2}) holds for $\bfs$ with $\nu_i = \sum_{k=1}^p \nu_i^{(k)}$.
\item\label{rmk:20190128e_iv} Using similar ideas as in the proof of \cite[Theorem~12]{Ric16} we show in Appendix~\ref{sec:appB} that
item \ref{thm:stepsizes_i} in Theorem~\ref{thm:stepsizes} remains true with
\begin{equation}
\label{eq:20200316a}
\beta_{1,i} := \sum_{t=1}^\eta t \max_{\substack{1 \leq k \leq p \\[0.3ex] i \in I_k}} 
\PP\Big(\sum_{j \in I_k} \varepsilon_j = t \,\big\vert\, \varepsilon_i=1 \Big).
\end{equation}
\end{enumerate}
\end{remark}
\begin{remark}
\label{rmk:20181219ee}
Referring to Theorem~\ref{thm:stepsizes},
for all $i \in [m]$, we have $1 \leq \beta_{1,i} \leq \beta_2
\leq \min\{\eta,\tau_{\max}\}$,
where
\begin{align*}
\tau_{\max} :=\esssup \Big(\sum_{i =1}^m \varepsilon_i \Big)  
&=\min\bigg\{ \tau^\prime \in \N \,\Big\vert\, \PP\bigg( \sum_{i =1}^m 
\varepsilon_i \leq \tau^\prime \bigg)=1 \bigg\}
\end{align*}
is the maximum number of blocks processed in parallel.
Indeed, since
$\PP(\bvarepsilon \equiv 0) = 0$
we have $\PP(\max_{1 \leq i \leq  m} \varepsilon_i \geq 1) = 1$. Moreover, 
since $\max_{1 \leq k \leq p} \big(\sum_{i \in I_k} \varepsilon_i \big) 
\geq \max_{1 \leq i \leq m} \varepsilon_i$, we have $1 \leq \beta_{1,i}$. 
The inequality $\beta_{1,i} \leq \beta_2$ is immediate, while the last one derives from the following
\begin{equation*}
(\forall\, k \in [p])
\quad\sum_{i \in I_k} \varepsilon_i \leq \min\Big\{\mathrm{card}(I_k), \sum_{i=1}^m \varepsilon_i \Big\}
\leq \min\Big\{\eta, \sum_{i=1}^m \varepsilon_i \Big\} \leq \min\{\eta, \tau_{\max}\}.
\end{equation*}
\end{remark}

\subsection{The smoothness parameters for some special block samplings.}
\label{subsec:beta}
Here
we show how to compute (or estimate) 
the constants $(\beta_{1,i})_{1 \leq i \leq m}$ and $\beta_2$ in 
Theorem~\ref{thm:stepsizes} and Remark~\ref{rmk:20190128e}\ref{rmk:20190128e_iv},
and the related $(\nu_i)_{1 \leq i \leq m}$, in some relevant scenarios, when \ref{eq:B0} and \ref{eq:B1}  are satisfied. 
\setlist[description]{font=\normalfont\itshape}
\begin{description}[leftmargin=0cm]
%---------------------------------------------
\item[Arbitrary parallel sampling.]
It follows from Theorem~\ref{thm:stepsizes}\ref{thm:stepsizes_ii} and Remark~\ref{rmk:20181219ee} 
that for an arbitrary (possibly nonuniform) block sampling $\bvarepsilon$, \ref{eq:S2} is satisfied provided that
$\nu_i = \min\{\eta,\tau_{\max}\} L_i$, for every $i \in [m]$. 
Additionally, if we denote by $L^{(k)}_i$ the blockwise Lipschitz constants of the gradient of the function 
 $\xx \mapsto \fs_k \big(\sum_{i=1}^m \UU_{k,i} \xx_i\big)$, then we derive from 
 Remark~\ref{rmk:20190128e}\ref{rmk:20190128e_iii}
 and the above discussion that \ref{eq:S2} holds 
 with\footnote{If $i \notin I_k$, then $L^{(k)}_i = 0$.} 
 $\nu_i = \sum_{k \vert i \in I_k} \min\{\mathrm{card}(I_k), \tau_{\max}\} L^{(k)}_i$.
However, the above estimates are rather conservative 
and can be improved for special choices of the block sampling as we will show below.
We refer to \cite{Qu16b,Ric16b} for further results on nonuniform samplings.
%---------------------------------------------
\item[Serial sampling or full separability.]
Suppose that $\tau_{\max}=1$  or $\eta=1$. 
Then, Remark~\ref{rmk:20181219ee} yields  
$\beta_{1,i}=\beta_2= \min\{\eta,\tau_{\max}\}=1$. Moreover, recalling
Theorem~\ref{thm:stepsizes}\ref{thm:stepsizes_ii}-\ref{thm:stepsizes_iv}, this also shows that
 \ref{eq:S1}, \ref{eq:S2},  and \ref{eq:B1} are indeed equivalent with 
the same smoothness parameters $\nu_i = L_i$.
So, conditions \ref{eq:S1} or \ref{eq:S2} find their justification
 only in the parallel case ($\tau_{\max}>1$) and when $\bfs$ is not 
fully separable ($\eta>1$).
%---------------------------------------------
\item[Fully Parallel.] If $\PP(\sum_{i=1}^m \varepsilon_i = m) = 1$, then 
 for every $i \in [m]$
$\pp_i =1$. This yields a fully parallel (deterministic) algorithm. 
Moreover, since $\PP\big(\bvarepsilon = (1,\dots, 1)\big)=1$, we have
$\beta_{1,i} = \eta =\beta_2$ and hence \ref{eq:S2} holds with $\nu_i = \eta L_i$. Actually,
also \ref{eq:S0} holds with $\nu_i = \eta L_i$ (see Corollary~\ref{p:20181014a}\ref{p:20181014a_3}).
%---------------------------------------------
\item[Uniform samplings.]
Suppose that $m>1$.
The sampling is \emph{uniform} if $\pp_i = \pp_j$, with $i \neq j$. 
In this case if we denote by $\bar{\tau}$ the average number of block updates per iteration, we have
$\bar{\tau}= \EE [\sum_{i=1}^m \varepsilon_i] = \sum_{i=1}^m \pp_i$ and hence $\pp_i = \bar{\tau}/m$,
for every $i \in [m]$.
In \cite{Ric16} several types of uniform samplings are studied. 
In the following we single out two of them. 
The sampling is said to be \emph{doubly uniform} if any
 two sets of blocks with the same number of blocks have the same probability to be chosen.
 In formula, this means that for every $J_1, J_2 \subset [m]$ such that $\mathrm{card}(J_1) = \mathrm{card}(J_2)$,  $\PP(\cap_{i \in J_1} \{\varepsilon_i = 1\}) = \PP(\cap_{i \in J_2} \{\varepsilon_i = 1\})$.
 For such sampling one directly derives from \eqref{eq:20200316a} 
 in Remark~\ref{rmk:20190128e}\ref{rmk:20190128e_iv} (see Appendix~\ref{sec:appB}) that
 \begin{equation}
 \label{eq:20200316b}
 \beta_{1,i} =  \beta_{1} :=  1 + \frac{\eta - 1}{m-1} \bigg( 
 \frac{\EE \big[\big(\sum_{i=1}^m \varepsilon_i \big)^2\big]}{\EE [\sum_{i=1}^m \varepsilon_i]} - 1 \bigg).
 \end{equation}
A special type of doubly uniform sampling is the \emph{$\tau$-nice} sampling in which 
$\sum_{i=1}^m \varepsilon_i = \tau$ $\PP$-a.s. for some $\tau \in [m]$. In this case \eqref{eq:20200316b}
reduces to
 \begin{equation}
 \label{eq:20200316c}
 \beta_{1,i} =  \beta_{1} := 1 + \frac{(\eta - 1)(\tau-1)}{m-1}.
 \end{equation}
 Now, according to Remark~\ref{rmk:20190128e}\ref{rmk:20190128e_iv}, 
 if we set, for every $i \in [m]$, $\nu_i = \beta_{1} L_i$, then condition \ref{eq:S1} holds.
Additionally, 
 if we denote by $L^{(k)}_i$ the blockwise Lipschitz constants of the gradient of the function 
 $\xx \mapsto \fs_k \big(\sum_{i=1}^m \UU_{k,i} \xx_i\big)$, then we derive from 
 Remark~\ref{rmk:20190128e}\ref{rmk:20190128e_iii}
 and \eqref{eq:20200316c} that \ref{eq:S1} is satisfied with $\nu_i = \sum_{k \vert i \in I_k} (1 + (\tau-1)(\mathrm{card}(I_k)-1)/(m-1)) L^{(k)}_i$. This result provides possibly even smaller values
 for the parameters $(\nu_i)_{1 \leq i \leq m}$ and was given, 
in the special setting of Remark~\ref{rmk:20190128e}\ref{rmk:20190128e_i},
 in \cite{Fer15,Tap18}.
\end{description}

\section{Convergence analysis}
\label{sec:conana}

In the rest of the paper, referring to Algorithm~\ref{algoRCD}, we set
\begin{equation}
\label{eq:W}
\bGammas^{-1} = \bigoplus_{i=1}^m \frac{1}{\gamma_i} \Id_i,
\quad(w_i)_{1 \leq i \leq m} = \Big( \frac{1}{\gamma_i \pp_i}\Big)_{1 \leq i \leq m},
\quad\bWW = \bigoplus_{i=1}^m w_i \Id_i,
\end{equation}
where $\Id_i$ is the identity operator on $\HH_i$, and
\begin{equation}
\label{eq:20171206a}
\bar{\bx}^{\iter+1} = \big(\prox_{\gamma_{i} \gs_i} ( x_i^{\iter}
- \gamma_i \nabla_i \bfs(\bx^\iter) ) \big)_{1 \leq i \leq m},
\quad
\bDelta^\iter =  \bx^\iter - \bar{\bx}^{\iter+1}.
\end{equation}
Then, we have 
\begin{equation}
\label{eq:20181114c}
\bar{\bx}^{\iter+1} = \prox_{\bgs}^{\bGammas^{-1}}\big( \bx^\iter - \nabla^{\bGammas^{-1}} \bfs(\bx^\iter) \big),
\quad\bx^{\iter+1} = \bx^\iter + \bvarepsilon^\iter \odot (\bar{\bx}^{\iter+1} - \bx^\iter),
\end{equation}
and, recalling \eqref{eq:algoPRCD2}, that for every $i \in [m]$ such that $\varepsilon^\iter_i = 1$,
\begin{equation}
\bar{x}^{\iter+1}_{i} = \prox_{\gamma_{i} \gs_{i}} \big( x^{\iter}_{i} 
- \gamma_{i} \nabla_{i} \bfs(\bx^\iter) \big) =x^{\iter+1}_i,
\quad \Delta^\iter_{i} = x^\iter_{i} - x^{\iter+1}_{i}.
\end{equation}
Note that $\bx^\iter$ and $\bar{\bx}^{\iter+1}$ are functions of the random
variables $\bvarepsilon^0, \dots, \bvarepsilon^{\iter-1}$ only, 
hence they are both discrete random variables, which are measurable with respect 
to $\Fsc_{\iter-1}$.

\subsection{An abstract principle for stochastic convergence}

We provide an abstract convergence principle for stochastic descent algorithms 
 in the same spirit of \cite[Theorem~3.10]{Sal17}. It simultaneously addresses
 the convergence of the iterates and that of the function values.

\begin{theorem}
\label{p:SMFejer}
Let $\HH$ be a separable real Hilbert space with norm $\norm{\cdot}$.
Let $\Fss\colon \HH \to \left]-\infty,+\infty\right]$ be a proper, lower semicontinuous, and convex function and set $\bSS_* = \argmin \Fss$ and $\Fss_* = \inf \Fss$.
Let $(\bx^{\iter})_{\iter \in \N}$ be a sequence of $\bHH$-valued random variables
such that $\bx^0 \equiv \bxx^0 \in \dom \Fss$ and, for every $\iter \in \N$, $\Fss(\bx^{\iter})$ is $\PP$-summable.
Consider the following conditions
\begin{enumerate}[{\rm P1}]
\item\label{SMFejer_a} $(\EE[\Fss(\bx^{\iter})])_{\iter \in \N}$ is decreasing.
\item\label{SMFejer_b} 
There exist a sequence $(\mathfrak{X}_{\iter})_{\iter \in \N}$ of sub-sigma algebras of 
$\mathfrak{A}$ such that, $(\forall\, \iter \in \N)$ 
$\mathfrak{X}_{\iter} \subset \mathfrak{X}_{\iter+1}$ and 
$\bx^{\iter}$ is $\mathfrak{X}_{\iter}$-measurable,
 a sequence $(\xi_\iter)_{\iter \in \N}$ 
of $\mathfrak{X}_{\iter}$-measurable real-valued positive random variables such that
$\sum_{\iter \in \N} \EE[\xi_\iter] \leq b<+\infty$, and
 $a>0$ such that, for every $\bxx \in \dom \Fss$ and $\iter \in \N$,
\begin{equation}
\label{eq:SMFejer}
\EE[\norm{\bx^{\iter+1} - \bxx}^2 \,\vert\, \mathfrak{X}_{\iter} ] 
\leq   \norm{\bx^{\iter} - \bxx}^2 
 + a \EE[ \Fss(\bxx) - \Fss(\bx^{\iter+1}) \,\vert\, \mathfrak{X}_{\iter}] +  \xi_n
\quad \PP\text{-a.s.}
\end{equation}
\item\label{SMFejer_c} There exist 
$(\by^\iter)_{\iter \in \N}$ and $(\bv^\iter)_{\iter \in \N}$,
sequences of $\bHH$-valued random variables, such that 
 $(\forall\, \iter \in \N)$ 
$\bv^\iter\in \partial \Fss(\by^{\iter})$,
  $\by^{\iter} - \bx^{\iter} \rightharpoonup 0$, and $\bv^{\iter} \to 0$ $\PP$-a.s.
\end{enumerate}
Assume \ref{SMFejer_a} and that 
$(\inf_{\iter \in \N} \EE[\Fss(\bx^\iter)]>-\infty ) \Rightarrow\ $\ref{SMFejer_b}.
 Then, the following hold.
\begin{enumerate}[{\rm (i)}]
\item\label{SMFejer_i} $\EE[\Fss(\bx^{\iter})] \to \Fss_*$.
\item\label{SMFejer_ii} Suppose that $\bSS_*\neq \varnothing$.
Then $\EE[\Fss(\bx^{\iter})] -  \Fss_* = o(1/\iter)$ and,
\begin{equation*}
(\forall\, \iter \in \N,\iter \geq 1)\quad
\EE[\Fss(\bx^{\iter})] - \Fss_*  \leq 
\bigg[ \frac{ \mathrm{dist}^2(\bxx^0, \bSS_*) }{a}
+ \frac{b}{a}\bigg] \frac{1}{n}.
\end{equation*}
\item\label{SMFejer_iii}
Suppose that \ref{SMFejer_c} holds and $\bSS_*\neq \varnothing$.
Then, there exists a random variable $\bx_*$ taking values in $\bSS_*$ such that
$\bx^\iter \rightharpoonup \bx_*$ $\PP$-a.s.
\end{enumerate}
\end{theorem}
\begin{proof}
Taking the expectation in \eqref{eq:SMFejer}, we obtain
\begin{equation}
\label{eq:SMFejer2}
a ( \EE[\Fss(\bx^{\iter+1})] -\Fss(\bxx)) \leq  
 \EE[\norm{\bx^{\iter} - \bxx}^2 ] - \EE[\norm{\bx^{\iter+1} - \bxx}^2 ]
 +  \EE[\xi_n].
\end{equation}

\ref{SMFejer_i}:
Since $(\EE[\Fss(\bx^\iter)])_{n \in \N}$ is decreasing, 
$\EE[\Fss(\bx^\iter)]\to\inf_{n \in \N} \EE[\Fss(\bx^\iter)] \geq \Fss_*$. Thus, 
the statement is true if $\inf_{n \in \N} \EE[\Fss(\bx^\iter)] = -\infty$.
Suppose that  $\inf_{n \in \N} \EE[\Fss(\bx^\iter)]>- \infty$ and let $\bxx \in \dom \Fss$. Then, \ref{SMFejer_b} holds and
the right hand side of \eqref{eq:SMFejer2}, being summable, converges to zero.
Therefore,
$\Fss_* \leq \lim_{n \to +\infty} \EE[\Fss(\bx^{\iter+1})] \leq \Fss(\bxx)$. Since $\bxx$
is arbitrary in $\dom \Fss$, $\EE[\Fss(\bx^{\iter})]\to \Fss_*$.

\ref{SMFejer_ii}:
Let $\bxx \in \bSS_*$. Then,  $\inf_{n \in \N} \EE[\Fss(\bx^\iter)]\geq \Fss(\bxx)>- \infty$.
Hence \ref{SMFejer_b} holds and \eqref{eq:SMFejer2} yields
\begin{align*}
a \sum_{\iter \in \N} \big(\EE[\Fss(\bx^{\iter+1})] - \Fss_* \big) 
&\leq \EE \big[ \norm{\bx^0 -  \bxx}^2\big] 
+ \sum_{\iter \in \N} \EE[\xi_{\iter}]
\leq \norm{\bxx^0 -  \bxx}^2
+ b.
\end{align*}
Therefore, $\sum_{\iter \in \N} (\EE[\Fss(\bx^{\iter+1})] - \Fss_* ) 
\leq (1/a) \norm{\bxx^0 -  \bxx}^2  + b/a$.
Since $(\EE[\Fss(\bx^{\iter+1})] - \Fss_*)_{\iter \in \N}$ is decreasing, 
the statement follows from Fact~\ref{lem:monosequence}.

\ref{SMFejer_iii}:
Let $\bxx \in \bSS_*$. Then \ref{SMFejer_b} holds and,
since $\Fss(\bxx) \leq \Fss(\bx^{\iter+1})$,
we derive from \eqref{eq:SMFejer} that, 
\begin{equation}
(\forall\, \iter \in \N)\quad
\EE\big[\norm{\bx^{\iter+1} - \bxx}^2\,\vert\, 
\mathfrak{X}_{\iter}\big]
\leq
 \norm{\bx^\iter -  \bxx}^2 + \xi_\iter \quad \PP\text{-a.s.}
\end{equation}
Note that  $\xi_\iter$ and $ \norm{\bx^\iter -  \bxx}^2$ 
are $\mathfrak{X}_{\iter}$-measurable.
Moreover $\EE[\sum_{\iter \in \N} \xi_\iter] = \sum_{\iter \in \N} \EE[\xi_\iter]<+\infty$
and hence $\sum_{\iter \in \N} \xi_\iter<+\infty$ $\PP$-a.s.
Therefore $(\bx^\iter)_{\iter \in \N}$ is a stochastic quasi-Fej\'er sequence with respect to $\bSS_*$ 
\cite{Erm69}. 
Then, in view of \cite[Proposition~2.3(iv)]{Com15} it is sufficient to prove that 
 the weak limit points of 
$(\bx^\iter)_{\iter \in \N}$ are contained in $\bSS_*$ $\PP$-a.s.
By assumption \ref{SMFejer_c} there exist two sequences of 
$\bHH$-valued random variables
 $(\by^\iter)_{\iter \in \N}$
and $(\bv^\iter)_{\iter \in \N}$  and $\tilde{\Omega} \subset \Omega$, $\PP(\tilde{\Omega})=1$
such that, for every $\omega \in \tilde{\Omega}$, 
 $\bv^\iter(\omega)\in \partial \Fss(\by^{\iter}(\omega))$,
$\by^{\iter}(\omega) - \bx^{\iter}(\omega) \rightharpoonup 0$, $\bv^{\iter}(\omega) \to 0$.
Let $\omega \in \tilde{\Omega}$ and let $(\bx^{n_k}(\omega))_{\iter \in \N}$ 
be a subsequence of $(\bx^\iter(\omega))_{\iter \in \N}$ such that
 $\bx^{n_k}(\omega) \rightharpoonup \bar{\bxx}$, for some $\bar{\bxx} \in \bHH$.
Then, 
\begin{equation*}
\by^{n_k}(\omega) \rightharpoonup \bar{\bxx}, \bv^{n_k}(\omega) \to 0, \ 
\bv^{n_k}(\omega) \in \partial \Fss(\by^{n_k}(\omega)).
\end{equation*}
Since $\partial \Fss$ is weakly-strongly closed \cite{book1}, we have $0 \in \partial \Fss(\bar{\bxx})$, 
so $\bar{\bxx} \in \bSS_*$.
\end{proof}

\begin{remark}
 Inequalities similar to \eqref{eq:SMFejer} appear implicitly in 
the analysis of several deterministic and stochastic 
 algorithms \cite{Ber11,Kiw06,Nem09}, to get rate of convergence for the function values. Moreover, \eqref{eq:SMFejer} is related also to the concept introduced in \cite{Lin18},
 in a deterministic setting. 
\end{remark}

\subsection{Convergence under convexity and strong convexity assumptions}

In this section we address the convergence of Algorithm~\ref{algoRCD} in the convex and strongly convex case. The main results consist in the $o(1/\iter)$ rate of convergence 
for the mean of the function values and in the 
almost sure weak convergence of the iterates.
We start by recalling a standard result
(see \cite[Lemma~3.12(iii)]{Sal17}). Here we give a slightly
more general version, including the moduli of strong convexity.
The proof is given in Appendix~\ref{sec:appB} for reader's convenience.
\begin{lemma}
\label{lem:20190313b}
Let $\HHs$ be a real Hilbert space.
Let $\varphi\colon \HHs \to \R$ be differentiable and convex with modulus of strong convexity $\mu_\varphi \geq 0$
and $\psi\colon \HHs\to \left]-\infty,+\infty\right]$ be proper, lower semicontinuous,
and convex with modulus of strong convexity $\mu_\psi \geq 0$. 
Let $\xx \in \HHs$ and set $\xx^+ = \prox_{\psi}(\xx - \nabla \varphi(\xx))$. 
Then, 
for every $\zz \in \HHs$,
\begin{align*}
\nonumber(1 + \mu_\psi) \scalarp{\xx - \xx^+ , \zz - \xx} 
&\leq  \Big((\varphi + \psi)(\zz) - (\varphi + \psi)(\xx)  - \frac{\mu_\varphi + \mu_\psi}{2} \norm{\zz - \xx}^2 \Big)\\[1ex]
&\ + \big(\psi(\xx) - \psi(\xx^+) +  \scalarp{\nabla \varphi(\xx), \xx - \xx^+}\big) 
-  \Big(1 + \frac{\mu_\psi}{2}\Big) \norm{\xx - \xx^+}^2.
\end{align*}
\end{lemma}

\begin{proposition}
\label{p:20190610a}
 Let \ref{eq:A1}--\ref{eq:A3} be satisfied.
Let $(\nu_{i})_{1 \leq i \leq m} \in \R^m_{++}$ and suppose that 
\ref{eq:S1} holds.
Let $(\bx^\iter)_{\iter \in \N}$ be generated by  Algorithm~\ref{algoRCD}
with, for every $i \in [m]$, $\gamma_i < 2/\nu_i$.
Set $\delta = \max_{1 \leq i \leq m} \gamma_i  \nu_i$
 and 
$\pp_{\min} = \min_{1 \leq i \leq m} \pp_i$.
Let $\Gammas^{-1}$ be as in \eqref{eq:W} and 
 $\mu_{\Gammas^{-1}}$ and $\sigma_{\Gammas^{-1}}$ be the moduli of strong convexity of 
$\bfs$ and $\bgs$ respectively, in the norm $\norm{\cdot}_{\Gammas^{-1}}$.
Set $\Fs = \bfs+\bgs$.
Then, 
\begin{align}
\nonumber
(1 + \sigma_{\Gammas^{-1}})) \scalarp{\bx^{\iter} - \bar{\bx}^{\iter+1}, \bxx - \bx^{\iter}}_{\Gammas^{-1}} 
&\leq  \frac{1}{\pp_{\min}}\EE \big[ \Fs(\bx^n) - \Fs(\bx^{n+1}) \,\vert\, \Fsc_{\iter-1}  \big]  \\[1ex]
\nonumber&\qquad + \bigg( \Fs(\bxx) -\Fs(\bx^\iter)
- \frac{\mu_{\Gammas^{-1}} + \sigma_{\Gammas^{-1}}}{2} \norm{\bx^\iter - \bxx}_{\bGammas^{-1}}^2 \bigg)\\[1ex]
\label{eq:20190610a}&\qquad +
\frac{\delta - 2 - \sigma_{\Gammas^{-1}}}{2} \norm{\bx^{\iter} - \bar{\bx}^{\iter+1}}_{\bGammas^{-1}}^2.
\end{align}
\end{proposition}
\begin{proof}
Let $\bxx \in \dom \Fs$ and $\iter \in \N$. Since for all $\bvv \in \bHH$, $\scalarp{\nabla^{\bGammas^{-1}} \bfs (\bx^n), \bvv}_{\bGammas^{-1}} =
 \scalarp{\nabla \bfs(\bx^n), \bvv}$, we 
derive from Lemma~\ref{lem:20190313b},
 written in the norm $\norm{\cdot}_{\bGammas^{-1}}$, and
 \eqref{eq:20181114c} that
\begin{align}
\nonumber
(1 + \sigma_{\Gammas^{-1}})) \scalarp{\bx^{\iter} - \bar{\bx}^{\iter+1}, \bxx - \bx^{\iter}}_{\Gammas^{-1}} 
&\leq  \big(\bgs(\bx^n) - \bgs(\bar{\bx}^{n+1}) +  \scalarp{\nabla \bfs(\bx^n), \bx^\iter - \bar{\bx}^{\iter+1}}\big)  \\[0.8ex]
\nonumber&\qquad + \Big( \Fs(\bxx) -\Fs(\bx^\iter)
- \frac{\mu_{\Gammas^{-1}} + \sigma_{\Gammas^{-1}}}{2} \norm{\bx^\iter - \bxx}_{\bGammas^{-1}}^2 \Big)\\[0.8ex]
\label{eq:20190313f}&\qquad - \Big(1 + \frac{\sigma_{\Gammas^{-1}}}{2} \Big)\norm{\bx^{\iter} - \bar{\bx}^{\iter+1}}_{\bGammas^{-1}}^2.
\end{align}
Now, we majorize $\bgs(\bx^n) - \bgs(\bar{\bx}^{n+1}) +  \scalarp{\nabla \bfs(\bx^n), \bx^n - \bar{\bx}^{n+1}}$. By
Fact~\ref{f:20190112d} and  Fact~\ref{lem:20181024a}, we have
\begin{align*}
\bgs(\bx^n) - \bgs(\bar{\bx}^{n+1}) &+  \scalarp{\nabla \bfs(\bx^n), \bx^n - \bar{\bx}^{n+1}}\\
& = \EE \bigg[ \sum_{i=1}^m \frac{\varepsilon^n_i}{\pp_i} 
\Big( \gs_i(x_i^n) - \gs_i(\bar{x}_i^{n+1}) + \scalarp{\nabla_i \bfs(\bx^n), x_i^n - \bar{x}_i^{n+1}} \Big) \,\Big\vert\, \Fsc_{\iter-1} \bigg].
\end{align*}
Moreover,
\begin{align*}
\sum_{i=1}^m \frac{\varepsilon^n_i}{\pp_i} 
&\Big( \gs_i(x_i^n) - \gs_i(\bar{x}_i^{n+1}) + \scalarp{\nabla_i \bfs(\bx^n), x_i^n - \bar{x}_i^{n+1}} \Big)\\
& =  \sum_{i=1}^m \frac{1}{\pp_i} 
\Big( \gs_i(x_i^n) - \gs_i(x_i^{n+1}) 
+ \scalarp{\nabla_i \bfs(\bx^n), x_i^n - x_i^{n+1}} \Big)\\[0.8ex]
& = \frac{1}{\pp_{\min}}
\Big( \bgs(\bx^n) - \bgs(\bx^{n+1}) +  \scalarp{\nabla \bfs(\bx^n), \bx^\iter - \bx^{\iter+1}} \Big)\\[0.8ex]
 &\qquad - \sum_{i=1}^m  \bigg(\underbrace{\frac{1}{\pp_{\min}} - \frac{1}{\pp_{i}}}_{\geq 0} \bigg)
\Big( \gs_i(x_i^n) - \gs_i(x_i^{n+1}) 
+ \scalarp{\nabla_i \bfs(\bx^n), x_i^n - x_i^{n+1}} \Big)\\[0.8ex]
&\leq \frac{1}{\pp_{\min}}
\Big( \bgs(\bx^n) - \bgs(\bx^{n+1}) 
+  \scalarp{\nabla \bfs(\bx^n), \bx^\iter - \bx^{\iter+1}} \Big)\\[0.8ex]
 &\qquad - \bigg( 1 + \frac{\sigma_{\Gammas^{-1}}}{2} \bigg)\sum_{i=1}^m \bigg(\frac{1}{\pp_{\min}} - \frac{1}{\pp_{i}} \bigg)
 \frac{\varepsilon^\iter_i}{\gamma_i} \norm{\Delta_i^\iter}^2,
\end{align*}
where in the last inequality we used that
\begin{equation}
\label{eq:20181108b}
-\big(\gs_i(x_i^{\iter}) - \gs_i(x_i^{\iter+1}) + \scalarp{\nabla_i \bfs(\bx^\iter), x_i^\iter - x_i^{\iter+1}} 
 \big) 
\leq - \frac{\varepsilon^\iter_i}{\gamma_i}
 \bigg( 1 + \frac{\sigma_{\Gammas^{-1}}}{2} \bigg)\norm{\Delta_i^\iter}^2,
\end{equation}
which was obtained from 
\eqref{eq:20181112a} with
\begin{equation*}
\xx_i = x_i^\iter, \quad
\yy_i = x_i^{\iter+1},\quad
\vv_i = \frac{x^\iter_{i} - x^{\iter+1}_{i}}{\gamma_{i}} - \nabla_{i} \bfs (\bx^\iter) 
\in \partial \gs_{i}(x^{\iter+1}_{i}),\quad\text{for }\varepsilon^\iter_i=1.
\end{equation*}
Therefore, 
\begin{align}
\nonumber \bgs(\bx^n) - \bgs(\bar{\bx}^{n+1}) 
&+  \scalarp{\nabla \bfs(\bx^n), \bx^n - \bar{\bx}^{n+1}}\\
\nonumber&\leq \frac{1}{\pp_{\min}}\EE \big[ \bgs(\bx^n) - \bgs(\bx^{n+1}) +  \scalarp{\nabla \bfs(\bx^n), \bx^n - \bx^{n+1}} \,\vert\, \Fsc_{\iter-1}  \big]\\[1ex]
\label{eq:20190313i}&\qquad - \frac{1}{\pp_{\min}}\bigg( 1 + \frac{\sigma_{\Gammas^{-1}}}{2} \bigg) \sum_{i=1}^m  
 \frac{\pp_i}{\gamma_i} \norm{\Delta_i^\iter}^2 
 + \bigg( 1 + \frac{\sigma_{\Gammas^{-1}}}{2} \bigg) \norm{\bar{\bx}^{n+1} - \bx^n}^2_{\bGammas^{-1}}.
\end{align}
Next, it follows from \eqref{eq:20181114c}, \ref{eq:S1}, and Fact~\ref{f:20190112d} that
\begin{equation*}
\EE[\scalarp{\nabla \bfs(\bx^\iter), \bx^{\iter} - \bx^{\iter+1}} \,\vert\, \Fsc_{\iter-1}]
 \leq \EE[\bfs(\bx^\iter) -  \bfs(\bx^{\iter+1}) \,\vert\, \Fsc_{\iter-1}]
+ \frac{1}{2}\sum_{i=1}^m \pp_i \nu_i \norm{\Delta_i^\iter}^2.
\end{equation*}
Then, we derive from \eqref{eq:20190313i} that
\begin{align*}
 \bgs(\bx^n) - \bgs(\bar{\bx}^{n+1}) 
&+  \scalarp{\nabla \bfs(\bx^n), \bx^n - \bar{\bx}^{n+1}}\\
&\leq \frac{1}{\pp_{\min}}\EE \big[ \Fs(\bx^n) - \Fs(\bx^{n+1}) \,\vert\, \Fsc_{\iter-1}  \big]\\[1ex]
&\qquad - \frac{1}{2\pp_{\min}}\sum_{i=1}^m  \bigg( 2 + \sigma_{\Gammas^{-1}} 
- \gamma_i \nu_i \bigg) 
 \frac{\pp_i}{\gamma_i} \norm{\Delta_i^\iter}^2 
 + \bigg( 1 + \frac{\sigma_{\Gammas^{-1}}}{2} \bigg) \norm{\bar{\bx}^{n+1} - \bx^n}^2_{\bGammas^{-1}}.
\end{align*}
The statement follows from \eqref{eq:20190313f}, considering that
\begin{align*}
\frac{1}{\pp_{\min}} \sum_{i=1}^m \bigg(\gamma_i\nu_i - 2  - \sigma_{\Gammas^{-1}}\bigg) 
\frac{\pp_i}{\gamma_i}\norm{\Delta_i^\iter}^2 
&\leq
\underbrace{\frac{\delta - 2 - \sigma_{\Gammas^{-1}}}{\pp_{\min}}}_{\leq0} \sum_{i=1}^m  
\frac{\pp_i}{\gamma_i}\norm{\Delta_i^\iter}^2\\
&\leq 
\frac{\delta - 2 - \sigma_{\Gammas^{-1}}}{\pp_{\min}} \sum_{i=1}^m  
\frac{\pp_{\min}}{\gamma_i}\norm{\Delta_i^\iter}^2\\
&=(\delta - 2 - \sigma_{\Gammas^{-1}}) \norm{\bx^n - \bar{\bx}^{n+1}}^2_{\bGammas^{-1}}.
\qedhere
\end{align*}
\end{proof}

\begin{proposition}
\label{p:20190313c}
 Let \ref{eq:A1}--\ref{eq:A3} be satisfied.
 Let $\bGammas^{-1}$ and $\bWW$ be as in \eqref{eq:W} and
 $(\bx^\iter)_{\iter \in \N}$ be generated by Algorithm~\ref{algoRCD}.
Let $n \in \N$ and $\bx$ be an $\bHH$-valued  random variable
which is measurable w.r.t. $\Fsc_{\iter-1}$. Then
\begin{equation}
\label{eq:20190313c}
\EE[\norm{\bx^{\iter+1} - \bx}_\bWW^2 \,\vert\, \Fsc_{\iter-1} ] -  \norm{\bx^\iter - \bx}^2_\bWW= 
\norm{\bar{\bx}^{\iter+1} - \bx}^2_{\bGammas^{-1}} 
 - \norm{\bx^\iter - \bx}^2_{\bGammas^{-1}}
\end{equation}
and
$\EE[\norm{\bx^{\iter+1} - \bx^\iter}_\bWW^2 \,\vert\, \Fsc_{\iter-1} ]
= \norm{\bar{\bx}^{\iter+1} - \bx^\iter}^2_{\bGammas^{-1}}$.
\end{proposition}
\begin{proof}
If follows from \eqref{eq:20181114c}, Fact~\ref{f:20190112d}, 
and Fact~\ref{lem:20181024a} that
\begin{align*}
\EE[\norm{\bx^{\iter+1} - \bx}_\bWW^2 \,\vert\, \Fsc_{\iter-1} ]
&= \EE\bigg[ \sum_{i=1}^m \frac{1}{\gamma_i \pp_i} \norm{x^{\iter+1}_i - x_i}^2
\,\big\vert\, \Fsc_{\iter-1}  \bigg]\\
& = \EE\bigg[ \sum_{i=1}^m \frac{\varepsilon_i^\iter}{\gamma_i \pp_i} \norm{\bar{x}^{\iter+1}_i - x_i}^2
\,\big\vert\, \Fsc_{\iter-1}  \bigg]
+ \EE\bigg[ \sum_{i=1}^m \frac{1 -\varepsilon_i^\iter}{\gamma_i \pp_i} \norm{x^{\iter}_i - x_i}^2
\,\big\vert\, \Fsc_{\iter-1}  \bigg]\\
& = \norm{\bar{\bx}^{\iter+1} - \bx}^2_{\bGammas^{-1}} 
+ \norm{\bx^\iter - \bx}^2_W - \norm{\bx^\iter - \bx}^2_{\bGammas^{-1}}.
\end{align*}
The second equation follows from \eqref{eq:20190313c}, by choosing $\bx = \bx^\iter$.
\end{proof}

The following result is a stochastic version of \cite[Proposition~3.15]{Sal17}.

\begin{proposition}
 \label{p:20180927a}
 Let \ref{eq:A1}--\ref{eq:A3} be satisfied.
Let $(\nu_{i})_{1 \leq i \leq m} \in \R^m_{++}$ and suppose that 
\ref{eq:S1} holds.
Let $(\bx^\iter)_{\iter \in \N}$ be generated by  Algorithm~\ref{algoRCD}
with, for every $i \in [m]$, $\gamma_i < 2/\nu_i$.
Set $\delta = \max_{1 \leq i \leq m} \gamma_i  \nu_i$
 and 
$\pp_{\min} = \min_{1 \leq i \leq m} \pp_i$.
Let $\Gammas^{-1}$ and $\bWW$ be as in \eqref{eq:W} and 
 $\mu_{\Gammas^{-1}}$ and $\sigma_{\Gammas^{-1}}$ be the moduli of strong convexity of 
$\bfs$ and $\bgs$ respectively, in the norm $\norm{\cdot}_{\Gammas^{-1}}$.
Set $\Fs = \bfs+\bgs$.
Then, the following hold.
\begin{enumerate}[{\rm (i)}]
\item\label{p:20180927a_0} $( \EE[\Fs(\bx^\iter)])_{\iter \in \N}$ is decreasing.
\item\label{p:20180927a_i} Suppose that $\inf_{\iter \in \N} \EE[\Fs(\bx^\iter)]>0$.
Then, 
\begin{equation*}
\sum_{\iter \in \N} \norm{\bar{\bx}^{\iter+1} - \bx^\iter}_{\Gammas^{-1}}^2
= \sum_{\iter \in \N} \EE \big[ \norm{\bx^\iter - \bx^{\iter+1}}_\bWW^2 \big\vert 
\Fsc_{\iter-1} \big]
<+\infty\quad\PP\text{ a.s.}
\end{equation*}
\item\label{p:20180927a_ii} For every $\iter \in \N$ and every $\bxx \in \dom \Fs$ 
\begin{align}
\nonumber
\nonumber(1 + \sigma_{\Gammas^{-1}})\EE[\norm{\bx^{\iter+1} - \bxx}_\bWW^2 \,\vert\, \Fsc_{\iter-1} ] &
\leq (1 + \sigma_{\Gammas^{-1}})  \norm{\bx^{\iter} - \bxx}_{\bWW}^2 
\\[1ex]
\nonumber&\ - 2 \bigg(  \Fs(\bx^\iter) - \Fs(\bxx)
+ \frac{\mu_{\Gammas^{-1}} + \sigma_{\Gammas^{-1}}}{2} \norm{\bx^\iter - \bxx}_{\bGammas^{-1}}^2 \bigg)\\[1ex]
\nonumber&\ +\frac{2}{\pp_{\min}}
\Big(\frac{(\delta-1)_+}{2+\sigma_{\Gammas^{-1}}-\delta}+1\Big)  
\EE[  \Fs(\bx^\iter) - \Fs(\bx^{\iter+1}) \,\vert\, \Fsc_{\iter-1} ].  
\end{align}
\end{enumerate}
\end{proposition}
\begin{proof}
Let $\iter \in \N$ and $\bxx \in \dom \Fs$.
Since
\begin{equation*}
\norm{\bx^\iter - \bxx}_{\bGammas^{-1}}^2 - \norm{\bar{\bx}^{\iter+1} - \bxx}_{\bGammas^{-1}}^2
= - \norm{ \bx^\iter - \bar{\bx}^{\iter+1}}_{\bGammas^{-1}}^2 
+ 2 \scalarp{\bx^{\iter} - \bar{\bx}^{\iter+1}, \bx^\iter - \bxx}_{\bGammas^{-1}},
\end{equation*}
we derive from \eqref{eq:20190610a}, multiplied by $2$, that
\begin{align}
\nonumber(1 + \sigma_{\Gammas^{-1}}) \norm{\bar{\bx}^{\iter+1} - \bxx}_{\bGammas^{-1}}^2
&\leq (1 + \sigma_{\Gammas^{-1}})  \norm{\bx^{\iter} - \bxx}_{\bGammas^{-1}}^2 
+ (\delta - 1)  \norm{\bar{\bx}^{\iter+1} - \bx^{\iter}}_{\bGammas^{-1}}^2 \\[1ex]
\nonumber&\quad+\frac{2}{\pp_{\min}}  
\EE[  \Fs(\bx^\iter) - \Fs(\bx^{\iter+1}) \,\vert\, \Fsc_{\iter-1} ]  \\[1ex]
\label{eq:20181124a}&\quad - 2 \bigg( \Fs(\bx^\iter) - \Fs(\bxx) 
+ \frac{\mu_{\Gammas^{-1}} + \sigma_{\Gammas^{-1}}}{2} \norm{\bx^\iter - \bxx}_{\bGammas^{-1}}^2 \bigg).
\end{align}
Then for an $\bHH$-valued $\Fsc_{\iter-1}$-measurable random variable $\bx$, Proposition~\ref{p:20190313c} yields
\begin{align}
\nonumber
(1 + \sigma_{\Gammas^{-1}})&\EE[\norm{\bx^{\iter+1} - \bx}_\bWW^2 \,\vert\, \Fsc_{\iter-1} ] \\[1ex]
\nonumber&\leq (1 + \sigma_{\Gammas^{-1}})  \norm{\bx^{\iter} - \bx}_{\bWW}^2 
+ (\delta - 1) \EE[ \norm{\bx^{\iter+1} - \bx^{\iter}}_\bWW^2 
\,\vert\, \Fsc_{\iter-1} ] \\[1ex]
\nonumber&\qquad\qquad+\frac{2}{\pp_{\min}}  
\EE[  \Fs(\bx^\iter) - \Fs(\bx^{\iter+1}) \,\vert\, \Fsc_{\iter-1} ]  \\[1ex]
\label{eq:20180925h}&\qquad\qquad - 2 \bigg( \Fs(\bx^\iter) - \Fs(\bx) 
+ \frac{\mu_{\Gammas^{-1}} + \sigma_{\Gammas^{-1}}}{2} \norm{\bx^\iter - \bx}_{\bGammas^{-1}}^2 \bigg).
\end{align}
Taking $\bx=\bx^\iter$ in \eqref{eq:20180925h}, we have
\begin{equation}
\label{eq:20180925g}
\frac{\pp_{\min}}{2}(2 +\sigma_{\Gammas^{-1}} - \delta)\EE[\norm{\bx^{\iter+1} - \bx^\iter}_\bWW^2 
\,\vert\, \Fsc_{\iter-1} ] 
\leq \EE\big[  \Fs(\bx^\iter) - \Fs(\bx^{\iter+1}) 
\,\vert\, \Fsc_{\iter-1} \big],
\end{equation}
which plugged into \eqref{eq:20180925h}, 
with $\bx\equiv \bxx \in \dom \Fs$, gives \ref{p:20180927a_ii}.
Moreover, taking the expectation in  \eqref{eq:20180925g}, we obtain
\begin{equation}
\label{eq:20190128a}
\frac{\pp_{\min}}{2}(2 + \sigma_{\Gammas^{-1}} - \delta)\EE\big[\norm{\bx^{\iter+1} - \bx^\iter}_\bWW^2 \big] 
\leq \EE[  \Fs(\bx^\iter)]  -  \EE[\Fs(\bx^{\iter+1})],
\end{equation}
which gives \ref{p:20180927a_0}.
Finally, set for all $\iter \in \N$,
 $\xi_\iter = \EE\big[ \Fs(\bx^\iter) - \Fs(\bx^{\iter+1}) \big\vert 
\Fsc_{\iter-1} \big] \geq 0$. Then 
\begin{equation*}
\EE \bigg[\sum_{\iter=0}^{+\infty} \xi_\iter \bigg] 
= \sum_{\iter=0}^{+\infty} \EE[\xi_\iter] = \sum_{\iter=0}^{+\infty} 
\EE[\Fs(\bx^\iter)] - \EE[\Fs(\bx^{\iter+1})] \leq \EE[\Fs(\bx^0)] - \inf_{\iter \in \N} \EE[\Fs(\bx^\iter)].
\end{equation*}
This shows that if $\inf_{\iter \in \N} \EE[\Fs(\bx^\iter)]>0$, then 
$\sum_{\iter=0}^{+\infty} \xi_\iter$ is $\PP$-integrable and hence 
it is $\PP$-a.s.~finite.
Then \ref{p:20180927a_i} follows from \eqref{eq:20180925g} and Proposition~\ref{p:20190313c}.
\end{proof}

\begin{proposition}
\label{p:20181219c}
Under the same assumptions of Proposition~\ref{p:20180927a},
suppose that 
condition \ref{eq:S1} is replaced by condition \ref{eq:S2}.
Then
\begin{equation*}
(\forall\, \iter \in \N)\qquad\frac{2 + \sigma_{\Gammas^{-1}} - \delta}{2}
\norm{\bx^{\iter+1} - \bx^\iter}_{\bGammas^{-1}}^2 \leq 
 \Fs(\bx^{\iter}) - \Fs(\bx^{\iter+1})
 \quad \PP \text{ a.s.}
\end{equation*}
\end{proposition}
\begin{proof}
We derive from \ref{eq:S2} (since $\bvarepsilon^n$ has the same distribution of
$\bvarepsilon$) and \eqref{eq:20181114c} that
\begin{equation}
\label{eq:20181219b}
\scalarp{\nabla \bfs(\bx^\iter), \bx^{\iter+1} - \bx^{\iter}}
 \leq \bfs(\bx^{\iter}) -  \bfs(\bx^{\iter+1}) 
+ \sum_{i=1}^m 
\frac 1 2 \varepsilon^\iter_i \nu_i \norm{\Delta^\iter_{i}}^2  \quad \PP \text{ a.s.}
\end{equation}
Therefore, summing \eqref{eq:20181108b}, from $i=1$ to $m$, we have
\begin{align*}
\frac{2 + \sigma_{\Gammas^{-1}}}{2}\sum_{i=1}^m \varepsilon_i^\iter
 \frac{1}{\gamma_i}\norm{\Delta_i^\iter}^2 &\leq 
 \bgs(\bx^{\iter}) - \bgs(\bx^{\iter+1}) + \scalarp{\nabla \bfs(\bx^\iter), \bx^\iter - \bxx^{\iter+1}}\\
 & \leq \bgs(\bx^{\iter}) - \bgs(\bx^{\iter+1}) +  \bfs(\bx^{\iter}) -  \bfs(\bx^{\iter+1}) 
 + \frac 1 2 \sum_{i=1}^m 
\varepsilon^\iter_i  \nu_i \norm{\Delta^\iter_{i}}^2\quad \PP \text{ a.s.}
\end{align*}
Hence 
$(1/2) \sum_{i=1}^m 
(2 + \sigma_{\Gammas^{-1}} - \gamma_i \nu_i )\gamma_i^{-1}\varepsilon_i^\iter\norm{\Delta_i^\iter}^2 
\leq  \Fs(\bx^{\iter}) - \Fs(\bx^{\iter+1})$ $\PP$-a.s. 
\end{proof}

\begin{proposition}
\label{prop:20180927a}
Under the assumptions of Proposition~\ref{p:20180927a},
suppose in addition that 
  $\Fs$ is bounded from below.
Then, there exist $(\by^\iter)_{\iter \in \N}$ and $(\bv^\iter)_{\iter \in \N}$,
sequences of $\bHH$-valued random variables, such that
the following hold.
\begin{enumerate}[{\rm (i)}]
\item\label{eq:20190113a_i} 
$\bv^\iter\in \partial \Fs(\by^{\iter})$ $\PP$-a.s.
\item\label{eq:20190113a_ii}  $\by^{\iter} - \bx^{\iter} \to 0$ 
and $\bv^{\iter} \to 0$ $\PP$-a.s.
\end{enumerate}
\end{proposition}
\begin{proof}
It follows from \eqref{eq:20171206a} that,
$(x_i^\iter(\omega) - \bar{x}_i^{\iter+1}(\omega))/\gamma_{i} - \nabla_i \bfs (\bx^\iter(\omega)) 
\in \partial \gs_i(\bar{x}_i^{\iter+1}(\omega))$, 
for all $i \in [m]$ and $\omega \in \Omega$. Hence
\begin{equation*}
\Big( \frac{x_i^\iter(\omega) - \bar{x}_i^{\iter+1}(\omega)}{\gamma_i} \Big)_{1 \leq i \leq m} - \nabla \bfs(\bx^\iter(\omega))
\in \partial \bgs(\bar{\bx}^{\iter+1}).
\end{equation*}
 Set $\by^\iter = \bar{\bx}^{\iter+1}$ and let
 $\bv^\iter \colon \Omega \to \bHH$ be such that, for every $\omega \in \Omega$,
\begin{multline*}
\bv^\iter(\omega) = \Big(\frac{x_i^\iter(\omega) - y_i^\iter(\omega)}{\gamma_i} \Big)_{1 \leq i \leq m}
 + \nabla \bfs (\by^\iter(\omega)) - \nabla \bfs (\bx^\iter(\omega))  \\
 \in \partial \bgs(\by^{\iter}(\omega)) + \nabla \bfs(\by^\iter(\omega)) = \partial \Fs(\by^\iter(\omega)).
\end{multline*}
Clearly $\bv^\iter$ is measurable and hence it is a random variable.
Moreover, 
for every $\omega \in \Omega$,
\begin{align*}
\norm{\bv^\iter(\omega)} \leq \frac{1}{\gamma_{\min}} \norm{\bx^\iter(\omega) - \by^\iter(\omega)}
+ \norm{\nabla \bfs(\by^\iter(\omega)) - \nabla \bfs(\bx^\iter(\omega))}.
\end{align*}
Now, since $\Fs$ is bounded from below, Proposition~\ref{p:20180927a}\ref{p:20180927a_i} yields that 
$(\norm{\by^\iter - \bx^\iter}_{\Gammas^{-1}}^2)_{\iter \in \N}$ is summable $\PP$-a.s.~and hence
$\by^\iter - \bx^\iter\to 0$ $\PP$-a.s. The statement follows from 
the fact  that $\nabla \bfs$ is Lipschitz continuous 
(see Theorem~\ref{thm:stepsizes}\ref{thm:stepsizes_iv}).
\end{proof}

Now we are ready to state one of the main convergence results of this paper.
From one hand, it
 extends to the stochastic setting
a well-known convergence rate of the (deterministic) forward-backward algorithm
\cite{Dav16,Gar17,Sal17}. On the other hand, it proves the almost sure weak convergence of the iterates of Algorithm~\ref{algoRCD} in the convex case.
We stress that none of the works \cite{LuX15,Nec16,Qu16b,Ric15,Ric16,Ric16b,Tap18} addresses this latter aspect.
To the best of our knowledge, \cite{Com15} is the only work that 
proves almost sure weak convergence of the iterates.
However, in \cite[Corollary~5.11]{Com15} the stepsize is set according to the (global) 
Lipschitz constant of $\nabla \bfs$ which,
in general, leads to smaller stepsizes and worse upper bounds on convergence rates. 
See the subsequent discussion.

\begin{theorem}
\label{thm:20171207a}
 Let \ref{eq:A1}--\ref{eq:A3} be satisfied.
Let $(\nu_{i})_{1 \leq i \leq m} \in \R^m_{++}$ and suppose that 
 \ref{eq:S1} holds.
Let $(\bx^\iter)_{\iter \in \N}$ be generated by  Algorithm~\ref{algoRCD}
with, for every $i \in [m]$, $\gamma_i < 2/\nu_i$.
Set $\delta = \max_{1 \leq i \leq m} \gamma_i  \nu_i$
 and 
$\pp_{\min} = \min_{1 \leq i \leq m} \pp_i$.
Let $\bWW$ be as in \eqref{eq:W} and
set  $\Fs= \bfs+\bgs$, $\Fs_* = \inf \Fs$, and $\bSS_* 
= \argmin \Fs \subset \bHH$.
Then, the following hold.
\begin{enumerate}[{\rm (i)}]
\item\label{thm:20171207a_0} $\EE[\Fs(\bx^{\iter})]\to \Fs_*$.
\item\label{thm:20171207a_i} Suppose that $\bSS_*\neq \varnothing$.
Then $\EE[\Fs(\bx^{\iter})] -  \Fs_* = o(1/\iter)$ and, for every integer $\iter \geq 1$,
\begin{equation}
\label{eq:20171207a}
\EE[\Fs(\bx^{\iter})] - \Fs_*  \leq 
\bigg[ \frac{ \mathrm{dist}_\bWW^2(\bxx^0, \bSS_*) }{2}
+ \Big(\frac{\max\big\{1, (2-\delta)^{-1}\big\}}{\pp_{\min}} -1\Big) 
 (\Fs(\bxx^{0})  - \Fs_*)\bigg] \frac{1}{n}.
\end{equation}
Moreover,
there exists a random variable $\bx_*$ taking values in $\bSS_*$ such that
$\bx^\iter \rightharpoonup \bx_*$ $\PP$-a.s.
\end{enumerate}
\end{theorem}
\begin{proof}
Proposition~\ref{p:20180927a}\ref{p:20180927a_ii} 
with $\mu_{\Gammas^{-1}}=\sigma_{\Gammas^{-1}}=0$ gives,
for all $\bxx \in \dom \Fs$ and $\iter \in \N$,
\begin{equation*}
\EE[\norm{\bx^{\iter+1} - \bxx}_\bWW^2 \,\vert\, \Fsc_{\iter-1} ] 
\leq   \norm{\bx^{\iter} - \bxx}_{\bWW}^2  + 2\EE[ \Fs(\bxx) -\Fs(\bx^{\iter+1})\,\vert\,\Fsc_{\iter-1}] + \xi_\iter,
\end{equation*}
where
\begin{equation*}
\xi_\iter = b_1  \EE[\Fs(\bx^\iter) - \Fs(\bx^{\iter+1})\,\vert\, 
\Fsc_{\iter-1}],\qquad b_1 = 2  \bigg(\frac{\max\{1,1/(2-\delta)\}}{\pp_{\min}} - 1\bigg).
\end{equation*}
Note that the random variables $\bx^\iter$'s are discrete with finite range 
and $(\EE[\Fs(\bx^\iter)])_{\iter \in \N}$ is decreasing.
Moreover, 
$\sum_{\iter \in \N} \EE[\xi_\iter] \leq b_1 (\Fs(\bxx^0) - \inf_{\iter \in \N} \EE[\Fs(\bx^\iter)])$.
Therefore, 
 the statement follows from 
 Theorem~\ref{p:SMFejer} and
Proposition~\ref{prop:20180927a}.
\end{proof}

\paragraph{Discussion.}
In the following we examine some crucial aspects related to Algorithm~\ref{algoRCD}.
We suppose that \ref{eq:B0} and \ref{eq:B1} hold and that
for every $i \in [m]$, $\gamma_i = \delta/\nu_i$ with $\delta \in \left]0,2\right[$.
\setlist[description]{font=\normalfont\itshape}
\begin{description}[leftmargin=0cm]
% -----------------------------------------------------
\item[The benefit of a parallel block update.]
Here we discuss the advantage of updating multiple blocks in parallel instead of just a single block.
We consider the setting of a $\tau$-nice uniform block sampling, which was described in 
Section~\ref{subsec:beta}.
In this case, for ever $i,j \in [m]$, with $i\neq j$,
$\pp_i = \pp_j$ and, since $\tau = \sum_{i=1}^m \EE[\varepsilon_i] 
= \sum_{i=1}^m \pp_i$,  we have, for every $i \in [m]$, $\pp_i = \pp: = \tau/m$. 
Moreover, we can set for every $i \in [m]$,
$\nu_i = \beta_1 L_i$ with $\beta_1$ defined as in \eqref{eq:20200316c}.
In order to compare different choices of $\tau$, we normalize the iterations so to match 
the same computational cost per iteration of the standard (full parallel) forward-backward algorithm (FB). 
It follows from \eqref{eq:20171207a} that after $\iter_\tau =\lceil m \bar{\iter}/\tau \rceil$ iterations of Algorithm~\ref{algoRCD},
which have the same total computational cost of  $\bar{\iter}$ iterations of FB, we have
\begin{equation}
\label{eq:20200330a}
\EE[\Fs(\bx^{\lceil m \bar{\iter}/\tau \rceil})] -  \Fs_*  
\leq \bigg[ \beta_1\frac{\mathrm{dist}_{\Lambdas}^2(\bxx^0, \bSS_*)}{2\delta}
+ \Big(\max\Big\{1,\frac{1}{2-\delta}\Big\} -\pp\Big)  ( \Fs(\bxx^{0}) - \Fs_*)\bigg] \frac{1}{\bar{\iter}},
\end{equation}
where $\Lambdas = \bigoplus_{i=1}^m L_i \Id_i$.
Now, since, $\beta_1 = 1 + (\tau-1)(\eta-1)/(m-1)$, we see that 
if $\eta \ll m$ and $\tau\ll m$, then $\beta_1$ 
is close to 1 (and $\pp$ is close to zero), so that $\beta_1$ nearly does not depend on $\tau$, as long as $\tau$
 remains sufficiently small.
For instance, in Section~\ref{sec:experiments} we consider the setting where $m = 10^5$ and
$\eta = 148$. In such case, if we let $\tau=1, 5, 10, 50$, the corresponding $\beta_1$'s
and $\pp$'s are essentially the same so that the right hand side of \eqref{eq:20200330a}
does not change much. Therefore, the above options for $\tau$
require the same total amount of computations (i.e., $m \bar{n}$ block-coordinate updates) and lead essentially 
to the same improvement in the objective function. 
However, a parallel implementation, 
say with $\tau = 50$ on a CPU with $50$ cores, 
will be $50$ times faster than a serial implementation ($\tau=1$)
which uses only one core per iteration.
In summary, in the large scale ($m$ large) and sparse ($\eta\ll m$) setting, 
the parallel strategy ($\tau>1$ and $\tau$ equal to the number of CPU cores) 
is definitely advantageous provided that $\tau$
is sufficiently small compared to $m$.
% -----------------------------------------------------
\item[Comparison with \cite{Com15}.]
The almost sure weak convergence of the iterates of Algorithm~\ref{algoRCD} is also obtained in \cite{Com15}, 
but with stepsizes set according to the global Lipschitz constant of the gradient of $\bfs$.
Let $L$ be the Lipschitz constant of $\nabla \bfs$ and note that 
$\bfs$ is also Lipschitz smooth in the norm $\norm{\cdot}_{\Lambdas}$, 
defined by the operator 
$\Lambdas = \bigoplus_{i=1}^m L_i \Id_i$, with constant $\eta$ (see Corollary~\ref{p:20181014a}\ref{p:20181014a_3}).
Therefore, the results in \cite{Com15} can be applied in the original norm $\norm{\cdot}$ or in the 
norm $\norm{\cdot}_{\Lambdas}$.
In this respect we note that 
since
\begin{equation*}
\prox_{\alpha \bgs}^\Lambdas(\bx^\iter - \alpha \nabla^\Lambdas \bfs(\bx^\iter)) = (\prox_{(\alpha/L_i) \gs_i} (x_i^\iter - (\alpha/L_i) \nabla_i \bfs(\bx^\iter)))_{1 \leq i \leq m}
\qquad (\alpha=\delta/\eta, 0<\delta<2),
\end{equation*} 
the implementation in the norm $\norm{\cdot}_{\Lambdas}$ is nothing but Algorithm~\ref{algoRCD}
with stepsizes $\gamma_i = \delta/(\eta L_i)$.
In both cases Corollary~5.11 in \cite{Com15}, applied in the corresponding norms, 
proves weak convergence of the iterates for  
Algorithm~\ref{algoRCD} with stepsizes $\gamma_i \equiv \delta/L$
and $\gamma_i = \delta/(\eta L_i)$ respectively.
However, Theorem~\ref{thm:20171207a}, together with Theorem~\ref{thm:stepsizes}, 
allows to set the stepsizes as $\gamma_i = \delta/(\beta_{1,i} L_i)$.
Since $\beta_{1,i} L_i$ may be much smaller than $L$ and, in view of Remark~\ref{rmk:20181219ee}, 
it is always smaller than $\eta L_i$, Theorem~\ref{thm:20171207a} 
provides a significant improvement over \cite{Com15} in terms of flexibility in the stepsizes.
% -----------------------------------------------------
\item[The advantage over the standard FB.]
We consider the forward-backward algorithm (FB) 
in the original norm of $\HH$ and in the norm $\norm{\cdot}_{\Lambdas}$. 
All the remarks about the stepsizes discussed in the previous paragraph apply also here.
Moreover, 
in the case $\delta = 1$, standard convergence rate for FB (see e.g., \cite{Dav16}) yields that after $\bar{\iter}$ iterations, we have
\begin{equation}
\label{eq:20200330b}
\Fs(\xx^{\bar{\iter}}) - \Fs_* 
\leq \frac{L}{2} \frac{\mathrm{dist}^2(\xx^0, \bSS_*)}{\bar{\iter}} 
\quad\text{or}\quad 
 \Fs(\xx^{\bar{\iter}}) - \Fs_*  
\leq \frac{\eta}{2} \frac{\mathrm{dist}_{\Lambdas}^2(\xx^0, \bSS_*)}{\bar{\iter}},
\end{equation}
depending on which of the two above implementations of FB we consider.
In order to appropriately compare the rates \eqref{eq:20200330b} with that
of Algorithm~\ref{algoRCD} given in Theorem~\ref{thm:20171207a}
in the following we set $\delta=1$ and analyze two choices of the block sampling.
\begin{enumerate}[(i)]
%--------------------------------------------------------------------------------------
\item Assume that we perform a $\tau$-nice block sampling.
Then we saw that the (normalized) convergence rate of Algorithm~\ref{algoRCD} is
\eqref{eq:20200330a}.
We first note that \eqref{eq:20200330a} reduces to 
the second inequality in \eqref{eq:20200330b} when $\delta=1$ and $\tau = m$.
Comparing the bounds in \eqref{eq:20200330b} with
\eqref{eq:20200330a} (with $\delta=1$) we see that, if we assume that the terms 
$\mathrm{dist}_{\Lambdas}^2(\xx^0, \bSS_*)$ and $\Fs(\bxx^{0}) - \Fs_*$ are about of the same magnitude, then
Algorithm~\ref{algoRCD} features always a better rate than FB if
implemented in the norm $\norm{\cdot}_{\Lambdas}$ (since $\beta_1 \leq \eta$), whereas if FB 
is implemented in the original norm of $\bHH$, Algorithm~\ref{algoRCD} is still a better choice provided that $\beta_1\max_{1 \leq i \leq m} L_i \leq L$.
%--------------------------------------------------------------------------------------
\item Suppose that the block sampling performs on average $\tau$ 
updates per iteration and that, for every $i \in [m]$,  $\pp_i$ is
proportional to the Lipschitz constant $L_i$, that is, 
$\pp_i = \tau L_i/(\sum_{j=1}^m L_j)$ (provided that $\tau \leq (\sum_{j=1}^m L_j)/\max_{1 \leq j \leq m} L_j$).
In this case,
as stated at the beginning of Section~\ref{subsec:beta},  we can let
  $\nu_i = \beta_2 L_i$ and $\gamma_i = 1/\nu_i$. 
Then, 
\eqref{eq:20171207a}
becomes for $\delta=1$ and $\iter =\lceil m \bar{\iter}/\tau \rceil$,
\begin{equation*}
\EE[\Fs(\bx^{\lceil m \bar{\iter}/\tau \rceil})] - \Fs_*  \leq \bigg[ \beta_2
\bar{L}\frac{ \mathrm{dist}^2(\bxx^0, \bSS_*) }{2}
+ \bigg( \frac{\bar{L}}{L_{\mathrm{min}}} - \frac{\tau}{m}\bigg)  
(\Fs(\bxx^{0}) - \Fs_*)\bigg] \frac{1}{\bar{n}},
\end{equation*}
where  $\bar{L} = \sum_{i=1}^m L_i/m$ and $L_{\min} = \min_{1 \leq i \leq m} L_i$.
Here we see that Algorithm~\ref{algoRCD} can be superior to FB if
$(\beta_2 +2 \bar{L}/L_{\min})\bar{L} \leq L$, under the assumption that $\bar{L} \mathrm{dist}^2(\bxx^0, \bSS_*) \simeq \Fs(\bxx^{0}) - \Fs_*$.
\end{enumerate}
\end{description}

We now provide an additional convergence theorem, analyzing the strongly convex case, 
which extends \cite[Theorem~1]{LuX15} and
\cite[Theorem~3]{Tap18} to an arbitrary (not necessarily uniform) sampling and 
to the more general stepsize rule \eqref{eq:stepsizerule}. 
The proof is still based on Proposition~\ref{p:20180927a}\ref{p:20180927a_ii}
and is postponed to Appendix~\ref{sec:appB}. 
%but will closely follow that of the above cited works,
%so it will be postponed in Appendix~\ref{sec:appB}. 

\begin{theorem}
\label{p:20181130a}
Under the same assumptions of Theorem~\ref{thm:20171207a},
let $\mu_{\Gammas^{-1}}$ and $\sigma_{\Gammas^{-1}}$ be the moduli of strong convexity of 
$\bfs$ and $\bgs$ respectively, in the norm $\norm{\cdot}_{\bGammas^{-1}}$,
and suppose that $\mu_{\Gammas^{-1}}+\sigma_{\Gammas^{-1}}>0$
and that $\bSS_* = \{\bxx_*\}$. %\neq \varnothing$
 Then,
for every $n \in \N$,
\begin{equation*}
\EE[\Fs(\bx^{\iter})] - \Fs_* 
\leq 
(1 - \pp_{\min} \bar{\lambda})^\iter
\bigg( \pp_{\min} (1+\sigma_{\Gammas^{-1}} - (\delta-1)_+) 
\frac{\norm{\bxx^0 - \bxx_*}^2_{\bWW}}{2} 
+ \Fs(\bxx^0) - \Fs_* \bigg),
\end{equation*}
where 
\begin{equation}
\label{eq:20200928d}
\bar{\lambda} = 
\begin{cases}
\dfrac{2 - \delta +\sigma_{\Gammas^{-1}}}{1+\sigma_{\Gammas^{-1}}} &\text{if } \delta>1\ \text{and}\ \mu_{\Gammas^{-1}} \geq 2 - \delta\\[2ex]
\dfrac{2(\mu_{\Gammas^{-1}}+\sigma_{\Gammas^{-1}})}{1+\sigma_{\Gammas^{-1}} + (\mu_{\Gammas^{-1}}+\sigma_{\Gammas^{-1}})(1+\sigma_{\Gammas^{-1}})/(1 +\sigma_{\Gammas^{-1}} - (\delta-1)_+)}
&\text{otherwise}.
\end{cases}
\end{equation}
\end{theorem}
\begin{remark}
\label{rmk:20201003b}
Let $\gamma_i = \delta/\nu_i$ and 
 $\bVV = \bigoplus_{i=1}^m \nu_i \Id_i$. Let  $\mu_{\bVV}$ and $\sigma_{\bVV}$ be the moduli of strong convexity 
of $\bfs$ and $\bgs$ respectively, in the norm $\norm{\cdot}_{\bVV}$.
Then it is easy to see that $\mu_{\Gammas^{-1}} = \delta \mu_{\bVV}$ and 
$\sigma_{\Gammas^{-1}} = \delta \sigma_{\bVV}$. Moreover, as in Remark~\ref{rmk:20190126e}, one can also see that
$\mu_{\bVV} \leq 1$. Then, \eqref{eq:20200928d} becomes
\begin{equation}
\label{eq:20200928e}
\bar{\lambda} = 
\begin{cases}
\dfrac{2/\delta - 1 +\sigma_{\bVV}}{1/\delta+\sigma_{\bVV}} &\text{if } 
\delta \geq \dfrac{2}{1+\mu_{\bVV}}\\[2ex]
\dfrac{2(\mu_{\bVV}+\sigma_{\bVV})}{1/\delta+\sigma_{\bVV} + (\mu_{\bVV}+\sigma_{\bVV})
(1/\delta+\sigma_{\bVV})/(2/\delta - 1 +\sigma_{\bVV})}
&\text{if } 1<\delta 
\leq \dfrac{2}{1+\mu_{\bVV}}\\[2ex]
\dfrac{2(\mu_{\bVV}+\sigma_{\bVV})}{1/\delta+\sigma_{\bVV} + (\mu_{\bVV}+\sigma_{\bVV})}
&\delta \leq 1.
\end{cases}
\end{equation}
One can check that the maximum of $\bar{\lambda}$ with respect to $\delta \in \left]0,2\right[$ is 
\begin{equation}
\bar{\lambda}_{\mathrm{opt}} =  \frac{4(\mu_{\bVV} + \sigma_{\bVV})}{(\sqrt{1+\sigma_{\bVV}} + \sqrt{\mu_{\bVV} + \sigma_{\bVV}})^2}
\in \left]0,1\right],
\end{equation}
which is achieved at
\begin{equation}
\delta = \delta_{\mathrm{opt}} := \frac{2}{1 - \sigma_{\bVV} + \sqrt{(\mu_{\bVV} + \sigma_{\bVV})(1 + \sigma_{\bVV})}}
 \in \left[ 1,2\right[.
\end{equation}
Note that if $\mu_{\bVV}<1$ (as is normally the case), then $\bar{\lambda}_{\mathrm{opt}}\in \left]0,1\right[$ and $\delta_{\mathrm{opt}}>1$.
\end{remark}
\begin{remark}
\label{rmk:20190130a}
If $\mu$ and $\sigma$ are the moduli of strong convexity 
of $\bfs$ and $\bgs$ respectively in the original norm.
Let, for every $i \in [m]$, $\gamma_i = \delta/\nu_i$ and set $\nu_{\max} = \max_{1 \leq i \leq m} \nu_i$. 
Then $\mu_{\bVV} = \mu/\nu_{\max}$
and $\sigma_{\bVV} = \sigma/\nu_{\max}$
Therefore, the optimal stepsizes are achieved for
\begin{equation}
\delta = \frac{2 \nu_{\max}}{\nu_{\max} - \sigma + \sqrt{(\mu + \sigma)(\nu_{\max} + \sigma)}}
\end{equation}
and the corresponding rate in 
Theorem~\ref{p:20181130a}  becomes
\begin{equation*}
\EE[\Fs(\bx^{\iter})] - \Fs_* \leq 
\bigg( 1 - \pp_{\min}
\frac{4 (\mu+\sigma)}{(\sqrt{\nu_{\max}+ \sigma}+ \sqrt{\mu+\sigma})^2}
\bigg)^\iter \mathrm{const}.
\end{equation*}
\end{remark}

\begin{remark}
Suppose that the block sampling is uniform, that is,
$\pp_i = \pp$ for all $i \in [m]$ and let, for every $i \in [m]$, $\gamma_i = 1/\nu_i$.
Then $\delta = 1$ and 
Theorem~\ref{p:20181130a} reduce to
\begin{equation}
\label{eq:20201003a}
\EE[\Fs(\bx^{\iter})] - \Fs_* 
\leq \bigg(1 - \pp\frac{2(\mu_{\Gammas^{-1}} + \sigma_{\Gammas^{-1}})}{1 + \mu_{\Gammas^{-1}} + 2 \sigma_{\Gammas^{-1}}}\bigg)^\iter 
\bigg(  (1+\sigma_{\Gammas^{-1}}) 
%\frac{\mathrm{dist}_{\bGammas^{-1}}^2(\bxx^0, \bSS_*)}{2} 
\frac{\norm{\bxx^0 - \bxx_*}^2_{\bGammas^{-1}}}{2} 
+ \Fs(\bxx^0) - \Fs_* \bigg).
\end{equation}
This result was obtained in \cite[Theorem~3]{Tap18},
which is in turn a generalization of \cite[Theorem~1]{LuX15},
treating the serial case ($\PP(\sum_{i=1}^m \varepsilon_i = 1)=1$).
Thus, Theorem~\ref{p:20181130a} and the subsequent Remark~\ref{rmk:20201003b} show that the rate in \eqref{eq:20201003a} can indeed be improved by choosing $\delta>1$.
\end{remark}

\subsection{Linear convergence under error bound conditions}
\label{sec:errorbounds}

In this section we analyze the convergence of Algorithm~\ref{algoRCD}
under error bound conditions. We improve and simplify
the results given in \cite{Nec16}.
In the rest of the section we assume \ref{eq:A1} and \ref{eq:A2}.
Moreover, we let $\bXX \subset \bHH$,  $\Fs = \bfs + \bgs$, $\Fs_* = \inf \Fs$, 
and suppose $\bSS_*:=\argmin \Fs \neq \varnothing$.  

We consider the following
condition,  which
was studied in \cite{Dru18} in connection with 
the proximal gradient method
and is known as \emph{Luo-Tseng error bound} condition \cite{Luo93}.

\begin{enumerate}[{\rm EB}]
\item\label{eq:EB1} For some $c_{\bXX,\bGammas^{-1}}>0$, we have
\begin{equation}
(\forall\,\bxx \in \bXX)\quad\mathrm{dist}_{\bGammas^{-1}}(\bxx, \bSS_*)
\leq c_{\bXX,\bGammas^{-1}}
\norm{\bxx - \prox^{\bGammas^{-1}}_{\bgs} (\bxx - \nabla^{\bGammas^{-1}}\bfs(\bxx) )}_{\bGammas^{-1}}.
\end{equation}
\end{enumerate}

 \begin{remark}\ 
 \label{rmk:20190603a}
 \begin{enumerate}[{\rm (i)}]
\item Another popular error bound condition is that of the metric subregularity of the subdifferential.
More precisely, $\partial^{\bGammas^{-1}} \Fs$ is $2$-\emph{metrically subregular} 
on $\bXX$ with respect to the metric $\norm{\cdot}_{\bGammas^{-1}}$\cite{Dru18,Gar17}
if for some 
 $\zeta_{\bXX,\bGammas^{-1}}>0$ the following holds
\begin{equation}
\label{eq:20181114h}
(\forall\, \bxx \in \bXX)\quad
\mathrm{dist}_{\bGammas^{-1}}(\bxx, \bSS_*) \leq \frac{1}{\zeta_{\bXX,\bGammas^{-1}}}
\mathrm{dist}_{\bGammas^{-1}}(\bzero, \partial^{\bGammas^{-1}} \Fs(\bxx)).
\end{equation}
 \item\label{rmk:20190129e_00} \ref{eq:EB1} and \eqref{eq:20181114h} are equivalent if $\bgs=0$,
 since in that case $\prox_{\bgs}^{\bGammas^{-1}} = \Id$
 and $c_{\bXX,\bGammas^{-1}} = \zeta_{\bXX, \bGammas^{-1}}^{-1}$.
 \item\label{rmk:20190129e_0} Since $\partial^{\bGammas^{-1}} \Fs (\bx) = \bGammas \partial \Fs(\bx)$
 and $\norm{\cdot} \geq \gamma_{\min}^{1/2}\norm{\cdot}_{\bGammas^{-1}}$,
it follows that if for every $\bx \in \bXX$, $
\mathrm{dist}(\bxx, \bSS_*) \leq \zeta^{-1}_{\bXX,\Id}
\mathrm{dist}(\bzero, \partial \Fs(\bxx))$,
then \eqref{eq:20181114h} holds with constant $\zeta_{\bXX, \bGammas^{-1}} = \gamma_{\min} \zeta_{\bXX,\Id}$.
\item\label{rmk:20190603a_iv} \cite[Theorem~3.5]{Dru18} yields that
for any Hilbert norm $\norm{\cdot}_\bWW$, 
 $\norm{\bxx - \prox^{\bWW}_{\bgs} (\bxx - \nabla^{\bWW}\bfs(\bxx) )}_{\bWW} 
\leq \mathrm{dist}_{\bWW}(\bzero, \partial^{\bWW}\Fs(\bxx))$. So,
 if \ref{eq:EB1} holds on $\bXX$, then \eqref{eq:20181114h} holds on $\bXX$ 
 with $\zeta_{\bXX,\bGammas^{-1}} = c_{\bXX,\bGammas^{-1}}^{-1}$. 
 In \cite[Theorem~3.4-3.5]{Dru18} also the reverse implication was shown when
 $\bfs$ is Lipschitz smooth and
 $\bXX$ is a sublevel set of $\Fs$. 
 \end{enumerate}
 \end{remark}
   
\begin{remark}
\label{rmk:20190129e}
In \cite[Corollary~3.6]{Dru18} condition \ref{eq:EB1} was shown to be equivalent to the
following \emph{quadratic growth condition} (also called $2$-conditioning in \cite{Gar17})
\begin{equation}
\label{eq:20181218b}
(\exists \alpha_{\bXX,\bGammas^{-1}} >0)(\forall\, \bxx \in \bXX)\quad \Fs(\bxx) - \inf \Fs 
\geq \frac{\alpha_{\bXX,\bGammas^{-1}}}{2} \mathrm{dist}_{\bGammas^{-1}}^2(\bxx, \bSS_*),
\end{equation}
on every sublevel set $\bXX = \{\bxx \in \bHH \,\vert\, \Fs(\bxx) - \Fs_* \leq r\}$.
Moreover, the relationships between the constants are $c_{\bXX,\bGammas^{-1}} = (1 + 2/\alpha_{\bXX,\bGammas^{-1}})(1+ L_{\norm{\cdot}_{\bGammas^{-1}}})$
and $\alpha_{\bXX,\bGammas^{-1}} <1/c_{\bXX,\bGammas^{-1}}$. 
Finally,  if the quadratic growth condition \eqref{eq:20181218b}
holds, then \eqref{eq:20181114h} holds on $\bXX$, 
with $\zeta_{\bXX,\bGammas^{-1}} = \alpha_{\bXX,\bGammas^{-1}}/2$.
\end{remark}

We now analyze the convergence of Algorithm~\ref{algoRCD}
under condition \ref{eq:EB1}. 

\begin{theorem}
\label{thm:20181214a}
Under the assumptions of Theorem~\ref{thm:20171207a},
suppose that $\bSS_* \neq \varnothing$ and that
\ref{eq:EB1} holds on a set $\bXX$ such that $\bXX \supset \big\{ \bx^\iter \,\vert\, \iter \in \N\big\}$ $\PP$-a.s.~with $c_{\bXX,\bGammas^{-1}}>0$.
Then, 
\begin{equation}
\label{eq:20190131e}
(\forall\, \iter \in \N)\quad
\EE[\Fs(\bx^{\iter})] - \Fs_* \leq \bigg( 1-  \pp_{\min}\min\bigg\{1,\frac{2 - \delta}{2 c_{\bXX,\bGammas^{-1}}} \bigg\} \bigg)^\iter
\big(\EE[\Fs(\bx^{0})] - \Fs_* \big).
\end{equation}
Moreover, 
there exists a random variable $\bx_*$ which takes values in $\bSS_*$ such that
$\bx^\iter \to \bx_*$ $\PP$-a.s.~and
$\EE[\norm{\bx^{n} - \bx_*}_{\bWW}] = O\big(
( 1-  \pp_{\min}\min\{1,(2 - \delta)/(2 c_{\bXX,\bGammas^{-1}})\})^{n/2} \big)$.
\end{theorem}
\begin{proof}
Let $\iter \in \N$ and $\bxx \in \bSS_*$. Then, 
 \eqref{eq:20190610a} 
with $\mu_{\Gammas^{-1}}=\sigma_{\Gammas^{-1}}=0$, yields
\begin{align}
\nonumber
\frac{1}{\pp_{\min}}  
\EE[  &\Fs(\bx^{\iter+1}) - \Fs(\bx^{\iter}) \,\vert\, \Fsc_{\iter-1} ]\\
\label{eq:20181213a}
&\leq  \frac{\delta - 2}{2} \norm{ \bx^\iter - \bar{\bx}^{\iter+1}}_{\bGammas^{-1}}^2 
+ \norm{\bx^{\iter} - \bar{\bx}^{\iter+1}}_{\bGammas^{-1}} \norm{\bx^\iter - \bxx}_{\bGammas^{-1}}-  \big( \Fs(\bx^\iter) - \Fs_* \big).
\end{align}
Since \eqref{eq:20181213a} holds for all $\bxx \in \bSS_*$,
using  \ref{eq:EB1}, \eqref{eq:20180925g}, and Proposition~\ref{p:20190313c}, 
we have
\begin{align}
\nonumber
\frac{2}{\pp_{\min}}&
\EE[  \Fs(\bx^{\iter+1}) - \Fs(\bx^\iter) \,\vert\, \Fsc_{\iter-1} ] \\[1ex]
\nonumber&\leq  (\delta - 2) \norm{ \bx^\iter - \bar{\bx}^{\iter+1}}_{\bGammas^{-1}}^2 
+ 2 \norm{\bx^{\iter} - \bar{\bx}^{\iter+1}}_{\bGammas^{-1}} \mathrm{dist}_{\bGammas^{-1}}(\bx^\iter, \bSS_*)
- 2 \big( \Fs(\bx^\iter) - \Fs_* \big)\\[1ex]
\label{eq:20181214g}&\leq \frac{2(2 c_{\bXX,\bGammas^{-1}} + \delta - 2)_+}{\pp_{\min}(2 - \delta)} 
\EE[  \Fs(\bx^{\iter}) - \Fs(\bx^{\iter+1}) \,\vert\, \Fsc_{\iter-1} ]
- 2 ( \Fs(\bx^\iter) - \Fs_*)\quad \PP \text{ a.s.},
\end{align}
which can be equivalently written as
\begin{multline*}
\EE[  \Fs(\bx^{\iter+1}) - \Fs_*  \,\vert\, \Fsc_{\iter-1} ] 
- \big(\Fs(\bx^\iter) - \Fs_* \big) \\[1ex]
\leq \bigg(\frac{2 c_{\bXX,\bGammas^{-1}}}{2 - \delta} -1\bigg)_+\Big(
\Fs(\bx^{\iter}) - \Fs_*  - \EE[  \Fs(\bx^{\iter+1}) - \Fs_* \,\vert\, \Fsc_{\iter-1} ] \Big)
- \pp_{\min} ( \Fs(\bx^\iter) - \Fs_* ).
\end{multline*}
Therefore,
\begin{equation*}
\max\bigg\{1,\frac{2 c_{\bXX,\bGammas^{-1}}}{2 - \delta}\bigg\}
\EE[  \Fs(\bx^{\iter+1}) - \Fs_* \,\vert\, \Fsc_{\iter-1} ] 
\leq \bigg(\max\bigg\{1,\frac{2 c_{\bXX,\bGammas^{-1}}}{2 - \delta}\bigg\} -\pp_{\min}\bigg)
\big( \Fs(\bx^{\iter}) - \Fs_* \big),
\end{equation*}
which gives \eqref{eq:20190131e}.
Now we set $\rho = 1-  \pp_{\min}\min\big\{1,(2 - \delta)/(2 c_{\bXX,\bGammas^{-1}}) \big\}$ and $\theta = \pp_{\min}(2 - \delta)/2$. 
Then, Jensen inequality, \eqref{eq:20190128a}, and \eqref{eq:20190131e} yield
\begin{equation}
\label{eq:20190131g}
\EE[\norm{\bx^n - \bx^{\iter+1}}_{\bWW}] 
\leq \theta^{-1/2}
\sqrt{\EE[\Fs(\bx^\iter)] - \Fs_*}
 \leq \theta^{-1/2}
\rho^{\iter/2}\sqrt{\EE[\Fs(\bx^0)] - \Fs_*}.
\end{equation}
Therefore, since $\rho^{1/2}<1$, we have
$\EE[\sum_{\iter \in \N} \norm{\bx^\iter - \bx^{\iter+1}}_{\bWW}] 
= \sum_{\iter \in \N} \EE[\norm{\bx^\iter - \bx^{\iter+1}}_{\bWW}] <+\infty$.
Hence $\sum_{\iter \in \N} \norm{\bx^\iter - \bx^{\iter+1}}_{\bWW}<+\infty$ $\PP$-a.s.,
which means that $(\bx^\iter)_{n \in \N}$ is a Cauchy sequence $\PP$-a.s.
Now, Theorem~\ref{thm:20171207a}\ref{thm:20171207a_i} yields that there exists a random variable $\bx_*$
with values in $\bSS_*$ such that $\bx^n \rightharpoonup \bx_*$
$\PP$-a.s. Therefore, $\bx^n \to \bx_*$ $\PP$-a.s.
Finally, let $\iter \in \N$. Then, for every $p \in \N$,
\begin{equation*}
\norm{\bx^{\iter} - \bx^{\iter+p}}_{\bWW} \leq 
\sum_{i=0}^{p-1} \norm{\bx^{n + i} - \bx^{\iter + i + 1}}_{\bWW}
\leq \sum_{i=0}^{+\infty} \norm{\bx^{\iter + i} - \bx^{\iter + i + 1}}_\bWW.
\end{equation*}
Hence, letting $p\to +\infty$, we have
$\EE[\norm{\bx^{\iter} - \bx_*}_\bWW] \leq \sum_{i=0}^{+\infty} \EE[\norm{\bx^{\iter + i} - \bx^{\iter + i + 1}}_\bWW]$. Therefore,  it follows from \eqref{eq:20190131g} that
\begin{align*}
\EE[\norm{\bx^{n} - \bx_*}_\bWW] 
& \leq \theta^{-1/2}
\sqrt{\EE[\Fs(\bx^0)] - \Fs_*}
\sum_{i=0}^{+\infty}  \rho^{(n+i)/2}
 = \theta^{-1/2}
\sqrt{\EE[\Fs(\bx^0)] - \Fs_*}
\frac{\rho^{n/2}}{1-\rho^{1/2}}.
\qedhere
\end{align*}
\end{proof}

\begin{remark}\ 
\label{rmk:20181219e}
\begin{enumerate}[(i)]
\item The rate given in Theorem~\ref{thm:20181214a} matches the one given in 
\cite[Theorem~3.2]{Dru18}
for the deterministic case ($\pp_{\min}=1$).
\item In Theorem~\ref{thm:20181214a}, the constant $c_{\bXX,\bGammas^{-1}}$ depends on the stepsizes $\gamma_i$'s which in turn depend on $\delta$ (usually $\gamma_i = \delta/\nu_i$ with $0<\delta<2$).
Therefore, the optimal value of $\delta$ in the rate \eqref{eq:20190131e} can be determined after specifying the expression of $c_{\bXX,\bGammas^{-1}}$. We did so in the special application of Section~\ref{sec:linearsystem}.
\item In \cite[Definition~5.2]{Nec16}, in relation to Algorithm~\ref{algoRCD} but with uniform block 
sampling and assuming \ref{eq:S0} and $R_{\bGammas^{-1}}(\bxx^0):=\sup_{\bxx \in \{\Fs \leq \Fs(\bxx^0)\} } \mathrm{dist}_{\bGammas^{-1}}(\bxx, \bSS_*)<+\infty$, the following error bound condition 
is considered
\begin{equation}
\label{eq:20181114b}
\mathrm{dist}_{\bGammas^{-1}}(\bxx, \bSS_*)
\leq \big(\kappa_{1,\bXX,\bGammas^{-1}} + \kappa_{2,\bXX,\bGammas^{-1}}
\mathrm{dist}_{\bGammas^{-1}}^2(\bxx, \bSS_*)\big)
\norm{\bxx - \prox^{\bGammas^{-1}}_{\bgs} (\bxx - \nabla^{\bGammas^{-1}}\bfs(\bxx) )}_{\bGammas^{-1}},
\end{equation}
for some constants $\kappa_{1,\bXX,\bGammas^{-1}}>0$ and $\kappa_{2,\bXX,\bGammas^{-1}} \geq 0$. 
The authors show several examples in which such condition is satisfied 
with $\bXX= \dom \bgs$ and possibly $\kappa_{2,\bXX,\bGammas^{-1}}>0$. 
The above error bound looks 
more general then \ref{eq:EB1}. However, for the purpose of analyzing Algorithm~\ref{algoRCD}
and under the assumptions considered in \cite{Nec16} this is not the case. Indeed,
in \cite[equation~(3.11)]{Nec16} it was shown 
that the algorithm is descending almost surely\footnote{Alternatively, 
note that \ref{eq:S0} implies \ref{eq:S2} which in turn, 
in view of Proposition~\ref{p:20181219c}, ensures the descending property.}, 
so  $\{\bx^\iter \,\vert\, \iter \in \N\} \subset \{\Fs \leq \Fs(\bxx^0)\}$ $\PP$-a.s.
Therefore, 
since $\sup_{\bxx \in \{\Fs \leq \Fs(\bxx^0)\} } \mathrm{dist}_{\bGammas^{-1}}(\bxx, \bSS_*) 
= R_{\bGammas^{-1}}(\bxx^0)<+\infty$,
if  \eqref{eq:20181114b} holds
on a set $\bXX$ containing $\PP$-a.s.~the set $\{\bx^\iter \,\vert\, \iter \in \N\}$,
then \ref{eq:EB1} holds on $\bXX^\prime := \bXX \cap \{\Fs \leq \Fs(\bxx^0)\} \supset 
\{\bx^\iter \,\vert\, \iter \in \N\}$ $\PP$-a.s.~with 
$c_{\bXX^\prime,\bGammas^{-1}} := \kappa_{1,\bXX,\bGammas^{-1}} + \kappa_{2,\bXX,\bGammas^{-1}} R_{\bGammas^{-1}}(\bxx^0)^2$.
Thus, Theorem~\ref{thm:20181214a}
applies accordingly. Moreover,
\cite[Theorem~5.5]{Nec16} gives the linear rate
\begin{multline*}
\EE[\Fs(\bx^{\iter})] - \Fs_* \leq \bigg(1 - \frac{1}{1+\bar{c}} \bigg)^\iter (\EE[\Fs(\bxx^{0})] - \Fs_*), 
\quad\text{where}\\
\bar{c} = \frac{1}{\pp} \bigg( 2 + \frac{2 c_{\bXX^\prime, \bGammas^{-1}}}{\sqrt{\pp}} + (1 - \pp) 
\frac{c_{\bXX^\prime, \bGammas^{-1}}^2}{\pp} + 2 c_{\bXX^\prime, \bGammas^{-1}} + 1 - \pp \bigg).
\end{multline*}
Then, we have $1+ \bar{c} \geq (3 + 4 c_{\bXX^\prime,\bGammas^{-1}})/\pp$ and hence
\begin{equation*}
\frac{1}{1 + \bar{c}} \leq \frac{\pp}{3 + 4 c_{\bXX^\prime, \bGammas^{-1}}} < \frac{\pp}{\max\{1,2c_{\bXX^\prime, \bGammas^{-1}}\}}.
\end{equation*}
This shows that Theorem~\ref{thm:20181214a}  improves the rate in \cite[Theorem~5.5]{Nec16}. Moreover, the analysis given here,
relying on Proposition~\ref{p:20180927a}, relaxes the assumptions and is significantly simpler.
\item\label{rmk:20181219e_iii}
It follows from \cite[Theorem~6.8]{Nec16} that if $\bfs$ is a quadratic function and 
$\bgs$ is an indicator function of a polyhedral set, then 
\eqref{eq:20181114b} is satisfied on $\dom \bgs$. Therefore, if $\dom \bgs$ is bounded,
then \ref{eq:EB1} holds on $\bXX =\dom \bgs$ with 
$c_{\bXX,\bGammas^{-1}} := \kappa_{1,\bXX,\bGammas^{-1}} + \kappa_{2,\bXX,\bGammas^{-1}} \diam^2_{\bGammas^{-1}}(\dom \bgs)$ and Theorem~\ref{thm:20181214a} can be applied, since
$\{x^\iter \,\vert\, \iter \in \N\} \subset \dom \bgs$.
\item Several works address the convergence of random coordinate descent methods
under error bound conditions. We mention \cite{Kar18} which considers a serial sampling and stepsizes 
 set according to the global Lipschitz constant of $\nabla \bfs$ and \cite{Fer20},
which analyzes restarting procedures for accelerated and parallel coordinate descent methods
using assumptions \ref{eq:S1} and \eqref{eq:20181218b}.
\end{enumerate}
\end{remark}

\begin{remark}
\label{rmk:20200331a}
Often error bound conditions or quadratic growth conditions are satisfied when $\bXX$ is a sublevel set (see Remark~\ref{rmk:20190129e}). So, in such scenarios, 
apart when $\dom \bgs$ is a sublevel set of $\Fs$, 
in order to fulfill the assumption 
$\bXX \supset \{\bx^\iter \,\vert\, \iter \in \N\}$ $\PP$-a.s.~in Theorem~\ref{thm:20181214a},
it is desirable for Algorithm~\ref{algoRCD} to be (a.s.) descending. 
This occurs if condition \ref{eq:S2} holds (Proposition~\ref{p:20181219c}),
whereas, in general, \ref{eq:S1} does not guarantee any such
descending property. 
However, especially when $\eta \ll m$, condition \ref{eq:S2} may be much more restrictive
than \ref{eq:S1}, thus leading to a significant reduction of the stepsizes, which ultimately 
slows down the convergence.
The next result shows that 
Algorithm~\ref{algoRCD} can be slightly modified so to ensure the descending property
while keeping the validity of Theorem~\ref{thm:20181214a}.
\end{remark}

\begin{theorem}
\label{thm20200218a}
 Let \ref{eq:A1}--\ref{eq:A3} be satisfied.
Let $(\nu_{i})_{1 \leq i \leq m} \in \R^m_{++}$ and suppose that 
 \ref{eq:S1} holds. Suppose in addition that
 $\bSS_* \neq \varnothing$, and that
\ref{eq:EB1} holds on the set $\bXX = \{\Fs \leq\Fs(\xx^0)\}$ 
with $c_{\bXX,\bGammas^{-1}}>0$.
Let $(\bx^\iter)_{\iter \in \N}$ be generated by the following variation of Algorithm~\ref{algoRCD}
\begin{equation}
\label{eq:algoPRCD3}
\begin{array}{l}
\text{for}\;n=0,1,\ldots\\
\left\lfloor
\begin{array}{l}
\text{for}\;i=1,\dots, m\\[0.7ex]
\left\lfloor
\begin{array}{l}
\tilde{x}^{\iter+1}_i = 
x^\iter_i + \varepsilon^\iter_i \big(\prox_{\gamma_i \gs_{i}} 
\big(x^\iter_i - \gamma_{i} \nabla_i \bfs (\bx^{\iter})\big) - x^\iter_i\big)
\end{array}
\right.\\[1ex]
\text{if } \Fs(\tilde{x}^{\iter+1}) \leq \Fs(x^\iter)\\[0.8ex]
\quad x^{\iter+1}=\tilde{x}^{\iter+1}\\
\text{else}\\
\quad x^{\iter+1} = x^{\iter}.
\end{array}
\right.
\end{array}
\end{equation}
Then the conclusions of Theorem~\ref{thm:20181214a} still hold.
\end{theorem}
\begin{proof}
It follows from the definition of $\bx^{n+1}$ that
$\Fs(\bx^{n+1}) \leq \Fs(\bx^n)$. Therefore 
$\bXX \supset \{\bx^\iter \,\vert\, \iter \in \N\}$.
Recalling \eqref{eq:20171206a} and \eqref{eq:20181114c},
algorithm \eqref{eq:algoPRCD3} can be alternatively written as
\begin{equation}
\label{eq:algoPRCD4}
\begin{array}{l}
\text{for}\;n=0,1,\ldots\\
\left\lfloor
\begin{array}{l}
\text{for}\;i=1,\dots, m\\[0.7ex]
\left\lfloor
\begin{array}{l}
\bar{\bx}^{\iter+1}_i = \prox_{\gamma_{i} \gs_i} ( x_i^{\iter}
- \gamma_i \nabla_i \bfs(\bx^\iter) ),\\
\end{array}
\right.\\[1ex]
\tilde{\bx}^{\iter+1} = \bx^{\iter} + \bvarepsilon^{\iter} \odot (\bar{\bx}^{\iter+1} - \bx^{\iter})\\
x^{\iter+1} = \tilde{\bx}^{\iter+1} \mathbf{1}_{\{\Fs(\tilde{x}^{\iter+1}) \leq \Fs(x^\iter)\}} 
+ \bx^{\iter}\mathbf{1}_{\{\Fs(\tilde{x}^{\iter+1}) > \Fs(x^\iter)\}}
\end{array}
\right.
\end{array}
\end{equation}
and we have $\Fs(x^{\iter+1}) \leq \Fs(\tilde{x}^{\iter+1})$.
Then we can essentially repeat the argument in the proof of Theorem~\ref{thm:20181214a}.
First we note that \eqref{eq:20190610a} and hence
\eqref{eq:20181213a} holds with $x^{\iter+1}$ replaced by $\tilde{x}^{\iter+1}$.
This follows from the definition of $\tilde{\bx}^{\iter+1}$.
Moreover, also \eqref{eq:20181124a} holds with $x^{\iter+1}$ replaced by $\tilde{x}^{\iter+1}$
and hence we derive (with $\bxx = \bx^{\iter}$ and $\sigma_{\Gammas^{-1}}=0$) that
\begin{equation}
\label{20200218b}
(2 -\delta) \norm{\bar{\bx}^{\iter+1} - \bx^{\iter}}_{\bGammas^{-1}}^2 \leq
\frac{2}{\pp_{\min}}  
\EE[  \Fs(\bx^\iter) - \Fs(\tilde{\bx}^{\iter+1}) \,\vert\, \Fsc_{\iter-1} ]. 
\end{equation}
Then, we have
\begin{align}
\nonumber
\frac{2}{\pp_{\min}}&
\EE[  \Fs(\tilde{\bx}^{\iter+1}) - \Fs(\bx^\iter) \,\vert\, \Fsc_{\iter-1} ] \\[1ex]
\nonumber&\leq  (\delta - 2) \norm{ \bx^\iter - \bar{\bx}^{\iter+1}}_{\bGammas^{-1}}^2 
+ 2 \norm{\bx^{\iter} - \bar{\bx}^{\iter+1}}_{\bGammas^{-1}} \mathrm{dist}_{\bGammas^{-1}}(\bx^\iter, \bSS_*)
- 2 \big( \Fs(\bx^\iter) - \Fs_* \big)\\[1ex]
\label{eq:20200218a}&\leq \frac{2(2 c_{\bXX,\bGammas^{-1}} + \delta - 2)_+}{\pp_{\min}(2 - \delta)} 
\EE[  \Fs(\bx^{\iter}) - \Fs(\tilde{\bx}^{\iter+1}) \,\vert\, \Fsc_{\iter-1} ]
- 2 ( \Fs(\bx^\iter) - \Fs_*)\quad \PP \text{ a.s.}
\end{align}
and hence, since $\Fs(\bx^{\iter+1}) \leq \Fs(\tilde{x}^{\iter+1})$,  \eqref{eq:20181214g}
still holds (for the new definition of $\bx^{\iter+1}$). Thus, 
 \eqref{eq:20190131e} follows. As for the second part of the statement,
we note that, since $\Fs(\bx^{\iter+1}) \leq \Fs(\tilde{x}^{\iter+1})$, by \eqref{20200218b}, we have $(2 -\delta) \norm{\bar{\bx}^{\iter+1} - \bx^{\iter}}_{\bGammas^{-1}}^2 \leq
(2/\pp_{\min}) \EE[  \Fs(\bx^\iter) - \Fs(\bx^{\iter+1}) \,\vert\, \Fsc_{\iter-1} ]$.
Moreover, it follows from the definitions of 
$\bx^{\iter+1}$ and $\tilde{\bx}^{\iter+1}$ in algorithm~\eqref{eq:algoPRCD4} that
\begin{equation*}
\bx^{\iter+1} - \bx^\iter = (\tilde{\bx}^{\iter+1} - \bx^\iter) \mathbf{1}_{\{\Fs(\tilde{\bx}^{\iter+1}) \leq \Fs(\bx^{\iter})\}} 
\end{equation*}
and hence, by Proposition~\ref{p:20190313c}, we have
\begin{equation*}
\EE[\norm{\bx^{\iter+1} - \bx^\iter}_\bWW^2 \,\vert\, \Fsc_{\iter-1} ] \leq
\EE[\norm{\tilde{\bx}^{\iter+1} - \bx^\iter}_\bWW^2 \,\vert\, \Fsc_{\iter-1} ]
= \norm{\bar{\bx}^{\iter+1} - \bx^\iter}^2_{\bGammas^{-1}}.
\end{equation*}
In the end \eqref{eq:20190128a} with $\sigma_{\bGammas^{-1}}=0$ still holds
and the proof can continue as in that of Theorem~\ref{thm:20181214a}.
\end{proof}

\vspace{-2ex}
\section{Applications}
\label{sec:app}

In this section we show some relevant optimization problems for which 
the theoretical analysis of Algorithm~\ref{algoRCD} 
can be particularly useful.

\label{sec:appl}

\subsection{The Lasso problem}
\label{subsec:lasso}
Many papers study the convergence of coordinate descent methods for the Lasso problem and recent works prove linear convergence (see e.g., \cite{Kar18,Nec16,Nec19}). In \cite{Kar18} a random serial update of blocks is considered while in \cite{Nec19} the general framework of feasible descend methods is analyzed which include
(nonrandom) cyclic coordinate methods. In the following we discuss our contribution comparing with \cite{Nec16}.
Let $\bAs \in \R^{p\times m}$ and $\bbs \in \R^p$.
We consider the problem
\begin{equation}
\label{eq:20181015a}
\min_{\bxx \in \R^m} \frac 1 2 \norm{\bAs \bxx - \bbs}_2^2 + \lambda \norm{\bxx}_1 \quad(\lambda>0).
\end{equation}
We denote by
$\bas^i$ and $\bas_k$ the $i$-th column and $k$-th row of $\bAs$ respectively.
Since
\begin{equation*}
\frac 1 2 \norm{\bAs \bxx - \bbs}^2 
= \frac1 2 \sum_{k=1}^p ( \scalarp{\bas_k, \bxx} - \bs_k)^2
= \frac 1 2 \sum_{k=1}^p \Big(\sum_{i=1}^m \as_k^i \xx_i - \bs_k \Big)^2,
\end{equation*}
\ref{eq:B0} holds and $\eta = \max_{1 \leq k \leq p} \mathrm{card}(\mathrm{spt}(\bas_k))$.
Moreover, since
 $\nabla_i \bfs(\bxx) = \scalarp{\bas^i, \bAs \bxx - \bbs}$ 
and recalling Remark~\ref{rmk:20190128e}\ref{rmk:20190128e_i},
 conditions \ref{eq:B1} and  \ref{eq:B2} are satisfied with  $L_i = \norm{\bas^i}^2$
 and $L^{(k)} = \norm{\bas_k}^2$ respectively.
Then, Algorithm~\ref{algoRCD} (assuming that each block is made of 
one coordinate only)
writes as
\begin{equation}
\label{eq:20170915a}
\begin{aligned}
\bx^{\iter+1} 
&= \bx^\iter 
+  \sum_{i=1}^m \varepsilon_i^\iter \big[ \mathrm{soft}_{\gamma_{i} \lambda} 
\big( x_{i}^\iter - \gamma_{i} {\bas^i}^\top(\bAs \bx^\iter - \bbs) \big) - x_{i}^\iter \big]\bee_{i},
\end{aligned}
\end{equation}
where the \emph{soft thresholding operator} $\mathrm{soft}_{\gamma_{i} \lambda}$
 is defined as $\mathrm{soft}_{\gamma_{i} \lambda} (t) 
 = \mathrm{sign}(t)\max\{0, \abs{t}- \gamma_{i} \lambda\}$ and $(\bee_{i})_{1 \leq i \leq m}$ 
 is the canonical basis of $\R^m$.
Now, define $u^\iter = \bAs \bx^\iter - \bbs$. Then multiplying the equation in \eqref{eq:20170915a}
by $\bAs$ and subtracting $\bbs$ by both terms, the algorithm is equivalently written as
\begin{equation}
\label{eq:algoLasso}
\begin{array}{l}
\left\lfloor
\begin{array}{l}
\text{for}\;i \in \mathrm{spt}(\varepsilon^\iter)\\
\left\lfloor
\begin{array}{l}
\xi_i = \mathrm{soft}_{\gamma_{i} \lambda} 
\big( x_{i}^\iter - \gamma_{i} {\bas^i}^\top u^\iter \big) - x_{i}^\iter,
\end{array}
\right.\\[1ex]
\bx^{\iter+1} = \bx^{\iter}  + \sum_{i \in \mathrm{spt}(\varepsilon^\iter)}
\xi_i \bee_{i}\\[0.5ex]
u^{\iter+1}  = u^{\iter}  + \sum_{i \in \mathrm{spt}(\varepsilon^\iter)}
\xi_i \bas^i,
\end{array}
\right.
\end{array}
\end{equation}
showing that each iteration costs $O(p\tau_{\max})$ multiplications,
where $\tau_{\max}$ is the maximum number of block updates per iteration.
We now address the determination of the smoothness parameters $(\nu_i)_{1 \leq i \leq m}$.
We first give a general rule which holds for any arbitrary sampling.
Recalling  Theorem~\ref{thm:stepsizes}\ref{thm:stepsizes_iii} and
Remark~\ref{rmk:20190128e}\ref{rmk:20190128e_i} and noting that
$\{k \vert i \in I_k\} = \{k \vert i \in \mathrm{spt}(\bas_k)\} = \mathrm{spt}(\bas^i)$, if we set 
$\nu_i = \sum_{k\in \mathrm{spt}(\bas^i)}
\norm{\bas_k}^2$, then \ref{eq:S0} and hence \ref{eq:S2} holds.
This choice was considered in \cite{Nec16}.
Moreover, according to the discussion at the beginning of Section~\ref{subsec:beta}
other options for satisfying \ref{eq:S2} are
$\nu_i=\min\{\eta,\tau_{\max}\} \norm{\bas^i}^2$ or
 $\nu_i = \sum_{k \in \mathrm{spt}(\bas^i)} \min\{\mathrm{card}(\mathrm{spt}(\bas_k)), \tau_{\max}\} (\bas_k^i)^2$. This latter choice is better than the second one and,
if we assume that the nonzero entries of $A$ are about of the same magnitude,
it is also better than the first one.
Next, we face the special case of the $\tau$-nice sampling
which allows to reduce the $\nu_i$'s while satisfying \ref{eq:S1}.
Recalling the corresponding discussion in Section~\ref{subsec:beta},
we have the following alternatives: (1) set, for every $i \in [m]$, 
$\nu_i = (1 + (\tau-1)(\eta-1)/(m-1)) \norm{\bas^i}^2$;
(2) set for every $i \in [m]$, $\nu_i = 
\sum_{k\in \mathrm{spt}(\bas^i)} (1 + (\tau-1)(\mathrm{card}(\mathrm{spt}(\bas_k))-1)/(m-1)) (\bas_k^i)^2$.
Finally, we make few remarks on the convergence properties of algorithm~\eqref{eq:algoLasso}.
Since the objective function in \eqref{eq:20181015a} satisfies a quadratic growth condition on its sublevel sets \cite[Example~3.8]{Gar17}, then Remark~\ref{rmk:20190129e}, Remark~\ref{rmk:20200331a},
and Theorem~\ref{thm:20181214a}, yield linear convergence of algorithm
\eqref{eq:algoLasso} provided that \ref{eq:S2} holds.
Whereas Theorem~\ref{thm20200218a} 
ensures that if we modify algorithm \eqref{eq:algoLasso}
so that we accept the next iterate $\bx^{\iter+1}$ only if $\norm{u^{\iter+1}}^2 
+ 2\lambda \norm{\bx^{\iter+1}}_1\leq \norm{u^\iter}^2 + 2\lambda \norm{\bx^{\iter}}_1$, 
then the resulting algorithm
converges linearly under condition \ref{eq:S1}. 
If the violation of the monotonicity condition above
occurs few times along all the iterations, this modification does not increase much 
the computational cost of the algorithm (see also Section~\ref{sec:experiments}).
We stress that both
Theorem~\ref{thm:20181214a} and Theorem~\ref{thm20200218a}
ensure also almost sure and linear convergence in mean of the iterates 
of \eqref{eq:20170915a}. 
This latter result is new and is especially relevant in this context, since
 the iterates carry sparsity information. 

\subsection{Computing the minimal norm solution of a linear system}
\label{sec:linearsystem}
Let $\bAs \in \R^{m\times p}$ and $\bbs \in R(\bAs)$ (the range of $\bAs$).
Let us consider the problem
\begin{equation}
\label{eq:20170920a}
\minimize{\substack{\bxx \in \R^p \\ \bAs \bxx = \bbs}}{\frac 1 2 \norm{\bxx}^2}.
\end{equation}
Here, we denote by $\bas_i \in \R^p$ and $\bas^k \in \R^m$
the $i$-th row and the $k$-th column of $\bAs$.
The dual problem is
\begin{equation}
\label{eq:20190118h}
\minimize{\buu \in \R^m}{\frac 1 2 \norm{ \bAs^\top\buu}^2 - \scalarp{\buu,\bbs}}:=\mathcal{D}(\buu),
\end{equation}
which is a smooth convex optimization problem.
Moreover, if $\bxx_*$ is the solution of \eqref{eq:20170920a} and 
$\bxx = \bAs^\top \buu$ (the primal-dual relationship), then we have
\begin{equation}
\label{eq:20190118j}
\frac 1 2 \norm{\bxx - \bxx_*}^2 
\leq \mathcal{D}(\buu) - \inf \mathcal{D}.
\end{equation}
Then, the dual problem is clearly of the form \eqref{eq:20190110b}, 
with $\bgs=0$,
and \ref{eq:B0} and \ref{eq:B1}
are satisfied, assuming that each block is made of 
one coordinate only, with $L_i = \norm{\bas_i}^2$ and $\eta = \max_{1 \leq k \leq p} \mathrm{card}(\mathrm{spt}(\bas^k))$. So, Algorithm~\ref{algoRCD} applied to \eqref{eq:20190118h}, turns into
\begin{equation*}
\bu^{\iter+1} = \bu^\iter - \sum_{i=1}^m \varepsilon_i^\iter\gamma_{i} 
\big(\scalarp{\bas_{i},
\bAs^\top \bu_\iter} - \bs_{i}\big) 
\bee_{i}.
\end{equation*}
Now, setting $\bx^\iter = \bAs^\top \bu^\iter$ and multiplying the above equality by $\bAs^\top$, we have
\begin{equation}
\label{eq:20190129f}
\bx^{\iter+1} = \bx^\iter - \sum_{i=1}^m 
\varepsilon_i^\iter\gamma_{i} \big(\scalarp{\bas_{i}, \bx^\iter} - \bs_{i}\big) 
\bas_{i}.
\end{equation}
Since, $\bbs \in R(\bAs)$, it is easy to see, through a singular value decomposition of $\bAs$, that, for every $\buu \in \R^m$, 
$ \sigma^2_{\min}(\bAs)\,\mathrm{dist}(\buu, \argmin \mathcal{D}) 
\leq \norm{\nabla \mathcal{D}(\buu)}$ (where $\sigma_{\min}(\bAs)$ is the minimum  singular value of $\bAs$) \cite[Example 3.6]{Gar17}.
So, in view of Remark~\ref{rmk:20190603a}\ref{rmk:20190129e_00}-\ref{rmk:20190129e_0}, 
\ref{eq:EB1} is satisfied on the entire space 
with constant $c_{\R^m, \bGammas^{-1}} = (\gamma_{\min}
\sigma^2_{\min}(\bAs))^{-1}$. Therefore, if, for every $i \in [m]$, $\gamma_i =\delta/(\beta_{1,i}\norm{\bas_i}^2)$ with $0<\delta<2$, Theorem~\ref{thm:20181214a} and \eqref{eq:20190118j}
ensure the linear convergence of the iterates $\bx^n$'s towards the solution of \eqref{eq:20170920a} with rate 
$
\big(1 - \pp_{\min}\min\big\{1, \gamma_{\min}\sigma^2_{\min}(\bAs)(2-\delta)/2\big\} \big)^{1/2}
$.
We remark that \eqref{eq:20190129f} is nothing but a stochastic 
gradient descent algorithm on the problem
\begin{equation*}
\minimize{\bxx \in \R^p}{\frac 1 2 \norm{\bAs \bxx - \bbs}^2}=
\frac 1 2 \sum_{i=1}^m (\scalarp{\bas_i, \bxx} - \bs_i)^2.
\end{equation*}
Since $\norm{\bAs \bxx - \bbs}^2 \leq \norm{\bAs}^2 \norm{\bxx - \bxx_*}^2$,
we have then showed the linear convergence rate
\begin{equation*}
\frac1 2 \norm{\bAs \bx^n - \bbs}^2 - \frac 1 2 \norm{\bAs \bxx_* - \bbs}^2 = O\bigg(\sigma^2_{\max}(\bAs)
\bigg(1 - \pp_{\min}\min\bigg\{1, \frac{\sigma^2_{\min}(\bAs)\delta(2-\delta)}{2 \max_{i} \beta_{1,i} \norm{\bas_i}^2}\bigg\} \bigg)^n \bigg),
\end{equation*}
of the stochastic gradient descent with arbitrary 
and possibly variable batch size for least squares problems.
This also shows that the best rate is achieved for $\delta=1$.
We finally note that in the serial case, that is, if for every $\iter \in \N$
$\mathrm{spt}(\bvarepsilon^\iter) = \{ i_\iter \}$, multiplying equation \eqref{eq:20190129f} by $\bas_{i_\iter}^\top$, 
we have
\begin{equation*}
\scalarp{\bas_{i_\iter}, \bx^{\iter+1}} = 
\scalarp{\bas_{i_\iter}, \bx^\iter} - \gamma_{i_\iter} 
\big(\scalarp{\bas_{i_\iter}, \bx^\iter} - \bs_{i_\iter}\big) 
\norm{\bas_{i_\iter}}^2.
\end{equation*}
Therefore, since in this case $\beta_{1,i}= \beta_2=1$, we can chose the stepsizes such that $\gamma_{i}\norm{\bas_{i}}^2 = 1$ (so that $\delta=1$) and hence
 $\bx^{\iter+1}$ is a solution of the $i_\iter$-th equation of the linear system 
 $\bAs \bxx = \bbs$.
Moreover, $\bx^{\iter+1}$ is the projection of $\bx^\iter$ onto the affine space defined 
by the equation $\bas_{i_\iter} \bxx = \bs_{i_\iter}$ \cite{Wri15}.
Thus, this method is nothing but the \emph{randomized Kaczmarz method} \cite{Stro09}  and we proved linear convergence for general probabilities $\pp_i$'s, although the 
constants we derive are not optimal (see \cite{Lev10,Stro09,Wri15}).

\subsection{Regularized empirical risk minimization}

Let $\HH$ be a separable real Hilbert space.
Regularized empirical risk estimation solves the following optimization problem
\begin{equation}
\label{eq:20170217q}
\minimize{\ww \in \HH}{\frac{1}{\lambda m} 
\sum_{i=1}^m \ell(\yy_i, \scalarp{\ww,\xx_i}) + \frac 1 2 \norm{\ww}^2}
:= \mathcal{P}(\ww),
\end{equation}
where $(\xx_i, \yy_i)_{1 \leq i \leq m}$ is the training set (input-output pairs), 
 $\ell\colon \YC\times\R \to \R_+$, $\YC \subset \R$, is the \emph{loss} function, which is convex in the second variable, 
 and $\lambda>0$
is a regularization parameter. 
The dual problem of \eqref{eq:20170217q} is
\begin{equation}
\label{eq:gendual}
\minimize{\buu \in \R^m}{\frac 1 2 \buu^\top \mathsf{K} \buu
+ \frac{1}{\lambda m} \sum_{i=1}^m \ell^* (\yy_i, - \uu_i \lambda m)}
:=\mathcal{D}(\buu),
\end{equation}
where $\ell^*(\yy_i, \cdot)$ is the Fenchel conjugate of $\ell(\yy_i, \cdot)$ and
$\mathsf{K} = \bXX \bXX^\top \in \R^{m\times m}$  is the Gram matrix
of $(\xx_i)_{1 \leq i \leq m}$.
Moreover, the solutions $(\bar{\ww}, \bar{\buu})$ 
of the primal and dual problems are characterized by the following KKT conditions
\begin{equation}
\label{eq:kkt}
\begin{cases}
\displaystyle\bar{\ww} = \bXX^\top \bar{\buu} 
= \sum\limits_{i=1}^m \bar{\uu}_i \xx_i,\\[2.5ex]
\forall\, i \in \{1,\dots, m\}\ \ 
- \bar{\uu}_i m \lambda \in \partial \ell(\yy_i, \scalarp{\xx_i, \bar{\ww}}),
\end{cases}
\end{equation}
where $\partial \ell(\yy_i,\cdot)$ is the subdifferential of $\ell(\yy_i, \cdot)$. 
Note also that the first of \eqref{eq:kkt} gives the link between the dual and the primal variable
and, if $\ww = \bXX^\top \buu$, then it holds $(1/2) \norm{\ww - \bar{\ww}}^2 \leq \mathcal{D}(\buu) - \inf \mathcal{D}$.
Now, the dual problem \eqref{eq:gendual} is of the form \eqref{eq:20190110b}
and hence Algorithm~\ref{algoRCD} can be applied. The following examples give implementation details for two specific losses.

\begin{example}[Ridge regression]\ 
The least squares loss is  $\ell(s,t) = (1/2) \abs{s - t}^2$.
Then $\ell^*(s,r) = (1/2) r^2 + r s$ and, 
in this case, \eqref{eq:gendual} reduces to 
\begin{equation*}
\minimize{\buu \in \R^m}
{\frac 1 2 \buu^\top ( \mathsf{K} + \lambda m \Id )\buu  
- \byy^\top \buu} :=\mathcal{D}(\buu)
\end{equation*}
which  is strongly convex with modulus $\lambda m$ and has  solution $\bar{\buu} = (\mathsf{K} + m \lambda \Id)^{-1} \byy$.
Since $\mathcal{D}$ is smooth and 
$\nabla_i \mathcal{D} (\buu) = \bee_i^\top (\mathsf{K} + \lambda m\Id) \buu - \yy_i$,  conditions \ref{eq:B0} and \ref{eq:B1} hold with $L_i = \norm{\xx_i}^2 + \lambda m = \mathsf{K}_{i,i} + \lambda m$ and
 Algorithm~\ref{algoRCD} (with $\bgs=0$) becomes
 \begin{equation}
\label{eq:20190123b}\bu^{\iter+1} 
  = \bu^\iter - \sum_{i=1}^m \varepsilon^\iter_i \gamma_i 
\big( \bee_i^\top\mathsf{K}\bu^\iter + \lambda m u^\iter_i 
- \yy_i \big) \bee_i.
\end{equation}
Moreover, multiplying 
\eqref{eq:20190123b} by $\bXX^\top$, defining 
$w^\iter = \bXX^\top \bu^\iter$,
and recalling that 
$\mathsf{K} = \bXX \bXX^\top$, we have
\begin{align}
\label{eq:20190131a}
w^{\iter+1} 
&= w^\iter - \sum_{i=1}^m \varepsilon^\iter_i \gamma_i
\big(\scalarp{w^\iter, \xx_i}  - \yy_i\big) \xx_i
-  \lambda m \sum_{i=1}^m \varepsilon^\iter_i \gamma_i  \bu^\iter_i \xx_i.
\end{align}
Note that, since the dual problem is strongly convex with modulus $\lambda m$,
then it follows from Theorem~\ref{p:20181130a}, Remark~\ref{rmk:20190130a}, and Theorem~\ref{thm:dualgap}\ref{thm:dualgap_i} 
that, setting, for every $i \in [m]$,
$\nu_i \geq \beta_{1,i} (\mathsf{K}_{i,i} + \lambda m)$
and $\gamma_i = 1/\nu_i$, we have
\begin{equation*}
\EE[\mathcal{P}(w^\iter)] - \inf \mathcal{P} \leq 
\bigg(1 + \frac{\norm{\mathsf{K}}}{\lambda m } \bigg)
\bigg( 1 - \pp_{\min} \frac{2 \lambda m}{\nu_{\max} + \lambda m} \bigg)^\iter
\mathrm{const.}
\end{equation*}
Now, we compare algorithm \eqref{eq:20190131a} with the stochastic gradient 
descent on problem \eqref{eq:20170217q}.
Assume that $\PP(\sum_{i=1}^m \varepsilon^\iter_i = \tau) = 1$ for some 
$\tau \in [m]$. Then, 
$\pp_{\min} = \tau/m$ and
we can take 
$\zeta_i \leq (\mathsf{K}_{i,i} + \lambda m)^{-1}$, 
and set
$\nu_i = \tau/\zeta_i$ and $\gamma_i = 1/\nu_i$, so that
 algorithm \eqref{eq:20190131a} turns into
\begin{equation}
\label{eq:20190123d}
w^{\iter+1} =w^\iter - \sum_{i=1}^m \varepsilon^\iter_i \frac{\zeta_i}{\tau}
\big(\scalarp{w^\iter, \xx_i}  - \yy_i\big) \xx_i
-  \lambda m \sum_{i=1}^m \varepsilon^\iter_i  \frac{\zeta_i}{\tau} u^\iter_i \xx_i.
\end{equation}
If we apply stochastic gradient descent with batch size $\tau \in [m]$
and stepsize $\zeta>0$ directly on the primal problem \eqref{eq:20170217q} (multiplied by $\lambda m$), and recalling that $w^\iter = \sum_{i=1}^m u^\iter_i \xx_i$, 
 we have
\begin{align}
 w^{\iter + 1}  
\label{eq:20190123e}= w^\iter - \frac{\zeta}{\tau} \sum_{i=1}^m \varepsilon^\iter_i
( \scalarp{w^\iter, \xx_i} - \yy_i ) \xx_i -  \lambda m
\frac{\zeta}{m} \sum_{i=1}^m u_i^\iter \xx_i.
\end{align}
Then, comparing \eqref{eq:20190123d} and \eqref{eq:20190123e} we see that,
provided that $\zeta_i = \zeta$ for every $i \in [m]$,
they only differ for the replacement $(1/m) \sum_{i=1}^m u^\iter_i \xx_i\ \leftrightarrow\ 
(1/\tau) \sum_{i=1}^m \varepsilon^\iter_i u^\iter_i \xx_i$.
We stress that the stepsize $\zeta$ in the stochastic gradient descent
algorithm \eqref{eq:20190123e} is normally
set according to the spectral norm of 
$\mathsf{K} + \lambda m \Id$, which may be difficult to compute. 
On the contrary in algorithm \eqref{eq:20190123d} the stepsizes $\zeta_i$'s
are simply set as $\zeta_i \leq 1/(\mathsf{K}_{i,i} + \lambda m)$,
so they allow possibly much longer steps and also do not require any SVD computation.
\end{example}

\begin{example}[Support vector machines]
The \emph{hinge loss} is 
$\ell(s,t) = (1-s t)_+$. Then we have $\ell^*(s,r) = r + \iota_{[0,1]}(s r)$
and the dual problem \eqref{eq:gendual} 
is 
\begin{equation}
\label{eq:20190131b}
\minimize{\buu \in \R^m} 
{\frac 1 2 \buu^\top \mathsf{K} \buu
 - \byy^\top\buu + \iota_{(\lambda m)^{-1}[0,1]^m}(\byy \odot \buu)}.
\end{equation}
Then Algorithm~\ref{algoRCD} on the dual turns into a parallel random block-coordinate
projected gradient descent method.
Moreover, it follows from Remark~\ref{rmk:20181219e}\ref{rmk:20181219e_iii}
that the objective in \eqref{eq:20190131b}
satisfies \ref{eq:EB1} on its domain.
Therefore, it follows from
Theorem~\ref{thm:20181214a}, Theorem~\ref{thm:stepsizes}\ref{thm:stepsizes_i}, and
Theorem~\ref{thm:dualgap}\ref{thm:dualgap_ii} 
that
$\EE[\mathcal{P}(w^\iter)]- \inf \mathcal{P}$ converges linearly to zero,
provided that, for all $i \in [m]$, $\nu_i \geq \beta_{1,i} \mathsf{K}_{ii}$ and $\gamma_i < 2/\nu_i$.
\end{example}

\section{Numerical Experiments}
\label{sec:experiments}
In this section we consider a Lasso problem, 
that is,
\begin{equation}
\label{eq:20200312a}
\min_{\bxx \in \R^m} \frac 1 2 \norm{\bAs \bxx - \bbs}_2^2 + \lambda \norm{\bxx}_1 \quad(\lambda>0),
\end{equation}
where $\bAs \in \R^{p\times m}$ is generated with random entries uniformly distributed in $[-1,1]$ so that
each row is sparse and $\bbs = \bAs \bar{\bxx} + 0.06\cdot \alpha$ with $\alpha \sim N(0,\Id_p)$ and $\bar{x}$
a sparse vector in $\R^m$.
We implement Algorithm~\ref{algoRCD} with $\gamma_i = \delta/\nu_i$ ($0<\delta<2$) and a $\tau$-nice uniform sampling, as described
in Section~\ref{subsec:lasso}.
We present two experiments. The first compares conditions
\ref{eq:S1} and \ref{eq:S2} for the determination of the stepsizes $\gamma_i$. The second one investigates the role played by $\delta$.
In all the experiments we empirically checked that the algorithm is essentially descending
in the sense that during the iterations there are very few violations of the descent property and with low magnitude. 
So, since the objective function in \eqref{eq:20200312a} satisfies \ref{eq:EB1} on the sublevel sets,
in virtue of Theorem~\ref{thm20200218a}, linear convergence holds.

%------------------------------------------------------------------------------
\subsection*{Condition~\ref{eq:S1} vs \ref{eq:S2} and the effectiveness of the parallel strategy.}
We compare the conditions \ref{eq:S1} and \ref{eq:S2} for the stepsizes selection
and we checked the critical role played by \ref{eq:S1}
for the effectiveness of the parallel strategy on problems with sparse structure.
Here we set $\delta=1$. In Figure~\ref{lasso_stepsizes}, Algo1 uses smoothness parameters specifically 
designed for the $\tau$-nice sampling, that is, $\nu_i=\beta_1 \norm{\bas^i}^2$ with
$\beta_1 = 1 + (\tau-1)(\eta-1)/(m-1)$ (making \ref{eq:S1} satisfied), while Algo2 uses a more conservative choice for the smoothness parameters which is valid for any sampling updating a maximum of $\tau$ blocks per iteration, that is $\nu_i = \beta_2 \norm{\bas^i}^2$ with $\beta_2 = \min\{\tau, \eta\}$ (making also \ref{eq:S2} satisfied). 
\begin{figure}[t]
\begin{center}
\vspace{-2ex}
\includegraphics[width=0.5\textwidth]{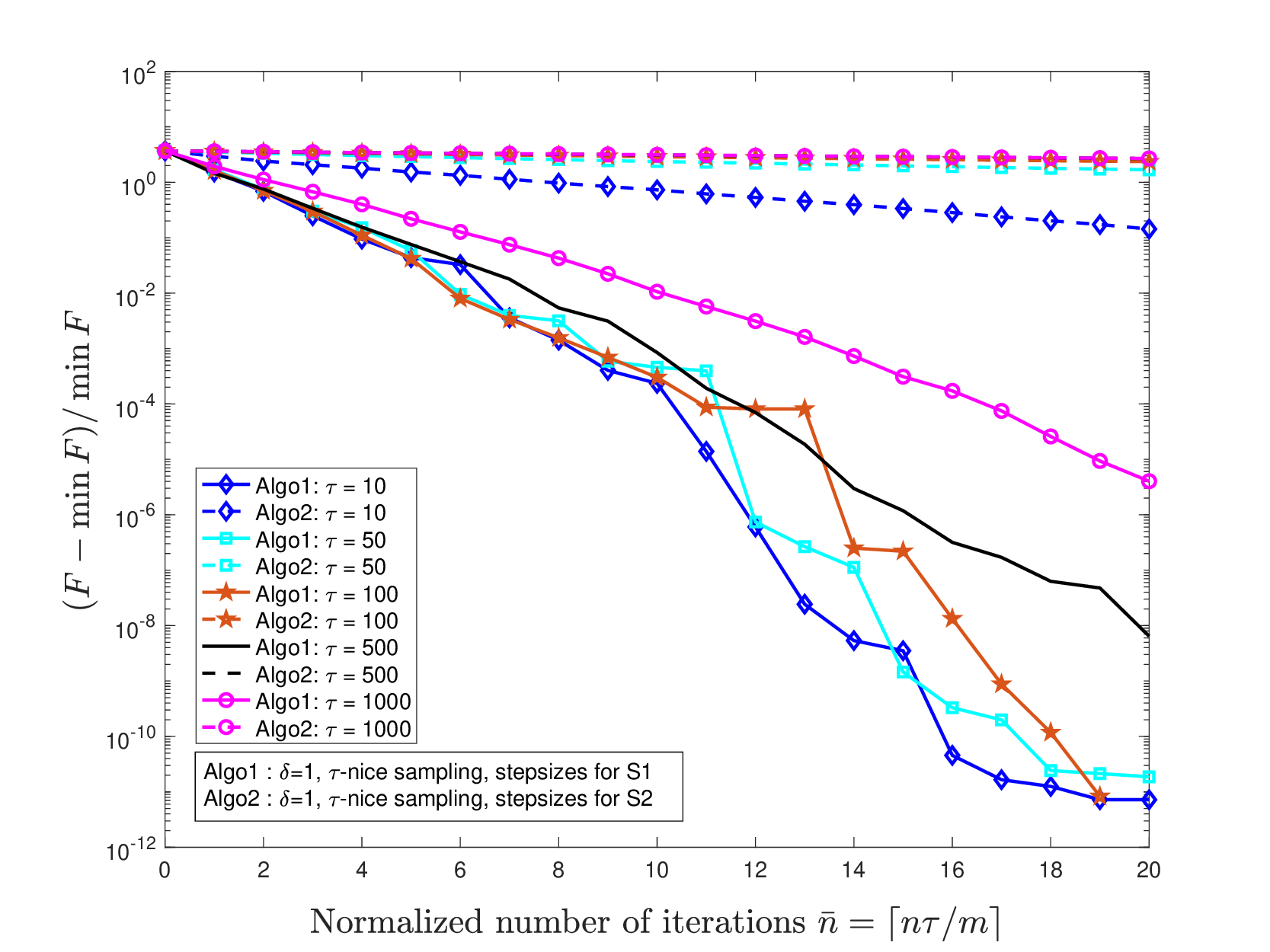}
\hspace{-3ex}
\includegraphics[width=0.5\textwidth]{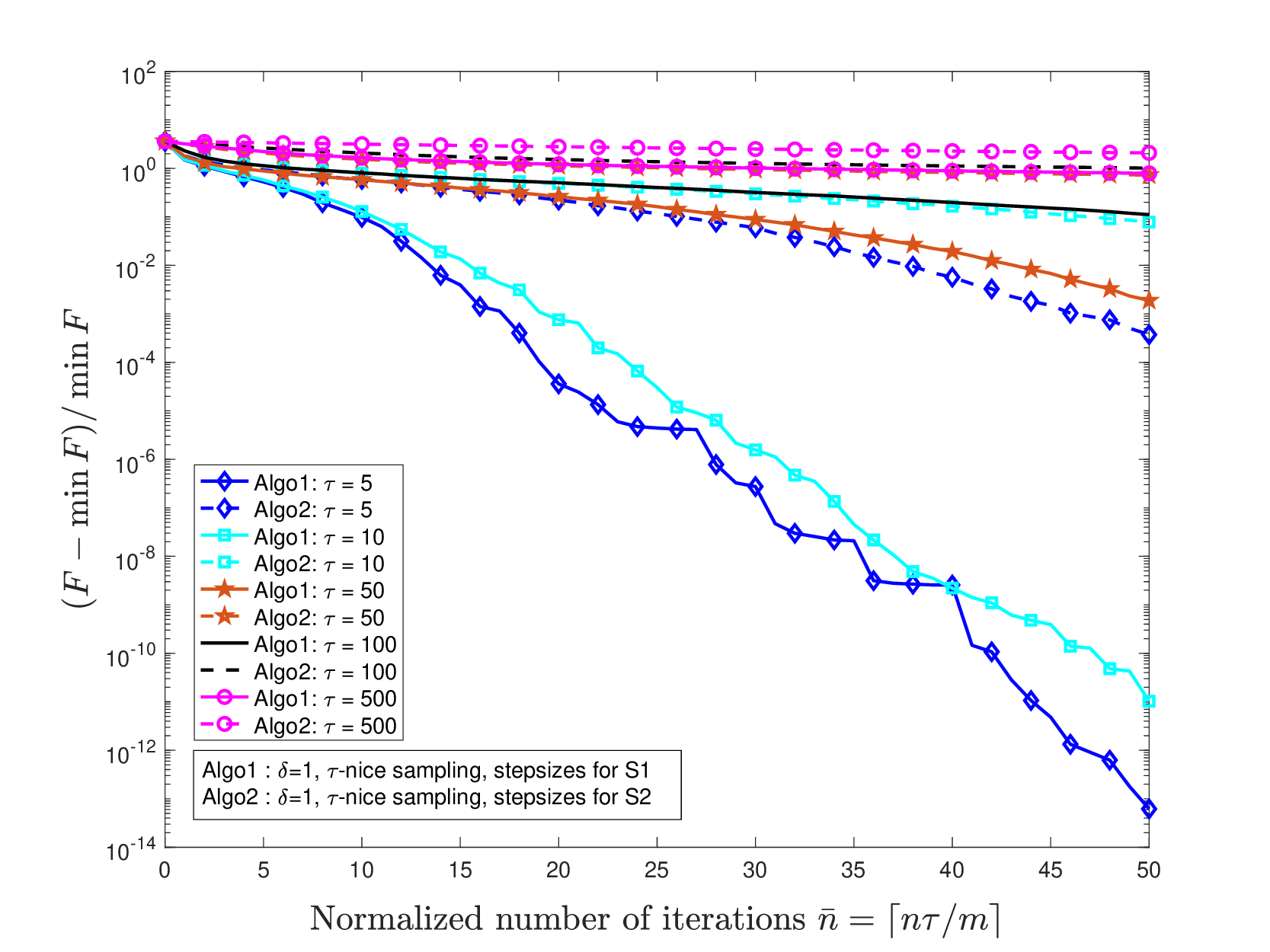}
\end{center}
\vspace{-2em}
\caption{\small Comparison between \ref{eq:S1} and \ref{eq:S2} for the the stepsizes selection in a Lasso problem.
Left: $5\cdot 10^4$ equations in $10^5$ unknowns; degree of partial separability $\eta=148$.
Right: $10^3$ equations in $5\cdot 10^3$ unknowns;
degree of partial separability $\eta = 563$.}
\label{lasso_stepsizes}
\end{figure}
\begin{figure}[h!]
\begin{center}
\vspace{-1ex}
\includegraphics[width=0.5\textwidth]{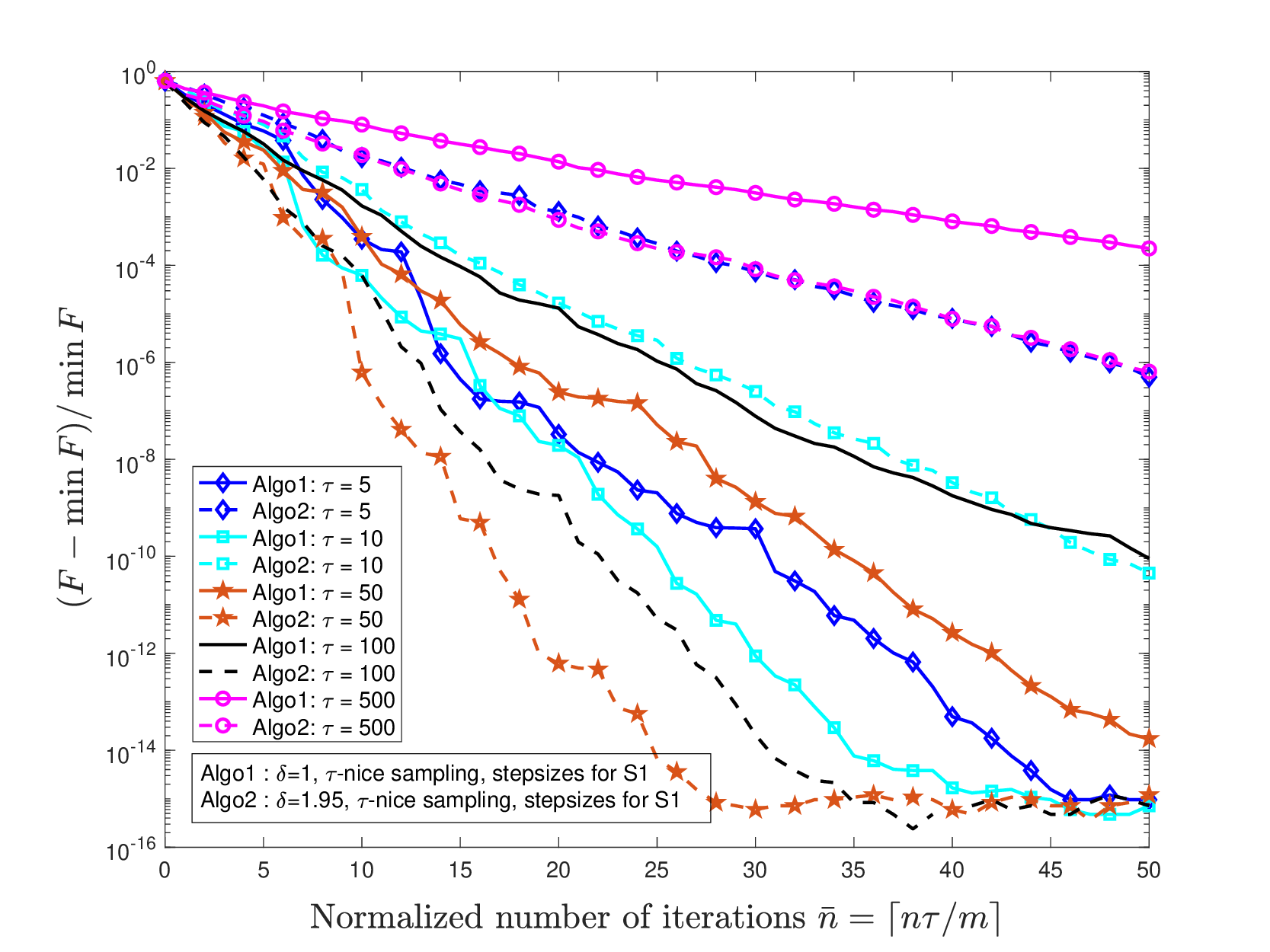}
\hspace{-3ex}
\includegraphics[width=0.5\textwidth]{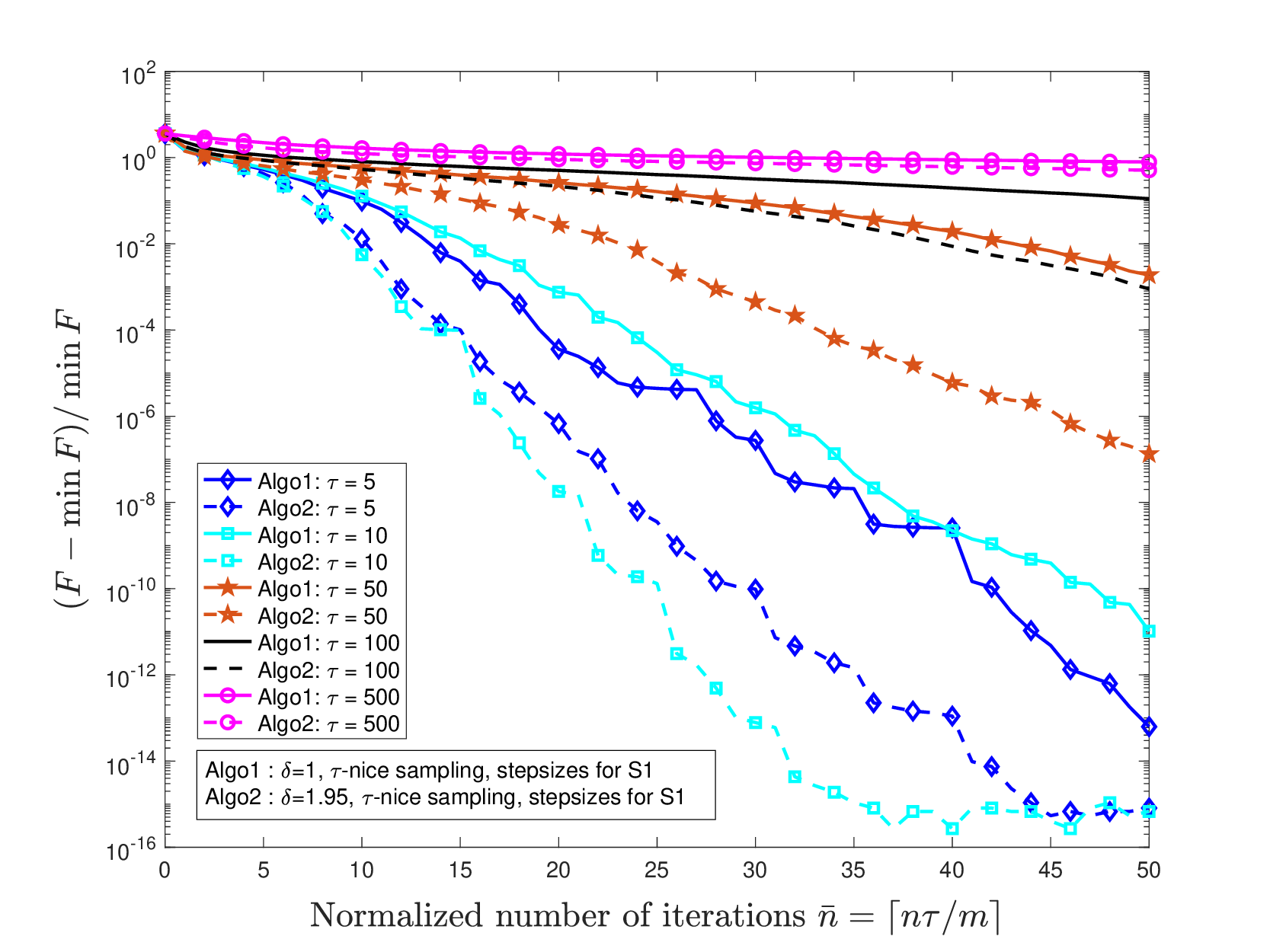}
\includegraphics[width=0.5\textwidth]{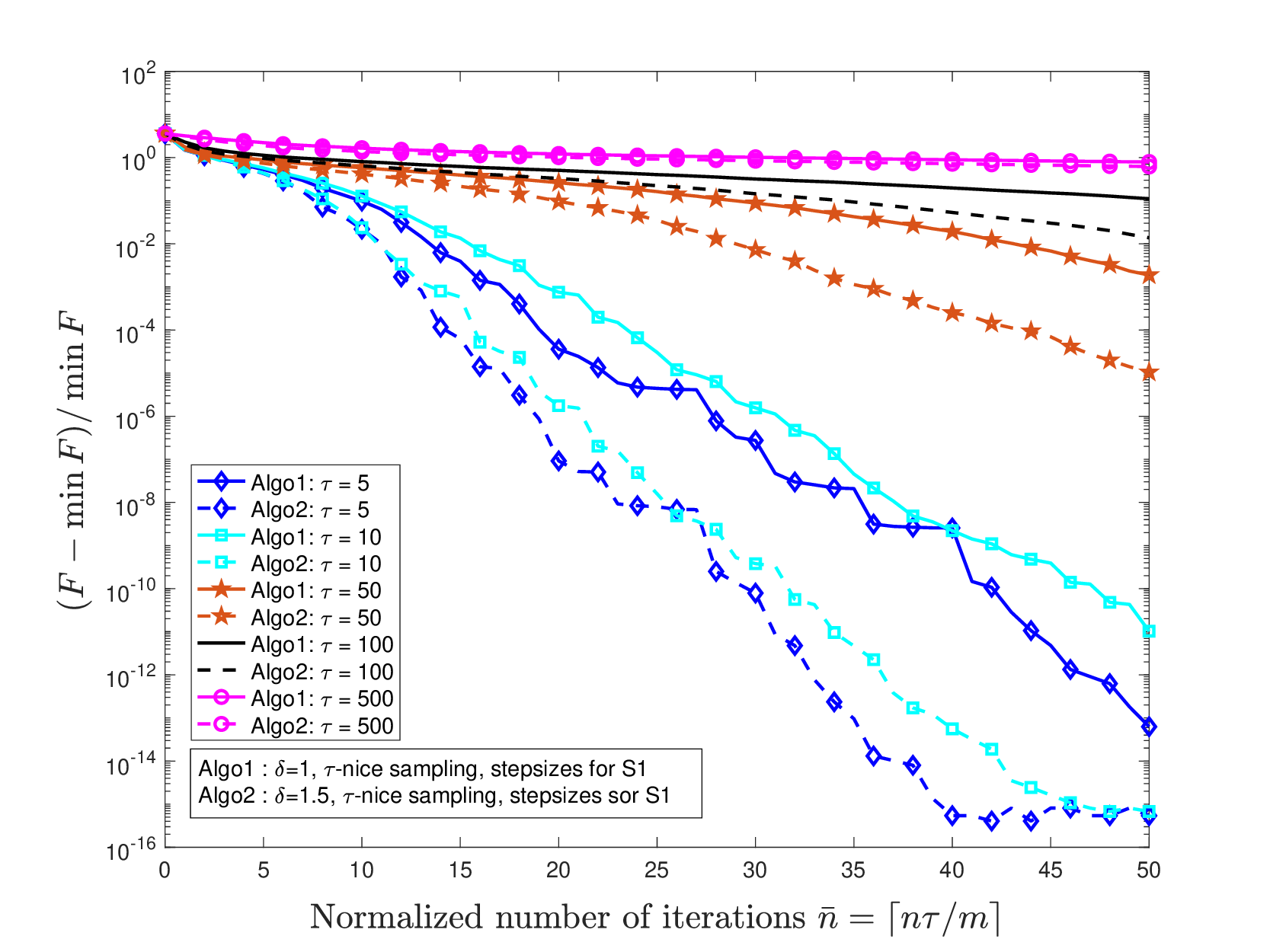}
\hspace{-3ex}
\includegraphics[width=0.5\textwidth]{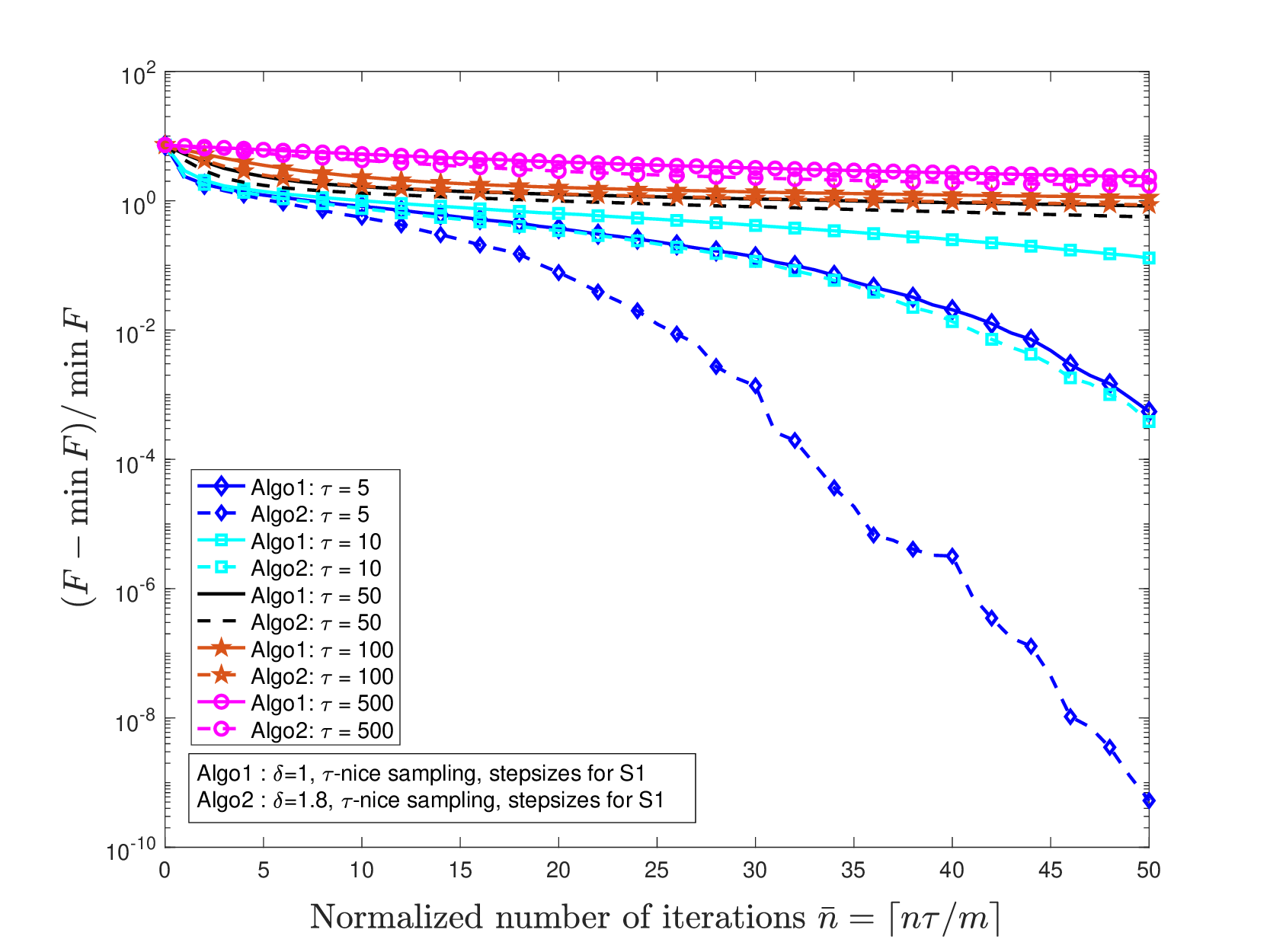}
\end{center}
\vspace{-2em}
\caption{\small Comparison between $\delta=1$ and $\delta>1$.
Lasso problem with $n=10^3$ equations in $m= 5\cdot 10^3$ unknowns. 
Degree of partial separability $\eta=71$ (top left), $\eta= 563$ (top right and bottom left), and $\eta = 2594$ (bottom right). The choice $\delta=1$ is better than $\delta>1$ only for $\eta=71$ and $\tau=5,10$.}
\label{lasso_overrelaxed2}
\end{figure}
In the left diagram, we considered a large scale setting with $\eta\ll m$. In that case
$\beta_1$ may be much smaller than $\beta_2$ leading to significantly larger stepsizes. Moreover, and more importantly, we note that as long as $\tau$ is small enough the behavior of Algo1 does not depend on $\tau$ (indeed $\tau=10,50,100$ perform equally well), whereas this is not true for Algo2.
This feature of \ref{eq:S1}, first noted in \cite{Ric16}, has been already discussed after 
Theorem~\ref{thm:20171207a} and  is at the basis 
of the effectiveness of the parallel strategy.
Indeed in the small-$\tau$ regime described above, 
the various versions of Algo1 depicted in Figure~\ref{lasso_stepsizes}(left)
have the same total computation cost ($\bar{n} m$ block-coordinate updates),
but the parallel implementation on $\tau$ cores is $\tau$ times faster than the serial one ($\tau=1$).
Finally, in the right diagram of Figure~\ref{lasso_stepsizes} we show a scenario in which 
$\eta/m$ is larger (the problem is less sparse).
In such situation we see that the difference between the two 
stepsize selection criteria is less evident for $\tau \geq 50$.
Moreover, Algo1 is more sensitive to $\tau$ (compare $\tau=10$ and $\tau= 50$),
so that the benefit of the parallel scheme is reduced.
%-----------------------------------------------------
\subsection*{The effect of $\delta>1$}
Here we study the effect of over-relaxing the stepsizes, meaning choosing $\delta>1$.
We compare Algorithm~\ref{algoRCD} with $\delta=1$ and several choices of $\delta>1$.
Figure~\ref{lasso_overrelaxed2} considers different scenarios for the degree of separability $\eta$ of $f$.
In those cases we see that choosing $\delta>1$ usually speeds up the convergence,
depending on the parameter $\eta$ of partial separability of $f$
and the number $\tau$ of parallel block updates.
This fact seems not to occur when both $\eta/m$ and $\tau/m$ are very small.

\appendix

\section{Structured Lipschitz smoothness}

In this section we discuss the Lipschitz smoothness properties of $\bfs$
under the hypotheses \ref{eq:B0} and \ref{eq:B1}
and we  prove Theorem~\ref{thm:stepsizes}. Most of the results presented in this section are basically given in \cite{Ric16}. However, here they are rephrased in our notation and extended to our more general assumptions.

\begin{proposition}
\label{p:20181026a}
Let $\bfs\colon \bHH \to \R$ be a convex function satisfying assumptions 
\ref{eq:B0} and \ref{eq:B1}.
Let $I$ be a nonempty subset of $[m]$ and
let $(q_i)_{i \in I} \in \mathbb{R}_{++}$ be such that $\sum_{i \in I \cap I_k} q_i \leq 1$, for every $k \in [p]$.
Then for every $\bxx$ and $\byy \in \bHH$ 
such that $\mathrm{spt}(\bxx - \byy) \subset I$, we have
\begin{equation}
\label{eq:20181014b}
\bfs(\byy) \leq \bfs(\bxx) + 
\scalarp{\nabla \bfs(\bxx), \byy - \bxx} + \frac 1 2 \sum_{i \in I} \frac{L_i}{q_i} \norm{\yy_i - \xx_i}^2.
\end{equation}
\end{proposition}
\begin{proof}
Let $\bvv = \byy- \bxx$ and, for every $k \in [p]$, set $\zz^{(k)} = \sum_{i=1}^m \UU_{k,i} \xx_i$. Then 
\begin{equation*}
\bfs(\byy) = \sum_{k=1}^p \fs_k \bigg(\sum_{i=1}^m \UU_{k,i}(\xx_i + \vv_i)\bigg)
= \sum_{k=1}^p \fs_k \Big(z^{(k)} + \sum_{i \in I} \UU_{k,i} \vv_i \Big)
= \sum_{k=1}^p \fs_k \Big(z^{(k)} + \sum_{i \in I\cap I_k} \UU_{k,i} \vv_i \Big).
\end{equation*}
Now, for every $k \in [p]$, we have
\begin{equation*}
\zz^{(k)} + \sum_{i \in I \cap I_k} \UU_{k,i} \vv_i
= \Big( 1 - \sum_{i \in I\cap I_k} q_i \Big) \zz^{(k)} + \sum_{i \in I\cap I_k} q_i \big( \zz^{(k)} 
+ q_i^{-1} \UU_{k,i} \vv_i \big).
\end{equation*}
Therefore, using the convexity of each $\fs_k$ we have
\begin{equation*}
\bfs(\byy)  = \sum_{k=1}^p \fs_k \Big(\zz^{(k)} + \sum_{i \in I\cap I_k} \UU_{k,i} \vv_i \Big)
\leq \sum_{k=1}^p \bigg[\Big( 1 - \sum_{i \in I\cap I_k} q_i \Big) \fs_k(\zz^{(k)}) 
+ \sum_{i \in I \cap I_k}q_i \fs_k \big( \zz^{(k)} + q_i^{-1} \UU_{k,i} \vv_i \big) \bigg].
\end{equation*}
It follows from the definition of $I_k$ that
$\sum_{i \in I\setminus I_k} q_i \fs_k ( \zz^{(k)} + q_i^{-1} \UU_{k,i} \vv_i)
=\sum_{i \in I\setminus I_k} q_i \fs_k (\zz^{(k)})$.
Hence, switching the order of summation,
and using the fact that 
$\bfs(\xx_1, \dots, \xx_{i-1}, \cdot, \xx_{i+1}, \dots, \xx_m)$
is Lipschitz smooth with constant $L_i$, we have
\begin{align*}
\bfs(\byy) 
&\leq 
\sum_{k=1}^p \bigg[\Big( 1 - \sum_{i \in I} q_i \Big) \fs_k(\zz^{(k)}) 
+ \sum_{i \in I}q_i \fs_k \big( \zz^{(k)} + q_i^{-1} \UU_{k,i} \vv_i \big) \bigg]\\
& = \Big( 1 - \sum_{i \in I} q_i \Big) \bfs(\bxx) +
\sum_{i \in I} q_i \sum_{k=1}^p \fs_k \bigg( \sum_{j=1}^m \UU_{k,j} \big( \xx_j + q_i^{-1} (\E_i \vv_i)_j \big)\bigg)\\
& = \Big( 1 - \sum_{i \in I} q_i \Big) \bfs(\bxx) +
\sum_{i \in I} q_i \bfs\big( \bxx + q_i^{-1} \E_i \vv_i\big)\\
&\leq  \Big( 1 - \sum_{i \in I} q_i \Big) \bfs(\bxx) +
\sum_{i \in I} q_i \bigg[ \bfs(\bxx) + \scalarp{\nabla_i \bfs(\bxx), q_i^{-1}\vv_i} + \frac{L_i}{2} \norm{q_i^{-1} \vv_i}^2\bigg]\\
& = \bfs(\bxx) + \scalarp{\nabla \bfs(\bxx), \bvv} + \frac 1 2 \sum_{i \in I} \frac{L_i}{q_i} \norm{\vv_i}^2.
\qedhere
\end{align*}
\end{proof}

\begin{corollary}
\label{p:20181014a}
Let $\bfs\colon \bHH \to \R$ be a convex function satisfying  
\ref{eq:B0} and \ref{eq:B1}.
Let $\eta = \max_{1 \leq k \leq m} \mathrm{card}(I_k)$, $(\gamma_i)_{1 \leq i \leq m} \in \R^m_{++}$, and
 $(q_i)_{1 \leq i \leq m} \in \mathbb{R}^m_{++}$ be such that, for every 
$k \in [p]$, 
$\sum_{i \in I_k} q_i \leq 1$. Let 
$\mathsf{\Lambda} = \bigoplus_{i=1}^m L_i \Id_i$,
$\mathsf{\Gamma} = \bigoplus_{i=1}^m \gamma_i \Id_i$, and
$\mathsf{Q} = \bigoplus_{i=1}^m q_i \Id_i$.
Then, the function $\bfs$ is
 Lipschitz smooth
 \begin{enumerate}[{\rm (i)}]
\item\label{p:20181014a_1} in the metric $\norm{\cdot}_{\bWW}$ 
defined by $\bWW = \mathsf{\Lambda} \mathsf{Q}^{-1}$ with constant $1$;
\item\label{p:20181014a_4} in the metric $\norm{\cdot}_{\mathsf{\Gamma}^{-1}}$, with 
constant\footnote{ 
This constant is $\leq \delta\eta/\min_{1 \leq i \leq m}\beta_{1,i}$ if the 
$\gamma_i$'s are set according to \eqref{eq:stepsizerule} with the $\nu_i$'s
as in Theorem~\ref{thm:stepsizes}\ref{thm:stepsizes_i}.
}
$\max_{1 \leq k \leq p} \sum_{i \in I_k} \gamma_i L_i$.
\item\label{p:20181014a_2} in the (original) metric of $\bHH$ with constant 
$\max_{1 \leq k \leq p}\sum_{i \in I_k} L_i$;
\item\label{p:20181014a_3} in the metric $\norm{\cdot}_{\mathsf{\Lambda}}$, with constant $\eta$.
\end{enumerate}
\end{corollary}
\begin{proof}
\ref{p:20181014a_1}:
It follows from Proposition~\ref{p:20181026a} with $I = [m]$ and  noting 
that $\scalarp{\nabla \bfs(\bxx), \byy - \bxx} = \scalarp{\nabla^\bWW \bfs(\bxx), \byy - \bxx}_\bWW$
and then invoking the characterization of the Lipschitz continuity of the gradient troughout the descent lemma (see \cite[Theorem~18.15(iii)]{book1}).

\ref{p:20181014a_4}: It follows from \eqref{eq:20181014b} by choosing $I = [m]$, 
$q_i = \gamma_i L_i/(\max_{1 \leq k \leq p}\sum_{j \in I_k} \gamma_j L_j)$, and noting 
that $\scalarp{\nabla \bfs(\bxx), \byy - \bxx} = \scalarp{\nabla^{\mathsf{\Gamma}^{-1}} \bfs(\bxx), \byy - \bxx}_{\mathsf{\Gamma}^{-1}}$
and then invoking \cite[Theorem~18.15(iii)]{book1}.

\ref{p:20181014a_2}: It follows from \ref{p:20181014a_4} with $\gamma_i = 1$.

\ref{p:20181014a_3}: It follows from \ref{p:20181014a_4} with $\gamma_i = 1/L_i$.
\end{proof}

\begin{remark}
\label{rmk:20181030a}
If $\eta = m$ ($\bfs$ is not partially separable),
Corollary~\ref{p:20181014a}-\ref{p:20181014a_2}-\ref{p:20181014a_3} establishes that
\begin{equation*}
L_i \leq 
L
\leq \sum_{i=1}^m L_i
\quad\text{and}\quad
L_{\norm{\cdot}_\mathsf{\Lambda}}
\leq m.
\end{equation*}
We show that the above bounds are tight.
Indeed,
 if we consider $\bfs(\bxx) = (1/2) \norm{\bAs \bxx - \bbs}_2^2$,
 where $\bAs \in \mathbb{R}^{n\times m}$ and $\bbs \in \mathbb{R}^m$,
then we have $L_i = \norm{\bas_i}^2$ (where $\bas_i$ is the $i$-th column of $\bAs$),
 so that $\sum_{i=1}^m L_i = \norm{\bAs}_F^2$.
Instead, since $\nabla \bfs (\bxx)= \bAs^*(\bAs \bxx - \bbs)$,
 the Lipschitz constant of $\nabla \bfs$ is $\norm{\bAs}^2$.
It is well-known that if $\bAs$ is rank one, then $\norm{\bAs}^2 = \norm{\bAs}_F^2$,
so in this case the Lipschitz constant of  $\nabla \bfs$ is exactly $\sum_{i=1}^m L_i$.
Moreover, if in addition the columns of $\bAs$ have the same norm, then 
$L = \sum_{j=1}^m L_j = m L_i$ and hence 
$L_{\norm{\cdot}_{\Lambdas}} = m$.
We finally note that if $\bAs$ is an orthonormal matrix, then $\norm{\bAs}^2 = 1$ and hence
 $1 = L_i = L = L_{\norm{\cdot}_{\Lambdas}} < \sum_{i=1}^m L_i = m$.
\end{remark}

\begin{corollary}
\label{cor:20181026b}
Let $\bfs\colon \bHH \to \R$ be a function satisfying  
\ref{eq:B0} and \ref{eq:B1}.
Let $\beps \in \{0,1\}^m$ and  $\bxx, \bvv \in \bHH$.
Then, 
\begin{equation}
\label{eq:20181109c}
\bfs(\bxx + \beps \odot \bvv) \leq \bfs(\bxx) + \sum_{i =1}^m \epsilon_i \scalarp{\nabla_i \bfs(\bxx), \vv_i}
+ \max_{1 \leq k \leq p}\Big(\sum_{i \in I_k} \epsilon_i\Big) \sum_{i=1}^m \epsilon_i \frac{L_i}{2} \norm{\vv_i}^2.
\end{equation}
\end{corollary}
\begin{proof}
It follows from Proposition~\ref{p:20181026a} with $\byy = \bxx 
+ \beps \odot \bvv$,
 $I = \mathrm{spt}(\beps)$, and $q_i 
=1/\big(\max_{1 \leq k \leq p}\mathrm{card}(I \cap I_k) \big)= 1/\big( \max_{1 \leq k \leq p}\sum_{i \in I_k} \epsilon_i\big)$.
\end{proof}

\begin{remark}
Most of the above results, appears in \cite{Ric16}
for the special case that 
$\UU_{k,i} = \E_i$ for $i \in I_k$ and $\UU_{k,i} = 0$ for $i \notin I_k$.
In particular, see \cite[Theorem~8]{Ric16}.
\end{remark}

\begin{proposition}
\label{p:20190108a}
Let $\bfs\colon \bHH \to \R$ be a function satisfying
\ref{eq:B0} and suppose that, for every $k \in [p]$,
 $\fs_k$ is $L^{(k)}$-Lipschitz smooth. Set
for every $i \in [m]$, $\tilde{L}_i = \norm{ \sum_{k=1}^p L^{(k)} \UU_{k,i}^\top \UU_{k,i}}$.
Then the following holds.
\begin{enumerate}[{\rm (i)}]
\item\label{p:20190108a_i} $\bfs$ is Lipschitz smooth with constant 
$\norm{\sum_{k=1}^p L^{(k)} U_k^\top U_k}$ in the original metric of $\bHH$;\footnote{
In \cite[Corollary~5.11]{Com15} the worse constant 
$\sum_{k=1}^p L^{(k)} \norm{\UU_k \UU_k^\top}$ was considered.}
\item\label{p:20190108a_ii} 
$\bfs$ satisfies assumption \ref{eq:B1} with $L_i = \tilde{L}_i$;
\item\label{p:20190108a_iii} 
Suppose that, for every $k \in [p]$ and for every $i,j \in [m]$, $i\neq j$, 
the range of the operators $\UU_{k,i}$ and 
$\UU_{k,j}$ are orthogonal. Then, for every $\bxx$, $\bvv \in \bHH$, 
\begin{equation}
\bfs(\bxx + \bvv) \leq \bfs(\bxx) + 
\scalarp{\nabla \bfs(\bxx), \bvv} + \frac 1 2 \sum_{i=1}^m \tilde{L_i}\norm{\vv_i}^2.
\end{equation}
\end{enumerate}
\end{proposition}
\begin{proof}
\ref{p:20190108a_i}:
For every $k \in [p]$, let $\UU_k \colon \bHH \to \GG_k$, 
$\UU_k \bxx = \sum_{i=1}^m \UU_{k,i} \xx_i$.
Let $\bxx, \bvv \in \bHH$. We have
\begin{equation}
\label{eq:20190111b}
\bfs(\bxx + \bvv) = \sum_{k=1}^p \fs_k( \UU_k \bxx + \UU_k \bvv )
\leq \sum_{k=1}^p \bigg(f_k( \UU_k \bxx)
+ \scalarp{\nabla f_k ( \UU_k \bxx), \UU_k \bvv} 
+ \frac{L^{(k)}}{2} \norm{\UU_k \bvv}^2 \bigg).
\end{equation}
Therefore, we have
\begin{align}
\nonumber \bfs(\bxx + \bvv) 
& \leq \bfs(\bxx) 
+ \sum_{k=1}^p \scalarp{\UU_k^\top \nabla f_k ( \UU_k \bxx), \bvv} 
+ \frac 1 2  \sum_{k=1}^p L^{(k)} \scalarp{\UU_k^\top \UU_k \bvv,\bvv}\\
\label{eq:20190112c}&= \bfs(\bxx) 
+  \scalarp{ \nabla f (\bxx), \bvv} 
+ \frac 1 2   \big\langle \sum_{k=1}^p L^{(k)}\UU_k^\top \UU_k \bvv,\bvv \big\rangle\\
\nonumber& \leq \bfs(\bxx) 
+  \scalarp{ \nabla f (\bxx), \bvv} 
+ \frac 1 2   \bigg\lVert \sum_{k=1}^p L^{(k)}\UU_k^\top \UU_k \bigg\rVert \norm{\bvv}^2.
\end{align}
\ref{p:20190108a_ii}:
It follows from \eqref{eq:20190112c} with $\bvv = \E_i \vv_i$ that
\begin{align*}
f(\bxx + \E_i \vv_i) &\leq \bfs(\bxx)+
\scalarp{ \nabla_i f (\bxx), \vv_i}
+ \frac 1 2   \big\langle \sum_{k=1}^p L^{(k)} \UU_{k,i}^\top \UU_{k,i}  \vv_i,\vv_i \big\rangle\\
& \leq \bfs(\bxx)+
\scalarp{ \nabla_i \bfs (\bxx), \vv_i}
+ \frac 1 2   \Big\lVert \sum_{k=1}^p L^{(k)} \UU_{k,i}^\top \UU_{k,i}\Big\rVert  
\norm{\vv_i}^2
\end{align*}
hence $\tilde{L}_i = \norm{\sum_{k=1}^p L^{(k)} \UU_{k,i}^\top \UU_{k,i}}$
is a Lipschitz constant of
 $\nabla_i f(\xx_i, \dots, \xx_{i-1},\cdot, \xx_{i+1}, \dots, \xx_m )$.

\ref{p:20190108a_iii}:
Since $\scalarp{\UU_{k,i}\vv_i, \UU_{k,j} \vv_j} = 0$ if $i \neq j$, 
it follows from \eqref{eq:20190111b} that
\begin{align*}
\bfs(\bxx + \bvv) 
& \leq \bfs(\bxx) 
+ \sum_{k=1}^p \scalarp{\UU_k^\top \nabla f_k ( \UU_k \bxx), \bvv} 
+ \frac 1 2  \sum_{i=1}^m \sum_{k=1}^p L^{(k)} 
\scalarp{\UU_{k,i}^\top \UU_{k,i} \vv_i, \vv_i}\\
& = \bfs(\bxx) 
+  \scalarp{ \nabla f ( \bxx), \bvv} 
+ \frac 1 2  \sum_{i=1}^m 
\big\langle \sum_{k=1}^p L^{(k)} \UU_{k,i}^\top \UU_{k,i} \vv_i, \vv_i \big\rangle\\
& \leq \bfs(\bxx) 
+  \scalarp{ \nabla f ( \bxx), \bvv} 
+ \frac 1 2  \sum_{i=1}^m \Big\lVert \sum_{k=1}^p L^{(k)} \UU_{k,i}^\top \UU_{k,i} \Big\rVert \norm{\vv_i}^2.
\qedhere
\end{align*}
%and the statement follows.
\end{proof}

\begin{remark}
 If $R(\UU_{k,i})$ and $R(\UU_{k,j})$ are orthogonal to each other, then
\begin{equation*}
\sum_{k=1}^p L^{(k)}\UU_k^\top \UU_k 
= \sum_{i=1}^m \E_i \bigg(\underbrace{\sum_{k=1}^p L^{(k)} \UU_{k,i}^\top \UU_{k,i}}_{\HH_i \to \HH_i}\bigg) \E_i^\top,
\end{equation*}
and hence $\norm{\sum_{k=1}^p L^{(k)}\UU_k^\top \UU_k } = \max_{1 \leq i \leq m} \tilde{L}_i$.
\end{remark}

\section{Additional proofs}
\label{sec:appB}

\noindent
{\bf Proof.~of Theorem~\ref{thm:stepsizes}}

\ref{thm:stepsizes_ii}:
It follows from \eqref{eq:20181109c} 
that, point-wise it holds
\begin{equation}
\label{eq:20180925c}
 \bfs(\bxx + \bvarepsilon \odot \bvv) 
\leq \bfs(\bxx) +
\scalarp{\nabla \bfs(\bxx), \bvarepsilon \odot \bvv}
+ \max_{1 \leq k \leq p}\Big(\sum_{i \in I_k} \varepsilon_i \Big) \sum_{i=1}^m 
\varepsilon_i\frac{L_{i}}{2} \norm{\vv_{i}}^2.
\end{equation}
Moreover, since
$\beta_2 =\esssup 
\big(\max_{1 \leq k \leq p} \textstyle\sum_{i \in I_k} \varepsilon_i \big)$,
we have that $L_i \max_{1 \leq k \leq p} \textstyle\sum_{i \in I_k}
 \varepsilon_i \leq L_i \beta_2 \leq \nu_i$  $\PP$-a.s. The statement follows.

\ref{thm:stepsizes_i}
It follows by taking the expectation in  \eqref{eq:20180925c} 
and noting that
\begin{align*}
\EE \bigg[ \max_{1 \leq k \leq p}\Big(\sum_{i \in I_k} \varepsilon_i \Big) 
\sum_{i=1}^m \varepsilon_i\frac{L_{i}}{2}  \norm{\vv_{i}}^2\bigg]
& = \sum_{i=1}^m \EE\bigg[ \varepsilon_i \max_{1 \leq k \leq p}\Big(\sum_{i \in I_k} \varepsilon_i \Big)\bigg] \frac{L_i}{2} \norm{\vv_i}^2\\
& = \sum_{i=1}^m \EE\bigg[ \max_{1 \leq k \leq p}\Big(\sum_{i \in I_k} \varepsilon_i \Big) 
\,\Big\vert\, \varepsilon_i = 1\bigg] \pp_i \frac{L_i}{2} \norm{\vv_i}^2,
\end{align*}
where we used the fact that 
for every discrete random variable $\zeta$, 
$\EE[\varepsilon_i \zeta] = \EE[\zeta\,\vert\, \varepsilon_i=1] \pp_i$.

\ref{thm:stepsizes_iii}
It follows from 
Proposition~\ref{p:20190108a}\ref{p:20190108a_iii} that \ref{eq:S0}
holds
with $\nu_i = \tilde{L}_i$.

\ref{thm:stepsizes_iv}:
Let $i \in [m]$ and $\vv_i \in \HH_i$ and set $\bvv = \E_i \vv_i =  (0,\dots, 0, \vv_i, 0, \dots, 0)$.
Then
\begin{align*}
\EE[\bfs(\bxx + \bvarepsilon \odot \bvv)] &= \EE[\bfs(\bxx + \E_i(\varepsilon_i \vv_i))]
= \pp_i \bfs(\bxx + \E_i \vv_i) + (1-\pp_i)\bfs(\bxx)\\
\EE[\scalarp{\nabla \bfs(\bxx), \bvarepsilon \odot \bvv}] & = \EE[\scalarp{\nabla_i \bfs(\bxx), \varepsilon_i \vv_i}]
 = \pp_i \scalarp{\nabla_i \bfs(\bxx), \vv_i}
\end{align*}
Hence, it follows from \ref{eq:S1} that
$\bfs(\bxx + \E_i \vv_i) \leq \bfs(\bxx) + \scalarp{\nabla_i \bfs(\bxx), \vv_i} + (1/2) \nu_i \norm{\vv_i}^2$.
This shows that $\bfs$ is Lipschitz smooth w.r.t.~the $i$-th block coordinate
with Lipschitz constant $\nu_i$.
The 
global Lipschitz smoothness of $\bfs$ follows from Corollary~\ref{p:20181014a}.
\qed

\vspace{2ex}
\noindent
{\bf Proof.~of Remark~\ref{rmk:20190128e}\ref{rmk:20190128e_iv}.}

We have
$\bfs(\bxx) = \sum_{k=1}^p \fs_k( \UU_k \bxx)$
and $\bfs(\bxx + \bvv) = \sum_{k=1}^p \fs_k \big( \UU_k \bxx + \sum_{i=1}^m\UU_{k,i} \vv_i \big)$.
Moreover, $\nabla \bfs(\bxx) = \sum_{k=1}^p \UU_k^\top \nabla \fs_k( \UU_k \bxx )$.
We let $\bxx \in \bHH$ and define
\begin{align*}
\varphi(\bvv) &:= \bfs(\bxx + \bvv) - \bfs(\bxx) - \scalarp{\nabla \bfs(\bxx), \bvv}\\
\psi_k(\buu) &:= \fs_k( \UU_k\bxx + \buu) - \fs_k(\UU_k\bxx) - \scalarp{\nabla \fs_k(\UU_k\bxx), \buu}.
\end{align*}
We clearly have $\varphi(\bvv) \geq 0$, $\varphi(0) = 0$
and $\psi_k(\buu) \geq 0$, $\psi_k(0) = 0$.
Moreover,
\begin{equation}
\varphi(\bvv) = \sum_{k=1}^p \fs_k (\UU_k \bxx + \UU_k \bvv ) - \sum_{k=1}^p \fs_k ( \UU_k \bxx )
- \sum_{k=1}^p \scalarp{\nabla \fs_k (\UU_k \bxx), \UU_k \bvv} =
\sum_{k=1}^p \psi_k(\UU_k \bvv).
\end{equation}
Therefore,
\begin{align}
\nonumber \EE [\varphi(\bvarepsilon \odot \bvv)]
\nonumber&= \sum_{k=1}^p \EE \Big[\psi_k\Big ( \sum_{i\in I_k}  \UU_{k,i}\varepsilon_i\vv_i\Big)\Big]\\
\label{eq:20200306d}& = \sum_{k=1}^p \sum_{t=1}^\eta\EE \Big[\psi_k\Big ( \sum_{i\in I_k}  \UU_{k,i}\varepsilon_i\vv_i\Big) \,\Big\vert\, \sum_{i \in I_k} \varepsilon_i = t\Big] \PP\bigg( \sum_{i \in I_k} \varepsilon_i = t\bigg).
\end{align}
Now, if $\omega \in \Omega$ is such that $\sum_{i \in I_k} \varepsilon_i(\omega)=t$
and we set $I = \{i \in [m] \,\vert\, \varepsilon_i(\omega) = 1\}$, then we have
$\mathrm{card}(I\cap I_k) = t$ and
\begin{equation*}
\psi_k\Big ( \sum_{i\in I_k}  \UU_{k,i}\varepsilon_i(\omega)\vv_i\Big)
= \psi_k\Big ( \sum_{i\in I\cap I_k}  \UU_{k,i}\vv_i\Big) 
\leq \frac 1 t \sum_{i\in I\cap I_k} \psi_k (t \UU_{k,i}\vv_i)
= \frac 1 t \sum_{i = 1}^m \varepsilon_i(\omega)\psi_k (t \UU_{k,i}\vv_i).
\end{equation*}
Hence $\psi_k\big ( \sum_{i\in I_k}  \UU_{k,i}\varepsilon_i\vv_i\big) \leq (1/t) \sum_{i=1}^m \varepsilon_i\psi_k (t \UU_{k,i}\vv_i)$ on the event $\sum_{i \in I_k} \varepsilon_i=t$.
Then,
\begin{align}
\nonumber \EE \bigg[\psi_k\Big ( \sum_{i\in I_k}  \UU_{k,i}\varepsilon_i\vv_i\Big) \,\Big\vert\, \sum_{i \in I_k} \varepsilon_i 
= t\bigg] 
\nonumber&\leq\frac 1 t \sum_{i=1}^m \psi_k (t \UU_{k,i}\vv_i) \EE \Big[ \varepsilon_i \,\Big\vert\,  \sum_{i \in I_k} \varepsilon_i = t \Big]\\
& = \frac 1 t \sum_{i=1}^m \psi_k (t \UU_{k,i}\vv_i) \PP \bigg( \varepsilon_i =1\,\Big\vert\,  \sum_{i \in I_k} \varepsilon_i = t \bigg).
\end{align}
Plugging the above inequality in \eqref{eq:20200306d} we get
\begin{align}
\label{eq:20200409a}
\nonumber \EE [\varphi(\bvarepsilon \odot \bvv)] & 
\leq \sum_{k=1}^p \sum_{t=1}^\eta
\frac 1 t \sum_{i=1}^m \psi_k (t \UU_{k,i}\vv_i) \PP \bigg( \varepsilon_i =1\,\Big\vert\,  \sum_{i \in I_k} \varepsilon_i = t \bigg)\PP\bigg( \sum_{i \in I_k} \varepsilon_i = t\bigg)\\
\nonumber & = 
\sum_{i=1}^m \sum_{t=1}^\eta \frac 1 t 
\sum_{k=1}^p  \psi_k (t \UU_{k,i}\vv_i) \PP \bigg( \sum_{i \in I_k} \varepsilon_i = t \,\Big\vert\, \varepsilon_i = 1 \bigg) \pp_i\\
\nonumber& \leq \sum_{i=1}^m \sum_{t=1}^\eta \frac 1 t  
\max_{\substack{1 \leq k \leq p \\[0.3ex] i \in I_k}} \PP \bigg( \sum_{i \in I_k} \varepsilon_i = t \,\Big\vert\, \varepsilon_i = 1 \bigg) \pp_i  \sum_{k=1}^p \psi_k ( \UU_{k,i} t\vv_i)\\
\nonumber& = \sum_{i=1}^m \sum_{t=1}^\eta \frac 1 t  
\max_{\substack{1 \leq k \leq p \\[0.3ex] i \in I_k}} \PP \bigg( \sum_{i \in I_k} \varepsilon_i = t \,\Big\vert\, \varepsilon_i = 1 \bigg) \pp_i  \varphi (\E_i t \vv_i)\\
& \leq \sum_{i=1}^m \sum_{t=1}^\eta \frac 1 t  
\max_{\substack{1 \leq k \leq p \\[0.3ex] i \in I_k}} \PP \bigg( \sum_{i \in I_k} \varepsilon_i = t \,\Big\vert\, \varepsilon_i = 1 \bigg) \pp_i  \frac{L_i}{2} t^2 \norm{\vv_i}^2,
\end{align}
where in the last inequality we used \ref{eq:B1}.
So, setting $\beta_{1,i}$ as in \eqref{eq:20200316a}, if $\nu_i \geq \beta_{1,i} L_i$, then
\ref{eq:S1} holds. Note that in deriving \eqref{eq:20200409a}, 
if for some $i \in [m]$ there are no $k \in [p]$ such that $i \in I_k$, the corresponding term $\max_{k \in \varnothing} \PP \big( \sum_{i \in I_k} \varepsilon_i = t \,\vert\, \varepsilon_i = 1 \big)$ can be set to zero.
\qed

\vspace{2ex}
\noindent
{\bf Proof.~of formula~\eqref{eq:20200316b}.}

Since in the proof of \eqref{eq:20200316a} in Remark~\ref{rmk:20190128e}\ref{rmk:20190128e_iv}, we only use the fact that $i \notin I_k\ \Rightarrow\ \UU_{k,i}=0$,
we can assume, without loss of generality, that, for every $k \in [p]$, $\mathrm{card}(I_k) = \eta$.
Let $i \in \bigcup_{k=1}^p I_k$.
Then, since the block sampling is doubly uniform, we have that
$\PP \big( \sum_{i \in I_k} \varepsilon_i = t \,\vert\, \varepsilon_i = 1 \big)$
does not depend on $k$ such that $i \in I_k$. Therefore, \eqref{eq:20200316a}
becomes $\beta_{1,i} = \sum_{t=1}^\eta t 
\PP\big(\sum_{j \in I_k} \varepsilon_j = t \,\vert\, \varepsilon_i=1 \big)$,
for some $k \in [p]$ such that $i \in I_k$.
Hence
\begin{align}
\nonumber\beta_{1,i} &= \EE\Big[\sum_{j \in I_k} \varepsilon_k \,\Big\vert\, \varepsilon_i=1 \Big]
= \sum_{j \in I_k} \EE[\varepsilon_j \,\vert\, \varepsilon_i=1]
= \sum_{j \in I_k} \PP(\varepsilon_j =1 \,\vert\, \varepsilon_i=1)\\
\label{eq:20200409b}&= \frac{1}{\pp}\sum_{i \in I_k} \PP(\varepsilon_j=1, \varepsilon_i=1)
=  1 + (\eta - 1) \frac{\tilde{\pp}}{\pp},
\end{align}
where $\pp = \PP(\varepsilon_i=1)$ and $\tilde{\pp} = \PP(\varepsilon_j=1, \varepsilon_i=1)$ (with $i \neq j$).
Now, since $\EE[\sum_{i=1}^m \varepsilon_i] = m \pp$ and
$\EE[(\sum_{i=1}^m \varepsilon_i)^2] = \sum_{i=1}^m \EE[\varepsilon_i] + \sum_{i\neq j} \EE[\varepsilon_i \varepsilon_j] = m \pp + m(m-1) \tilde{\pp}$, we have that
$\tilde{\pp} = \big( \EE[(\sum_{i=1}^m \varepsilon_i)^2] -m \pp\big)/(m(m-1))$,
which plugged into \eqref{eq:20200409b} gives \eqref{eq:20200316b}.

%\newpage
\vspace{2ex}
\noindent
{\bf Proof.~of Lemma~\ref{lem:20190313b}}

Let $\zz \in \HHs$.
It follows from the definition of $\xx^+$ 
that $\xx - \xx^+ - \nabla \varphi(\xx) \in \partial \psi(\xx^+)$.
Therefore, 
$\psi(\zz) \geq \psi(\xx^+) + \scalarp{\xx - \xx^+ - \nabla \varphi(\xx), 
\zz - \xx^+} + (\mu_\psi/2)
\norm{\zz - \xx^+}^2,
$
hence
\begin{equation*}
\scalarp{\xx - \xx^+ , \zz - \xx^+} \leq \psi(\zz) - \psi(\xx^+) +  \scalarp{\nabla 
\varphi(\xx), \zz - \xx^+} - \frac{\mu_\psi}{2} \norm{\zz - \xx^+}^2.
\end{equation*}
Now, we note that
$
\norm{\xx^+ - \zz}^2 = \norm{\xx^+ - \xx}^2 + \norm{\xx - \zz}^2 + 2 \scalarp{\xx^+ - \xx, \xx - \zz}$.
Then,
\begin{align*}
\scalarp{\xx - \xx^+ , \zz - \xx} + \scalarp{\xx - \xx^+ , \xx - \xx^+} 
&\leq \psi(\zz) - \psi(\xx^+) 
+  \scalarp{\nabla \varphi(\xx), \zz - \xx} +  \scalarp{\nabla \varphi(\xx), \xx - \xx^+} \\[1ex]
&\quad- \frac{\mu_\psi}{2} \norm{\zz - \xx}^2 -  \frac{\mu_\psi}{2} \norm{\xx - \xx^+}^2 - \mu_\psi
 \scalarp{\xx - \xx^+,\zz - \xx}
\end{align*}
and hence
\begin{align*}
\nonumber(1 + \mu_\psi)\scalarp{\xx - \xx^+\!\!, \zz - \xx} 
\nonumber&\leq \psi(\zz) - \psi(\xx)  +  \scalarp{\nabla \varphi(\xx), \zz - \xx} - \frac{\mu_\psi}{2} \norm{\zz - \xx}^2 \!+\! \psi(\xx) \!-\! \psi(\xx^+) \\[1ex]
&\qquad +  \scalarp{\nabla \varphi(\xx), \xx - \xx^+} -  \Big( 1 +\frac{\mu_\psi}{2} \Big) \norm{\xx - \xx^+}^2.
\end{align*}
Since $\scalarp{\nabla \varphi(\xx), \zz - \xx} \leq \varphi(\zz) - \varphi(\xx) 
- (\mu_\varphi/2) \norm{\zz - \xx}^2$, the statement follows.
\qed
\vspace{2ex}

\begin{lemma}
\label{lem:20200928}
Let $a,b,c \in \R_{++}$. Then the largest constant $\bar{\lambda}>0$ satisfying the following inequality
\begin{equation}
\label{eq:20200928a}
\forall\, (s,t) \in \R_+^2,\ \text{with}\ t\geq cs,\quad cs + t \geq \bar{\lambda}(as + b t)
\end{equation}
is
\begin{equation}
\bar{\lambda} = \min\bigg\{ \frac{1}{b}, \frac{2 c}{a + bc} \bigg\} =
\begin{cases}
\dfrac{1}{b} &\text{if } b\geq \dfrac{a}{c}\\[1.5ex]
\dfrac{2c}{a + bc} &\text{if } b\leq \dfrac{a}{c}.
\end{cases}
\end{equation}
\end{lemma}
\begin{proof}
Property \eqref{eq:20200928a} is equivalent to
\begin{equation*}
\forall\, (s,t) \in \R_+^2,\ \text{with}\ t\geq cs\ \text{and}\ cs+t>0,\quad 
\frac{a s + b t}{c s + t} \leq \frac{1}{\bar{\lambda}}.
\end{equation*}
Therefore,
\begin{align}
\nonumber\frac{1}{\bar{\lambda}} 
&= \sup\bigg\{ \frac{as+bt}{cs+t} \,\bigg\vert\, s,t \in \R_+, cs+t>0, t \geq cs \bigg\}\\[1.7ex]
\label{20200928c}&= \sup\big\{ as+bt \big\vert\, s,t \in \R_+, cs+t=1, t \geq cs\big\}.
%\sup_{\substack{x,y \in \R_+\\cx+y>0\\y \geq cx}} \frac{ax+by}{cx+y}
\end{align}
Now, since 
\begin{equation*}
(cs+t=1\ \text{and}\ t\geq cs)\ \Leftrightarrow\ 
(cs=1-t\ \text{and}\ t\geq 1-t)\ \Leftrightarrow\ 
(cs=1-t\ \text{and}\ t\geq 1/2),
\end{equation*}
it follows from \eqref{20200928c} that
\begin{equation}
\frac{1}{\bar{\lambda}} = \sup_{t \in [1/2, 1]} \frac{a}{c}(1-t) + bt =
\max\bigg\{b, \frac{1}{2}\bigg( \frac{a}{c} + b \bigg)\bigg\} =
\begin{cases}
b &\text{if } b \geq \dfrac{a}{c}\\[1.7ex]
\dfrac{1}{2} \bigg(\dfrac{a}{c} + b \bigg) &\text{if } b \leq \dfrac{a}{c}.
\end{cases}
\end{equation}
Therefore, the statement follows.
\end{proof}

\vspace{2ex}
\noindent
{\bf Proof.~of Theorem~\ref{p:20181130a}}

We first note that, since, $\norm{\cdot}^2_{\bGammas^{-1}} \geq \pp_{\min} \norm{\cdot}^2_\bWW$,
the conclusion of Proposition~\ref{p:20180927a}\ref{p:20180927a_ii}  can be stated 
as follows:
\begin{align}
\nonumber\EE\bigg[ \pp_{\min} \frac{1+\sigma_{\Gammas^{-1}}}{2} &\norm{\bx^{\iter+1} - \bxx}_\bWW^2  + \Fs(\bx^{\iter+1}) - \Fs(\bxx)
\,\Big\vert\, \Fsc_{\iter-1} \bigg] \\[1ex]
\nonumber&\leq \pp_{\min} \frac{1+\sigma_{\Gammas^{-1}}}{2} \norm{\bx^{\iter} - \bxx}_\bWW^2  + \Fs(\bx^{\iter}) - \Fs(\bxx)\\[1ex]
\nonumber&\qquad - \pp_{\min} \bigg( 
 \frac{\mu_{\Gammas^{-1}} + \sigma_{\Gammas^{-1}}}{2} \pp_{\min} \norm{\bx^\iter - \bxx}_{\bWW}^2 + \Fs(\bx^\iter) - \Fs(\bxx)  \bigg)\\[1ex]
\label{eq:20181113a}&\qquad+ \frac{(\delta - 1)_+}{2+\sigma_{\Gammas^{-1}}-\delta} \EE[ \Fs(\bx^{\iter}) -\Fs(\bx^{\iter+1}) 
\,\vert\, \Fsc_{\iter-1} ].
\end{align}
%\ref{p:20181130a_ii}: 
%Let $\bxx \in \dom \Fs$, $\iter \in \N$,
%and set $\xi_{\iter} = \pp_{\min}\norm{\bx^{\iter} - \bxx}_\bWW^2/2
%+ \Fs(\bx^{\iter}) - \Fs(\bxx)$. 
%Then, taking the mean in \eqref{eq:20181113a}, 
%with $\mu_{\Gammas^{-1}}=\sigma_{\Gammas^{-1}} = 0$, we derive
%\begin{equation}
%\label{eq:20190608a}
%\EE[\xi_{\iter+1}] 
%\leq \EE[\xi_\iter] - \pp_{\min} \EE[\Fs(\bx^\iter) - \Fs(\bxx)] 
%+  \frac{(\delta - 1)_+}{2-\delta} \EE[ \Fs(\bx^{\iter}) -\Fs(\bx^{\iter+1}) ].
%\end{equation}
%Then, applying \eqref{eq:20190608a} recursively and recalling that 
%$(\EE[\Fs(\bx^i)])_{i \in \N}$ is decreasing, we have
%\begin{align*}
%\EE[\Fs(\bx^{\iter+1})] - \Fs(\bxx) 
%&\leq \EE[\xi_{\iter+1}] \\
%&\leq \EE[\xi_0]  - \pp_{\min}\sum_{i=0}^\iter \EE[\Fs(\bx^i) - \Fs(\bxx)] 
%+ \frac{(\delta - 1)_+}{2-\delta}\pp_{\min}\sum_{i=0}^\iter \EE[ \Fs(\bx^{i}) -\Fs(\bx^{i+1}) ].\\
%&\leq \EE[\xi_0]  -  \pp_{\min}(\iter+1) \EE[\Fs(\bx^{\iter+1}) - \Fs(\bxx)] 
%+ \frac{(\delta - 1)_+}{2-\delta}\pp_{\min}  \EE[ \Fs(\bx^{0}) -\Fs(\bx^{\iter+1}) ].
%\end{align*}
%Therefore, noting that $(\delta-1)_+/(2-\delta) +1= \max\{1, 1/(2-\delta)\}$
%and choosing $\bxx$  in $\bSS_*$, we derive
%\begin{align}
%\nonumber(1 + \pp_{\min}&(\iter+1))(\EE[\Fs(\bx^{\iter+1})] - \Fs_*) \\[1ex]
%\label{eq:20190126a}&\leq \pp_{\min}\frac{\EE\norm{\bx^{0} - \bxx}_\bWW^2}{2} 
%+ \max\Big\{1, \frac{1}{2-\delta}\Big\}(\EE [\Fs(\bx^{0})] -  \Fs_*),
%\end{align}
%and the statement follows.
%\ref{p:20181130a_iii}: 
Let $\bxx = \bxx_*$ and
set for brevity $r_\iter^2 = (\pp_{\min}/2) \norm{\bx^\iter - \bxx_*}_\bWW^2$,
and $F_\iter = \Fs(\bx^\iter)$. Then, \eqref{eq:20181113a} yields
\begin{align*}
\EE[ (1+\sigma_{\Gammas^{-1}}) r_{\iter+1}^2 + F_{\iter+1} - \Fs_* \,\vert\, \Fsc_{\iter-1}]
&\leq (1+\sigma_{\Gammas^{-1}}) r_\iter^2 + F_\iter - \Fs_* \\
&\quad- \pp_{\min} \big( (\mu_{\Gammas^{-1}}+\sigma_{\Gammas^{-1}}) r_\iter^2 + F_\iter - \Fs_* \big)\\
&\quad + \frac{(\delta - 1)_+}{2+\sigma_{\Gammas^{-1}}-\delta} \EE[ F_\iter -F_{\iter+1}
\,\vert\, \Fsc_{\iter-1} ].
\end{align*}
Let $b = 1+ (\delta - 1)_+ /(2+\sigma_{\Gammas^{-1}}-\delta)$. Then the above inequality can be rewritten as
\begin{align}
\nonumber \EE[ (1+\sigma_{\Gammas^{-1}}) r_{\iter+1}^2 + b(F_{\iter+1} - \Fs_*) \,\vert\, \Fsc_{\iter-1}]
\nonumber&\leq (1+\sigma_{\Gammas^{-1}}) r_\iter^2 + b(F_\iter - \Fs_*) \\
\label{eq:20200926a}&\quad- \pp_{\min} \big( (\mu_{\Gammas^{-1}}+\sigma_{\Gammas^{-1}}) r_\iter^2 + F_\iter - \Fs_* \big).
%\,\vert\, \Fsc_{\iter-1} ].
\end{align}
Now, we derive from \eqref{eq:20181112b}-\eqref{eq:20181112b1} that
\begin{equation*}
F_\iter - \Fs_* \geq \frac{\mu_{\Gammas^{-1}} + \sigma_{\Gammas^{-1}}}{2} \sum_{i=1}^m \frac{1}{\gamma_i} \norm{x_i^\iter - x_i}^2
\geq \frac{\mu_{\Gammas^{-1}} + \sigma_{\Gammas^{-1}}}{2} \pp_{\min}\norm{\bx^\iter - \bxx}^2_{\bWW} 
= (\mu_{\Gammas^{-1}} + \sigma_{\Gammas^{-1}}) r_\iter^2.
\end{equation*}
Therefore, it follows from Lemma~\ref{lem:20200928} (with $c = \mu_{\Gammas^{-1}}+\sigma_{\Gammas^{-1}}$ and
$a = 1+\sigma_{\Gammas^{-1}}$) that 
\begin{equation}
 \label{eq:20200926b}
(\mu_{\Gammas^{-1}}+\sigma_{\Gammas^{-1}}) r_\iter^2 + F_\iter - \Fs_*
\geq \bar{\lambda} \big( (1+\sigma_{\Gammas^{-1}}) r_\iter^2 + b(F_\iter - \Fs_*)\big),
\end{equation}
where
\begin{equation}
\bar{\lambda} = 
\begin{cases}
\dfrac{1}{b} &\text{if } b\geq \dfrac{1+\sigma_{\Gammas^{-1}}}{\mu_{\Gammas^{-1}}+\sigma_{\Gammas^{-1}}}\\[1.5ex]
\dfrac{2(\mu_{\Gammas^{-1}}+\sigma_{\Gammas^{-1}})}{1+\sigma_{\Gammas^{-1}} + b(\mu_{\Gammas^{-1}}+\sigma_{\Gammas^{-1}})} &\text{if } b\leq \dfrac{1+\sigma_{\Gammas^{-1}}}{\mu_{\Gammas^{-1}}+\sigma_{\Gammas^{-1}}}.
\end{cases}
\end{equation}
Then, by \eqref{eq:20200926a} and \eqref{eq:20200926b}, we have that
\begin{align*}
\EE[ (1+\sigma_{\Gammas^{-1}}) r_{\iter+1}^2 + b(F_{\iter+1} - \Fs_*) \,\vert\, \Fsc_{\iter-1}]
\leq (1 - \pp_{\min} \bar{\lambda})\big((1+\sigma_{\Gammas^{-1}}) r_\iter^2 + b(F_\iter - \Fs_*) \big)
\end{align*}
and hence, taking the expectation, and applying the resulting inequality recursively, we have,
\begin{equation}
b (\EE[F_{\iter}] - \Fs_*) \leq \EE[ (1+\sigma_{\Gammas^{-1}}) r_{\iter}^2 + b(F_{\iter} - \Fs_*)]
\leq (1 - \pp_{\min} \bar{\lambda})^\iter 
\big((1+\sigma_{\Gammas^{-1}}) r_0^2 + b(F_{0} - \Fs_*) \big).
\end{equation}
To conclude it is sufficient to note that, since
\begin{equation*}
b = \max\bigg\{ 1, \frac{1+\sigma_{\Gammas^{-1}}}{2 - \delta +\sigma_{\Gammas^{-1}}} \bigg\}
\quad\text{and (in virtue of \eqref{eq:20190126e})}\quad \mu_{\Gammas^{-1}} \leq \delta,
\end{equation*}
we have
\begin{equation*}
\bar{\lambda} = 
\begin{cases}
\dfrac{2 - \delta +\sigma_{\Gammas^{-1}}}{1+\sigma_{\Gammas^{-1}}} &\text{if } \delta>1\ \text{and}\ \mu_{\Gammas^{-1}} \geq 2 - \delta\\[2ex]
\dfrac{2(\mu_{\Gammas^{-1}}+\sigma_{\Gammas^{-1}})}{1+\sigma_{\Gammas^{-1}} + (\mu_{\Gammas^{-1}}+\sigma_{\Gammas^{-1}})(1+\sigma_{\Gammas^{-1}})/(2 - \delta +\sigma_{\Gammas^{-1}})}&\text{if } \delta>1\ \text{and}\ \mu_{\Gammas^{-1}} \leq 2 - \delta\\[2ex]
\dfrac{2(\mu_{\Gammas^{-1}}+\sigma_{\Gammas^{-1}})}{1+\sigma_{\Gammas^{-1}} + (\mu_{\Gammas^{-1}}+\sigma_{\Gammas^{-1}})} &\text{if } \delta \leq 1.
\hspace{20ex}\qed
\end{cases}
\end{equation*}

\section{Some results on duality theory}
In this section, for the reader's convenience, 
we recap 
the results obtained in \cite{Dun16}.
Let $\varphi\colon \HH \to \R$ and $\psi\colon \GG \to \left]-\infty,+\infty\right]$
be two lower semicontinuous and convex functions defined on Hilbert spaces,
and let $\As\colon \HH \to \GG$ be a bounded linear operator. 
In this section we suppose that $\varphi$ is $\mu$-strongly convex.
We consider the following optimization problems in duality (in the sense of Fenchel-Rockafellar)
\begin{equation}
\min_{\xx \in \HH} \varphi(\xx) + \psi(\As \xx):=\mathcal{P}(\xx)
\quad\text{and}\quad
\min_{\uu \in \GG} \psi^*(\uu) + \varphi^*(-\As^\top \uu):=\mathcal{D}(\uu)
\end{equation}
We define the \emph{duality gap function}
$G\colon \HH \times\GG \to \left]-\infty,+\infty\right]$, 
$G(\xx,\uu) = \mathcal{P}(\xx) + \mathcal{D}(\uu)$
and recall that 
\begin{equation*}
(\mathcal{P}(\xx) - \inf \mathcal{P}) + (\mathcal{D}(\uu) - \inf \mathcal{D}) \leq G(\xx,\uu).
\end{equation*}
So, the duality gap function bounds the primal and dual objectives.
We have the following theorem
\begin{theorem}
\label{thm:dualgap}
Suppose that $R(\As) \subset \dom \partial \psi$. Then the following holds:
\begin{enumerate}[{\rm (i)}]
\item\label{thm:dualgap_i} Suppose that $\psi^*$ is $\alpha$-strongly convex.
Let $\uu \in \dom \psi^*$ and set $\xx = \nabla \varphi^*(-\As^\top \uu)$. Then, 
\begin{equation}
\label{eq:dualgap_1}
G(\xx,\uu) \leq \bigg(1 + \frac{\norm{\As}^2}{\alpha\mu} \bigg)(\mathcal{D}(\uu) - \inf \mathcal{D}).
\end{equation}
\item\label{thm:dualgap_ii} Suppose that $\psi$ is $\theta$-Lipschitz continuous.
Let $\uu \in \dom \psi^*$ be such that $\mathcal{D}(\uu)- \inf \mathcal{D}<\norm{\As}^2 L^2/\mu$
and set $\xx = \nabla \varphi^*(-\As^\top \uu)$.
Then, we have
\begin{equation}
\label{eq:dualgap_2}
G(\xx,\uu) \leq 2 \frac{\norm{\As}\theta}{\mu^{1/2}} (\mathcal{D}(\uu) - \inf \mathcal{D})^{1/2}.
\end{equation}
Moreover, if $u$ is a random variable taking values in $\dom \psi^*$ and such that 
$\EE[\mathcal{D}(u)]- \inf \mathcal{D}<\norm{\As}^2 L^2/\mu$ and we set $x = \nabla \varphi^*(-\As^\top u)$, then $\EE[G(x,u)] \leq 2 \norm{\As}\theta/\mu^{1/2} (\EE[\mathcal{D}(u)] - \inf \mathcal{D})^{1/2}$.
\end{enumerate}
\end{theorem}

\end{document}